\documentclass[11pt]{article}

\usepackage{graphicx}
\usepackage{latexsym}
\usepackage{amssymb}
\usepackage{amsthm}
\usepackage{indentfirst}
\usepackage{amsmath}
\usepackage{color}
\usepackage{mathrsfs}
\usepackage[shortlabels]{enumitem}
\usepackage[titletoc]{appendix}
\usepackage{hyperref}
\hypersetup{colorlinks=false, citecolor=green, linkcolor=red, hypertexnames=false}
\usepackage{microtype}
\usepackage{galois}
\usepackage{titlesec}
\titleformat{\section}
  {\vskip 1em\normalfont\Large\sc\centering}{\thesection}{1em}{}
\usepackage[numbers,sort&compress]{natbib}
\usepackage{booktabs}
\usepackage[normalem]{ulem}
\usepackage{extarrows}

\newtheorem{Theorem}{Theorem}[section]
\newtheorem*{Main Theorem}{Main Theorem}
\newtheorem*{Main Proposition}{Main Proposition}
\newtheorem*{Theorem A}{Theorem A}
\newtheorem*{Theorem B}{Theorem B}
\newtheorem*{Theorem C}{Theorem C}
\newtheorem*{Theorem D}{Theorem D}

\newtheorem{Definition}[Theorem]{Definition}
\newtheorem{Proposition}[Theorem]{Proposition}
\newtheorem{Lemma}[Theorem]{Lemma}
\newtheorem{Question}{Question}

\newtheorem{Remark}{Remark}[section]
\newtheorem{Corollary}[Theorem]{Corollary}

\newtheorem{Claim}{Claim}[section]

 \def\NN{{\mathbb N}} 

 \def\RR{{\mathbb R}}

\def\al{\alpha}
\def\be{\beta}
\def\Si{\Sigma}
\def\si{\sigma}
\def\La{\Lambda}
\def\la{\lambda}
\def\De{\Delta}
\def\de{\delta}

\def\om{\omega}
\def\Ga{\Gamma}
\def\ga{\gamma}

\def\th{\theta}

\def\sv{{\rm sv}}
\def\Cl{{\rm Cl}}

  \def\cG{\mathcal{G}} \def\cM{\mathcal{M}} 
   \def\cN{\mathcal{N}} 
    \def\cU{\mathcal{U}}
   \def\cP{\mathcal{P}} 
    \def\cW{\mathcal{W}}

\def\xX{\mathscr{X}}
\def\fF{\mathscr{F}}

\def\diff{\operatorname{Diff}}

\def\dim{\operatorname{dim}}
\def\ind{\operatorname{Ind}}

\def\Sing{\operatorname{Sing}}
\def\orb{\operatorname{Orb}}
\def\supp{\operatorname{Supp}}
\def\vep{{\varepsilon}}
\def\CR{{\mathrm {CR}}}
\def\e{\mathrm{e}}
\def\wt{\widetilde}
\def\txtrm[#1]{\ \textrm{#1}\ }
\def\nibf[#1]{\vskip 1.5em\noindent{\bfseries #1}}
\def\Gibbs{\operatorname{Gibbs}}
\def\PhL{\operatorname{PhL}}

\numberwithin{equation}{section}

\title{Lyapunov stable chain recurrence classes for singular flows}
\author{Shaobo Gan, Jiagang Yang and Rusong Zheng\footnote{SG is partially supported by National Key R\&D Program of China 2022YFA1005801 and NSFC 12161141002. JY is partially supported by CNPq, FAPERJ, PRONEX, MATH-AmSud 220029, and NSFC grants 12271538, 11871487 and 12071202. RZ is partially supported by NSFC grants 12071007, 12101293.}}

\begin{document}
\maketitle

\begin{abstract}
We show that for a $C^1$ generic vector field $X$ away from homoclinic tangencies, a nontrivial Lyapunov stable chain recurrence class is a homoclinic class. The proof uses an argument with $C^2$ vector fields approaching $X$ in $C^1$ topology,  with their Gibbs $F$-states  converging to a Gibbs $F$-state of $X$.
\end{abstract}

\setcounter{tocdepth}{2}
\tableofcontents

\section{Introduction}

The differentiable flows on a closed manifold with the simplest dynamics are generated by gradient-like vector fields and belong to the so-called  {\em Morse-Smale} systems. The non-wandering set of a Morse-Smale vector field is composed of a finite set of {\em critical elements} (singularities or periodic orbits), and the unstable manifold of one critical element is transverse to the stable manifold of any other critical element (including itself). Moving beyond the Morse-Smale systems, one then considers the set of {\em (uniformly) hyperbolic} systems, whose chain recurrence sets (as defined in \cite{Co78}, see Section \ref{sect:technical-results}) are composed of finitely many pieces with each piece being a hyperbolic invariant set.
For a flow $\phi_t$ generated by a $C^1$ vector field $X$ on a closed manifold $M$, a $\phi_t$-invariant set $\La$ is called {\em hyperbolic} if there are constants $C\geq 1$, $\lambda>0$, and a continuous splitting of the tangent bundle $T_{\La}M=E^s\oplus \langle X\rangle\oplus E^u$, invariant under the tangent flow $\Phi_t\xlongequal{def} D\phi_t:TM\to TM$, such that
\begin{itemize}
  \item $\Phi_t|_{E^s}$ is {\em contracting}~: $\|\Phi_t|_{E^s(x)}\|\leq C\e^{-\la t}$,\ for every $x\in \Lambda$ and $t\geq 0$;
  \item $\Phi_t|_{E^u}$ is {\em expanding}~: $\|\Phi_{-t}|_{E^u(x)}\|\leq C\e^{-\la t}$,\ for every $x\in \Lambda$ and $t\geq 0$.
\end{itemize}
When $\La$ is hyperbolic, for each $x\in\La$ there exist a strong stable manifold $W^s(x)$ and a strong unstable manifold $W^u(x)$ that are tangent to $E^s(x)$ and $E^u(x)$, respectively. For a hyperbolic periodic orbit $\ga$, its stable manifold and unstable manifold are given by $W^s(\ga)=\bigcup_{x\in\ga}W^s(x)$ and $W^u(\ga)=\bigcup_{x\in\ga}W^u(x)$, respectively.

It is shown in~\cite{Pe62} that Morse-Smale vector fields form an open and dense subset of $C^1$ vector fields on a given closed surface. Then Smale asked whether hyperbolicity is generic even on higher-dimensional manifolds. The answer turns out to be negative, as Abraham and Smale himself \cite{AS} gave a non-empty $C^1$-open set of non-hyperbolic systems (for diffeomorphisms) in dimension 4. Around the same time, Newhouse \cite{Ne70,Ne74} exhibited a non-empty $C^2$-open set of non-hyperbolic diffeomorphisms on surfaces. These two types of examples provide all the known obstructions to hyperbolicity for diffeomorphisms (similar definitions can be made for flows):
\begin{itemize}
	\item {\em heteroclinic cycle} as in the Abraham-Smale example: two hyperbolic periodic orbits with different indices ({\itshape i.e.} stable dimensions) are linked by the dynamics, such that the stable manifold of each of these orbits intersects with the unstable manifold of the other;
	\item {\em homoclinic tangency} as in Newhouse's example: there is a hyperbolic periodic orbit whose stable manifold and unstable manifold have a non-transverse intersection.
\end{itemize}

It has been conjectured by Palis \cite{Pa90,Pa00,Pa05,Pa08} that these two mechanisms constitute the main obstructions to hyperbolicity for diffeomorphisms. One can also refer to \cite{Bon11} for more related results and conjectures. This justifies the interest in dynamics away from homoclinic tangencies, and more restrictively, dynamics away from {\em homoclinic bifurcations} (homoclinic tangency and heteroclinic cycle).
Precisely, a differentiable dynamical system is called {\em away from homoclinic tangencies} if there is a $C^1$ neighborhood $\cU$ of the system such that none of the systems in $\cU$ has a homoclinic tangency.

The dynamics for $C^1$ diffeomorphisms away from homoclinic tangencies has been extensively studied, {\itshape e.g.} \cite{PuS, Wen02, BDP03, Yan07, Yan09, Go, Cro10, CSY15}.
Many of these results can be translated in a similar way to non-singular flows ({\itshape i.e.} flows without singularities). For example, Pujals and Sambarino's classical result \cite{PuS} confirming the Palis Conjecture on surface diffeomorphisms has been generalized to non-singular flows by Arroyo and Rodriguez Hertz \cite{ArRH}. Also, the Weak Palis Conjecture confirmed by \cite{Cro10} has been generalized to non-singular flows \cite{XZ}.

However, for general $C^1$ flows, the study of dynamics away from homoclinic tangencies is much more difficult due to the existence of singularities.
The notion of {\em singular-hyperbolicity} was introduced by Morales, Pacifico and Pujals \cite{MPP} to characterize Lorenz-like dynamics and also to generalize the concept of hyperbolicity to singular sets ({\itshape i.e.} compact invariant sets with singularities).
To be precise, let $\La$ be a compact invariant set of $X$. A $\Phi_t$-invariant splitting $T_{\La}M=E\oplus F$ is called {\em dominated} if there exists $T>0$ such that for any $x\in \La$ and $t>T$, it holds
\[\|\Phi_t|_{E(x)}\|\|\Phi_{-t}|_{F(\phi_t(x))}\|\leq \frac{1}{2}.\]
The splitting is called {\em partially hyperbolic} if it is dominated and either $\Phi_t|_E$ is contracting or $\Phi_t|_F$ is expanding. The set $\Lambda$ is called {\em singular hyperbolic} if it admits a partially hyperbolic splitting $T_{\La}M=E\oplus F$, together with two constants $C\geq 1$, $\la>0$, such that
\begin{itemize}
  \item either $\Phi_t|_{E}$ is contracting and $\Phi_t|_F$ is {\em sectionally expanding}~: for any $x\in \La$ and any $2$-dimensional subspace $L\subset F(x)$, for any $t\geq 0$, it holds
      \[|\det(\Phi_{-t}|_L)|\leq C\e^{-\la t};\]
  \item or $\Phi_t|_F$ is expanding and $\Phi_t|_E$ is {\em sectionally contracting}~: for any $x\in \La$ and any $2$-dimensional subspace $L\subset E(x)$, for any $t\geq 0$, it holds
      \[|\det(\Phi_t|_L)|\leq C\e^{-\la t}.\]
\end{itemize}
It is proven in \cite{CrY15,CrY18} that for a generic 3-dimensional vector field away from homoclinic tangencies, the chain recurrence set is the union of finitely many singular-hyperbolic sets. Also in dimension 3, \cite{GY} proved the Weak Palis Conjecture for singular flows ({\itshape i.e.} flows with singularities). In general dimensions, some progress on entropy theory has been obtained recently in \cite{SYY}, \cite{PYY} and~\cite{PYY2}.

In this paper, we study $C^1$ generic dynamics of singular flows away from homoclinic tangencies and in general dimensions.
A good strategy is to begin with elementary pieces of the dynamics. The term ``elementary piece'' has been used for the basic sets of hyperbolic systems (given by Smale's spectral decomposition theorem \cite{Sma}), and it needs to be clarified for a general dynamical system.
Newhouse \cite{Ne72} introduced the notion of {\em homoclinic class} to generalize Smale's basic sets: the homoclinic class of a periodic orbit is the closure of transverse intersections of its stable and unstable manifolds. Homoclinic classes have many good properties. For example, they are transitive, and each has a dense subset of periodic orbits. For $C^1$ generic systems, one usually considers {\em chain recurrence classes} whose precise definition will be given in Section \ref{sect:technical-results}. The union of all chain recurrence classes is called the chain recurrence set, which covers all the interesting dynamics in the sense that chain-recurrence is the weakest possible recurrence property. It is known \cite{BoC} that for $C^1$ generic systems, a chain recurrence class containing a periodic orbit coincides with the homoclinic class of said periodic orbit.
A chain recurrence class is called {\em nontrivial} if it is not reduced to a critical element.

A chain recurrence class $\Lambda$ of a flow $\phi_t$ is called {\em Lyapunov stable} if for any neighborhood $U$ of $\La$, there is a neighborhood $V$ of $\La$ such that $\phi_t(V)\subset U$ for any $t\geq 0$. It is known \cite{BoC} that $C^1$ generically, every Lyapunov stable chain recurrence class is a {\em quasi-attractor} (as defined in \cite{Hur82}), and there always exists Lyapunov stable chain recurrence classes.
Lyapunov stability generalizes the attracting property: $\Lambda$ is said to be {\em attracting} if there exists a neighborhood $V$ such that $\phi_1(V)\subset V$ and $\bigcap_{t\geq 0}\phi_t(V)=\Lambda$. If $\Lambda$ is attracting and transitive, then it is called an {\em attractor}. Attractors are important because it essentially captures all the complexity of the dynamics in a neighborhood. By \cite{BoC}, for a $C^1$ generic system, if a nontrivial chain recurrence class is an attractor, then it contains periodic orbits and is a homoclinic class. Thus, $C^1$ generically, the chain recurrence classes containing periodic orbits constitute all the candidates for nontrivial attractors.

A nontrivial chain recurrence class containing no periodic orbits is called an {\em aperiodic class} \cite{BD}. It is an open problem to prove or disprove the existence of aperiodic classes for $C^1$ generic vector fields away from homoclinic tangencies (see \cite[Conjecture 20]{Bon11}).
We shall prove the following result, which excludes the possibility of Lyapunov stable aperiodic classes in such a setting.

\begin{Theorem A}[Main Theorem]\label{thm:main}
	Let $X$ be a $C^1$ generic vector field away from homoclinic tangencies. Then any nontrivial Lyapunov stable chain recurrence class of $X$ is a homoclinic class, {\itshape i.e.,} it contains a periodic orbit.
\end{Theorem A}

This result contributes to the Weak Palis Conjecture \cite{Pa90,Pa00,Pa05,Pa08}, which asserts that the systems that are Morse-Smale or exhibiting a transverse homoclinic intersection form a dense and open subset of all $C^1$ systems. Since a non-transverse homoclinic intersection ({\it i.e.} a tangency) can be turned into a transverse one by an arbitrarily small $C^1$ perturbation, a general idea of tackling this conjecture is to start with $C^1$ generic systems away from homoclinic tangencies and show that a nontrivial chain recurrence class (if exists) should be a homoclinic class, see {\it e.g.} \cite{PuS,BGW07,Cro10,GY}. We have mentioned that for $C^1$ singular flows the conjecture has been proved only in dimension 3 \cite{GY}, where every nontrivial chain recurrence class of a $C^1$ generic $X$ is Lyapunov stable either for $X$ or for $-X$. Our theorem generalizes \cite{GY} by proving the Weak Palis Conjecture in the important subcase --- the Lyapunov stable case, {\em for singular flows in arbitrary dimensions}.
For $C^1$ diffeomorphisms, the same result was first obtained in \cite{Yan07}, whose proof motivates the current paper. With a bit more effort (see Proposition \ref{prop:non-singular-set}), one can generalize \cite{Yan07} to non-singular flows. 

In \cite{PYY}, it is proven that for $C^1$ generic vector fields, a Lyapunov stable chain recurrence class which is singular-hyperbolic is a transitive attractor; in particular, it is a homoclinic class. Recently, \cite{CrY20} showed that this result holds robustly. These results heavily rely on the singular-hyperbolicity. In contrast, Theorem A is stated for vector fields away from homoclinic tangencies but without any assumption on the hyperbolicity of the class.

We remark that removing the assumption on hyperbolicity is a big step forward in the study of higher dimensional (dim $\geq 4$) singular flows. Other studies on higher dimensional singular flows confine themselves either to the singular hyperbolic flows ({\it e.g.} \cite{PYY,PYY3,CrY20}), or to the {\em star} flows (or equivalently in the $C^1$ generic setting, the {\em multi-singular hyperbolic} flows, {\it e.g.} \cite{SGW,BdL,PYY2}).

For the proof of Theorem A, we will eventually reduce it to the case when the chain recurrence class admits a partially hyperbolic splitting $E^{ss}\oplus F$ such that $\Phi_t|_{E^{ss}}$ is contracting while $\Phi_t|_F$ is {\em not} sectionally expanding. See our main proposition --- Proposition \ref{prop:main-proposition}. The proof of the main proposition uses an argument with $C^2$ vector fields approaching $X$ in $C^1$ topology, with their Gibbs $F$-states (see Section \ref{sect:physical-like-measure} for precise definition) converging to a Gibbs $F$-state of $X$. This argument is new and it indeed plays a key role in our proof (see Section \ref{sect:proof-of-main-prop}). More discussions of this argument can be found in Section \ref{sect:technical-results}.
Note that our argument is a measure-theoretical one that differs completely from the classical $C^2$ argument used by Pujals and Sambarino \cite{PuS}.

In view of the main result in \cite{MP02}, which shows that the union of the basins of all Lyapunov stable chain recurrence classes is a residual subset of $M$,
Theorem A (with a standard generic argument using the connecting lemmas \cite{Ha, WeX, BoC}) has the following corollary:
\begin{Corollary}\label{cor:dense-stable-manifold}
	Let $X$ be a $C^1$ generic vector field away from homoclinic tangencies, then the union of stable manifolds of all periodic orbits of $X$ is dense in $M$.
\end{Corollary}

This corollary gives a partial answer to a conjecture by Bonatti stating that $C^1$ generically the union of stable manifolds of all periodic orbits is dense in the manifold. Our results may also contribute to another conjecture of Bonatti concerning residual attractors (\cite[Conjecture 19]{Bon11}). A {\em residual attractor} (as defined in \cite{BLY}) is a quasi-attractor such that its basin contains a residual subset of a neighborhood. One may ask if quasi-attractors are residual attractors for $C^1$ generic flows away from homoclinic tangencies. This was proven to hold for $C^1$ diffeomorphisms by \cite{BGLY}.

\vspace{.5cm}
\noindent{\bf Acknowledgements.} The authors are grateful to Fan Yang for carefully reading the paper and for many useful comments and corrections. The authors would also like to thank Ming Li for patiently listening to the proof and for the useful discussions.

\section{More results and discussions}
\label{sect:technical-results}

In this section, we present some properties of critical elements in a neighborhood of the chain recurrence class as in the main theorem. These results will give us a better understanding of the dynamics close to a Lyapunov stable chain recurrence class for a generic vector field away from homoclinic tangencies.

In the remaining part of this paper, $M$ denotes a boundaryless compact Riemannian manifold, and $\xX^r(M)$ is the set of $C^r$ vector fields on $M$ endowed with the $C^r$ topology. We mainly deal with the case $r=1$, but in Section~\ref{s.7} and onward, we will need to consider $C^2$ vector fields close to $X$ in $C^1$ topology.
Let $X\in\xX^1(M)$ and $\phi_t$ be the flow generated by $X$. Following Conley's theory \cite{Co78}, for any $\vep>0$, a finite sequence of points $x_0,x_1,\ldots,x_n$ is called an {\em $\vep$-chain} (or {\em $\vep$-pseudo-orbit}) from $x_0$ to $x_n$ if there exist $t_i\geq 1$ such that $d(\phi_{t_i}(x_i),x_{i+1})<\vep$ for all $0\leq i\leq n-1$, where $d(\cdot,\cdot)$ is the distance function over $M$ induced by the Riemannian metric. For any $x,y\in M$, we say that $y$ is {\em chain attainable} from $x$ if for any $\vep>0$, there exists an $\vep$-chain from $x$ to $y$. If $x$ is chain attainable from itself, then it is called a {\em chain recurrent point}. The set of chain recurrent points is called the {\em chain recurrence set} of $X$, which is denoted by $\CR(X)$. An equivalence relation ``$\sim$'' can be defined on $\CR(X)$ as the following: for any $x,y\in\CR(X)$, $x\sim y$ if and only if $x$ is chain attainable from $y$ and also $y$ is chain attainable from $x$. An equivalence class of $\CR(X)$ by ``$\sim$'' is called a {\em chain recurrence class}. For any $x\in\CR(X)$, the chain recurrence class containing $x$ is denoted by $C(x,X)$, or simply $C(x)$ when there is no risk of confusion.

Following the main theorem, it is natural to ask questions about the stable indices of critical elements in the class. For a hyperbolic singularity $\si$, its {\em (stable) index} is defined as the dimension of the stable manifold of $\si$ and will be denoted as $\ind(\si)$. For a hyperbolic periodic orbit $\ga$, the {\em index} of $\ga$, denoted by $\ind(\ga)$, is defined as $\ind(\ga)=\dim W^s(\ga)-1$.

\begin{Question}\label{quest:index}
	Let $X$ be a $C^1$ generic vector field away from homoclinic tangencies. Suppose $C(\si)$ is a nontrivial Lyapunov stable chain recurrence class associated to a singularity $\si$ of $X$. What is the relation between the indices of periodic orbits in $C(\si)$ and the index of $\si$? Is $C(\si)$ index-complete, i.e., does the set of indices of periodic orbits in $C(\si)$ form an interval in $\NN$?
\end{Question}

The answers to these questions are all trivial if $\dim M=3$, since $C(\si)$ can contain only periodic orbits of saddle type, which has index 1 in this case. In the general case, positive answers are provided in Theorem B and Corollary \ref{cor:index-completeness}. Note that for $C^1$ generic diffeomorphisms, index-completeness of homoclinic classes was proven in \cite{ABCDW}.

\begin{Theorem B}
	Let $X$ be a $C^1$ generic vector field away from homoclinic tangencies. Suppose $C(\si)$ is a nontrivial Lyapunov stable chain recurrence class containing a singularity $\si$ of $X$, then the following properties hold:
	\begin{itemize}
		\item All singularities in $C(\si)$ have the same index;
		\item For any periodic orbit $\ga$ contained in $C(\si)$, one has $\ind(\ga)\leq \ind(\si)-1$;
		\item There is a neighborhood $U$ of $C(\si)$ and a neighborhood $\cU$ of $X$ such that for any $Y\in\cU$ and any periodic orbit $\ga$ of $Y$ contained in $U$, one has $\ind(\ga)\leq\ind(\si)$.
	\end{itemize}
\end{Theorem B}

\begin{Corollary}\label{cor:index-completeness}
Under the same assumptions as in Theorem B, the chain recurrence class $C(\si)$ is index-complete: the set $\left\{\,\ind(\gamma):\gamma\subset C(\si) \mbox{ is a periodic orbit.}\right\}$ is of the form $[\alpha,\beta]\cap\NN$ with $\beta = \ind(\sigma)-1$.	
\end{Corollary}

More properties of singularities contained in a nontrivial chain recurrence class (not necessarily Lyapunov stable) are obtained along the proof of Theorem A and B.

To be precise, let us introduce more definitions.
Let $X\in\xX^1(M)$ and $\si$ be a hyperbolic singularity of $X$. Let $\{\la_i\}_{i=1}^{\dim M}$ be the Lyapunov exponents of the tangent flow at $\sigma$, i.e.,  $\Phi_t(\si)=\Phi_t|_{T_{\si}M}$, $\la_1\leq\la_2\leq\cdots\leq\la_s<0<\la_{s+1}\leq\cdots\leq\la_{\dim M}$. The {\em saddle value} of $\si$ is defined as
\[\sv(\si)=\la_s+\la_{s+1}.\]

Suppose $\la_{s-1}<\la_s$, then there is a partially hyperbolic splitting $T_{\si}M=E^{ss}\oplus E^{cu}$, such that $\dim E^{ss}=s-1$ and $E^{ss}$ is the invariant subspace corresponding to $\la_1,\cdots,\la_{s-1}$. Similarly, suppose $\la_{s+1}<\la_{s+2}$, then there is a partially hyperbolic splitting $T_{\si}M=E^{cs}\oplus E^{uu}$, such that $E^{uu}$ is the invariant subspace corresponding to $\la_{s+2},\cdots,\la_{\dim M}$.

\begin{Definition}[Lorenz-like singularity]\label{def:lorenz-like}
	Let $X\in\xX^1(M)$ and $\si$ be a hyperbolic singularity of $X$. Assume that $C(\si)$ is nontrivial and the Lyapunov exponents of $\Phi_t(\si)$ are $\la_1\leq\la_2\leq\cdots\leq\la_s<0<\la_{s+1}\leq\cdots\leq\la_{\dim M}$.
	The singularity $\si$ is said to be {\em Lorenz-like}, if the following conditions are satisfied:
	\begin{itemize}
		\item $\sv(\si)>0$, $s>1$ and $\la_{s-1}<\la_s$;
		\item Let $E^{ss}$ be the invariant subspace corresponding to $\la_1,\cdots,\la_{s-1}$, and $W^{ss}(\si)$  the strong stable manifold corresponding to $E^{ss}$, then $W^{ss}(\si)\cap C(\si)=\{\si\}$.
	\end{itemize}
	The singularity $\si$ is said to be {\em reversed Lorenz-like} if it is Lorenz-like for $-X$.
\end{Definition}
Note that our definition of Lorenz-like singularity is more restrictive compared to that defined in \cite{MPP}, where the condition about the strong manifolds is not assumed. By definition, $\ind(\si)>1$ if $\si$ is Lorenz-like, and $\ind(\si)<\dim M-1$ if it is reversed Lorenz-like.

\begin{Theorem C}
	Let $X$ be a $C^1$ generic vector field away from homoclinic tangencies. Then any singularity $\si$ of $X$ contained in a nontrivial chain recurrence class is either Lorenz-like or reversed Lorenz-like. Moreover, for a Lyapunov stable chain recurrence class, singularities contained in it are all Lorenz-like and have the same index.
\end{Theorem C}

Finally, let us mention some key ingredients in proving our main theorem.
\begin{itemize}
	\item Li-Gan-Wen \cite{LGW05} ``extended'' the linear Poincar\'e flow to singularities and discovered an important mechanism governing the way how periodic orbits with a given dominated splitting can accumulate on a singularity (see Section \ref{sect:matching}). In particular, the mechanism guarantees a ``restricted area'' on the tangent space of the singularity for such periodic orbits. This mechanism can be used to obtain a contradiction if, under some erroneous assumptions, one can show the existence of a proper sequence of periodic orbits accumulating on some singularity such that it violates the mechanism by ``stepping into'' the corresponding ``restricted area''. This is how the main theorem is proven.
	\item An argument concerning $C^2$ vector fields is used to ensure that we can carry out the proof-by-contradiction process mentioned above. The $C^2$ regularity is required as the argument relies on Pesin's entropy formula for invariant measures and the result of Ledrappier-Young \cite{LeY85}. More precisely, Ledrappier and Young proved a $u$-saturatedness property for such invariant measures (see Theorem \ref{thm:LeY}), which roughly means that the unstable manifold of almost every point is contained in the support of the measure. This property ensures the existence of a proper sequence of periodic orbits that converges to the entire chain recurrence class, under the assumption that the class is aperiodic. We then show that (Lemma \ref{lem:gen_1}) these periodic orbits do ``step into'' the corresponding ``restricted area'' guaranteed by Li-Gan-Wen's mechanism, and hence obtain the contradiction.
	\item Physical-like measures as defined in \cite{CE, CCE} are considered for $C^2$ vector fields that are $C^1$ close to $X$. Physical-like measures are good candidates for our $C^2$ argument as above, since the entropy of such measures has a certain lower bound (Theorem \ref{thm:physical-like}) that may lead to Pesin's entropy formula under proper conditions.
	\item Liao's theory on singular flows is used extensively throughout the paper, including rescaling the linear Poincar\'e flow, relatively uniform size of invariant manifolds, and an intersection result of the invariant manifolds between singularities and periodic orbits (Theorem \ref{thm:Liao}). These tools play an essential role in the local analysis near singularities.
	\item Last but not least, the central model by Crovisier \cite{Cro10} and an argument concerning homoclinic orbits of singularities by Gan-Yang \cite{GY} are used to give the first reduction of our main theorem, see Section \ref{sect:reduction}.
\end{itemize}

The remaining part of this paper is organized as follows.
In Section \ref{sect:preliminaries} we introduce some preliminaries of singular flows, such as extended linear Poincar\'e flow and its rescaling, plaque family and invariant manifolds, and in particular, matching of dominated splittings near a singularity (Section \ref{sect:matching}) and a theorem of Liao (Section \ref{sect:Liao-thm}).
Then in Section \ref{sect:away_HT} we consider vector fields away from homoclinic tangencies in general, proving Theorem C and part of Theorem B.
In Section \ref{sect:non-singular-set} we consider non-singular compact chain transitive subsets of Lyapunov stable chain recurrence classes, proving that the existence of such subsets implies the existence of periodic orbits in the chain recurrence class, for generic vector field away from homoclinic tangencies. Also in this section, we complete the proof of Theorem B and give a proof of Corollary~\ref{cor:index-completeness} assuming Theorem A.
Section \ref{sect:reduction} is dedicated to the reduction of Theorem A to the main proposition (Proposition \ref{prop:main-proposition}), assuming which the proof of Theorem A is given.
Proof of the main proposition can be found in the last section, {\it i.e.} Section \ref{sect:proof-of-main-prop}, after some preparation work done in Section \ref{sect:preliminaries-measure} regarding invariant measures.

\section{Preliminaries on \texorpdfstring{$C^1$}{C\^1} flows}
\label{sect:preliminaries}

In this section we list the preliminaries on $C^1$ flows that will be used, including various associated flows, the size of invariant manifolds at hyperbolic points, Li-Gan-Wen's matching mechanism \cite{LGW05}, a theorem of Liao \cite{Lia89}, and some results on $C^1$ generic vector fields. Properties regarding invariant measures of $X$ (ergodic closing lemma, Pesin's entropy formula, physical-like measures, etc) are postponed to Section~\ref{sect:preliminaries-measure} since they will not be used in the first half of this paper.

\subsection{Flows associated to a vector field}
\label{sect:associated-flows}

Let $X\in\xX^1(M)$ and $\Sing(X)=\{x\in M: X(x)=0\}$ be the set of singularities. Let $(\phi^X_t)_{t\in\RR}$ be the flow on the manifold generated by $X$.
Other associated flows are listed below.

\nibf[The {\em tangent flow $(\Phi^X_t)_{t\in\RR}$}.]
By taking differentiation of the flow $(\phi^X_t)_{t\in\RR}$, one obtains the tangent flow $(D\phi^X_t)_{t\in\RR}$, defined on the tangent bundle $TM$. We will use the convention $\Phi^X_t=D\phi^X_t$.

\nibf[The {\em linear Poincar\'e flow $(\psi^X_t)_{t\in\RR}$}.]
The hyperbolicity of a flow is described by the dynamics on the normal bundle (except for singularities).
For $x\in M\setminus\Sing(X)$, let $\cN^X_x$ be the orthogonal complement of the flow direction $\langle X(x)\rangle$, i.e., $\cN^X_x=\{v\in T_xM:~v\perp X(x)\}$.
Then for any subset of $\La\subset M$, we denote
\[\cN^X_{\La}=\bigcup_{x\in\La\setminus\Sing(X)}\cN^X_x.\]
When $\La=M$, we write $\cN^X=\cN^X_M$ for convenience.
The linear Poincar\'e flow $\psi^X_t:\cN^X\to\cN^X$ is defined in the following way:
given $v\in \cN^X_x$, $x\in M\setminus\Sing(X)$, define $\psi^X_t(v)$ to be the orthogonal projection of $\Phi^X_t(v)$ to
$\cN^X_{\phi^X_t(x)}$ along the flow direction, i.e.,
\[
\psi^X_t(v)=\Phi^X_t(v)-\frac{\langle \Phi^X_t(v), X(\phi^X_t(x))\rangle}{\|X(\phi^X_t(x))\|^2}X(\phi^X_t(x)),
\]
where $\langle\cdot, \cdot\rangle$ is the inner product on $T_xM$ given by the Riemannian metric.

The use of linear Poincar\'e flow to understand the hyperbolic structure of a flow is now standard practice; see \cite{MPP} for example. However, due to the presence of singularities, uniform hyperbolicity for the linear Poincar\'e flow does not lead to uniform size for the invariant manifolds. This problem can be solved by rescaling the linear Poincar\'e flow.

\nibf[The {\em rescaled linear Poincar\'e flow $(\psi^{X,*}_t)_{t\in\RR}$}.]
Rescaling the linear Poincar\'e flow is necessary for estimating the size of the invariant manifolds near singularities. The rescaled linear Poincar\'e flow $\psi_t^{X,*}:\cN\to\cN$ is defined as follows: for any $x\in M\setminus\Sing(X)$ and $v\in \cN^X_x$,
\[
\psi_t^{X,*}(v)=\frac{\|X(x)\|}{\|X(\phi^X_t(x))\|}\psi^X_t(v)=\frac{\psi^X_t(v)}{\|\Phi^X_t|_{\langle X(x)\rangle}\|}.
\]

The rescaling technique was first introduced by S. Liao \cite{Lia74, Lia89}, and then further explored by Gan-Yang \cite{GY}. See also \cite{SGW}, \cite{PYY}, \cite{CrY18}, \cite{CrY23}, \cite{WeW} for more applications.

\begin{Remark}
When there is no risk of any confusion, we will simply write $\phi_t$ instead of $\phi^X_t$, $\Phi_t$ instead of $\Phi^X_t$ and so on. This simplification will be applied also to other objects and flows related to a given vector field throughout this paper.
\end{Remark}

\subsection{Extending (rescaled) linear Poincar\'e flow to singularities}

As the linear Poincar\'e flow is not defined on singularities, it is inconvenient, if not impossible, to apply the linear Poincar\'e flow to singular flows.
The (rescaled) linear Poincar\'e flow can be extended ``to singularities'', as was done by Li-Gan-Wen \cite{LGW05} when generalizing the result of \cite{MPP}; this results in linear flows on the normal bundle with compact domains. See also \cite{GY}, \cite{CrY18}, \cite{CrY23} and \cite{BdL} for more discussions.

Following \cite{LGW05}, we denote by $G^1$ the Grassmannian of $M$,
\[G^1=G^1(M)=\{L:\ L \textrm{ is a 1-dimensional subspace of } T_xM , x\in M\}. \]
The tangent flow $(\Phi_t)_{t\in\RR}$ then induces a flow $(\hat{\Phi}_t)_{t\in\RR}$ on $G^1$ defined by $\hat{\Phi}_t(\langle v\rangle)=\langle \Phi_t(v)\rangle$,
where $v\in TM$ is a nonzero vector and $\langle v\rangle$ is the linear subspace spanned by $v$.

Note that $G^1$ is a fiber bundle with base $M$.
Let $\beta:G^1\to M$ and $\xi:TM\to M$ be the corresponding bundle projections. A vector bundle over $G^1$ can be obtained as follows:
\[\beta^*(TM)=\{(L,v)\in G^1\times TM:~\beta (L)=\xi (v)\},\]
with the bundle projection $\iota:\beta^*(TM)\to G^1$, $(L,v)\mapsto L$.

\nibf[The {\em extended tangent flow $(\tilde{\Phi}_t)_{t\in\RR}$}.]
Given a vector field $X\in\xX^1(M)$, we can lift its tangent flow $(\Phi_t)_{t\in\RR}$ to a flow $(\tilde{\Phi}_t)_{t\in\RR}$ on $\beta^*(TM)$, such that
$$\tilde{\Phi}_t(L,v)=(\hat{\Phi}_t(L),\Phi_t(v)),\ \forall (L,v)\in\beta^*(TM).$$

Since there is a trivial identification between $T_xM$ and $\{L\}\times T_xM$ with $\be(L)=x$,
the inner product $\langle\cdot,\cdot\rangle$ on $TM$ induces also an inner product on $\beta^*(TM)$:
$$\langle(L,u),(L,v)\rangle=\langle u,v\rangle,
~\mbox{for any}~(L,u),(L,v)\in\beta^*(TM).$$
With this metric, for any $L\in G^1$ and any subspace $E\subset T_{\beta(L)}M$, one has
\[\|\tilde{\Phi}_t|_{\{L\}\times E}\|=\|\Phi_t|_E\|, \quad\textrm{for any}\ t\in\RR.\]

Note that the one-dimensional sub-bundle $\cP=\{(L,v)\in\beta^*(TM):~v\in L\}$ is invariant under any extended tangent flow. Denote by $\wt{\cN}$ the orthogonal complementary to $\cP$, {\itshape i.e.}
\[\widetilde{\cN}=\cP^\perp=\{(L,v)\in\beta^*(TM):~v\perp L\}.\]
In general, for any nonempty subset $\Delta\subset G^1$, one can define
$\widetilde{\cN}_{\Delta}=\{(L,v)\in\iota^{-1}(\De):~v\perp L\}$.

\nibf[The {\em extended linear Poincar\'e flow $(\tilde{\psi}_t)_{t\in\RR}$}.] Given a vector field $X\in\xX^1(M)$, we can lift its linear Poincar\'e flow to a flow $(\tilde{\psi}_t)_{t\in\RR}$ on $\wt{\cN}$:
\[\tilde{\psi}_t(L,v)=\tilde{\Phi}_t(L,v)-\frac{\langle \Phi_t(v),\Phi_t(u)\rangle}{\|\Phi_t(u)\|^2}\tilde{\Phi}_t(L,u),\]
where $u\in L\setminus\{0\}$.

\nibf[The {\em extended rescaled linear Poincar\'e flow $(\tilde{\psi}^*_t)_{t\in\RR}$}.] Given a vector field $X\in\xX^1(M)$, the extended rescaled linear Poincar\'e flow $\tilde{\psi}^*_t:\wt{\cN}\to\wt{\cN}$ is defined as follows:
\[\tilde{\psi}^*_t(L,v)=\frac{\tilde{\psi}_t(L,v)}{\|\tilde{\Phi}_t|_{\mathcal{P}_L}\|}=\frac{\tilde{\psi}_t(L,v)}{\|\Phi_t|_L\|},\quad \forall (L,v)\in\wt{\cN}.\]

\begin{Remark}\label{rmk:elp}
  It will be helpful to keep in mind the following properties, which can be easily verified by going through definitions:
  \begin{enumerate}
    \item For any $x\in M\setminus\Sing(X)$, and $L=\langle X(x)\rangle$, $\tilde{\psi}_t|_{\widetilde{\cN}_L}$ can be naturally identified with the linear Poincar\'{e} flow $\psi_t|_{\cN_x}$, and $\tilde{\psi}^*_t|_{\widetilde{\cN}_L}$ can be naturally identified with $\psi^*_t|_{\cN_x}$;
    \item For a singularity $\si\in\Sing(X)$, if there is a $\Phi_t$-invariant splitting $T_{\si}M=E(\si)\oplus F(\si)$ such that $\dim F(\si)=1$ and $E(\si)\perp F(\si)$, then $\tilde{\psi}_t|_{\widetilde{\cN}_{F(\si)}}$ can be naturally identified with $\Phi_t|_{E(\si)}$, and $\tilde{\psi}^*_t|_{\wt{\cN}_{F(\si)}}$ can be naturally identified with $\frac{1}{\|\Phi_t|_{F(\si)}\|}\cdot\Phi_t|_{E(\si)}$.
  \end{enumerate}
\end{Remark}

\subsection{Sectional Poincar\'e flow and rescaling}
\label{sect:sectional-Poinc-flow}

We will also consider a non-linear flow on the normal bundle, called  {\em sectional Poincar\'e flow}, which is a lift of local holonomies induced by the flow on the manifold.

Given $X\in\xX^1(M)$, for any $x\in M\setminus\Sing(X)$ and $r>0$, define $\cN_x(r)=\{v\in\cN_x:\|v\|\leq r\}$ and $N_x(r)=\exp_x(\cN_x(r))$, where $\exp_x:T_xM\to M$ is the exponential map. Then one can take $\be_0>0$ small such that for any $x\in M\setminus\Sing(X)$, the exponential map $\exp_x$ is a diffeomorphism from $\cN_x(\be_0)$ onto its image.

\begin{figure}[hbtp]
	\centering
	\includegraphics[width=.6\textwidth]{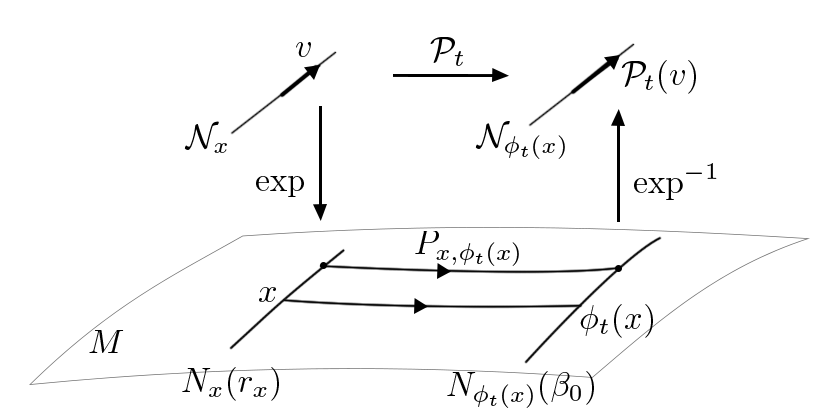}	
	\caption{The sectional Poincar\'e flow}\label{fig:spf}
\end{figure}

Figure \ref{fig:spf} illustrates the sectional Poincar\'e flow $\mathcal{P}_t$ defined as the following. One can refer to \cite{GY} and \cite{CrY18} for more details.

\nibf[The {\em sectional Poincar\'e flow $(\cP_t)_{t\in\RR}$}.]
Given $x\in M\setminus\Sing(X)$, there is $r_x>0$ such that the flow $\phi_t$ induces a local holonomy map $P_{x,\phi_t(x)}$ from $N_x(r_x)$ to $N_{\phi_t(x)}(\be_0)$, for any $t\in[-1,1]$. We denote by $(P_t)_{t\in\RR}$ the family of local holonomy maps, $P_t|_{N_x(r_x)}=P_{x,\phi_t(x)}$.
This allows us to define a local flow $(\cP_t)_{t\in\RR}$, which is called the sectional Poincar\'e flow, in a neighborhood of the 0-section in $\cN=\cup_{x\in M\setminus\Sing(X)}\cN_x$ as follows:
\[\cP_t(v)=\exp^{-1}_{\phi_t(x)}\circ P_{x,\phi_t(x)}\circ \exp_x(v),\]
for any $t\in[-1,1]$ and $v\in\cN_x(r_x)$.

\begin{Remark}\label{rmk:spf-uniform-size}
It is shown in \cite[Lemma 2.3]{GY} that for any $T>0$, there is $\be_1>0$ such that for any $x\in M\setminus\Sing(x)$, $\cP_t$ is well-defined on $\cN_x(\be_1\|X(x)\|)$ for any $t\in[-T,T]$. Moreover, the constant $\be_1$ can be chosen uniformly for vector fields in a $C^1$ neighborhood of $X$.
\end{Remark}

\nibf[The {\em rescaled sectional Poincar\'e flow $(\cP^*_t)_{t\in\RR}$}.]
Similar to the linear Poincar\'e flow, one can rescale the sectional Poincar\'e flow to obtain the rescaled sectional Poincar\'e flow $(\cP^*_t)_{t\in\RR}$:
\[\cP^*_t(v)=\|X(\phi_t(x))\|^{-1}\cdot\cP_t(\|X(x)\|\cdot v),\]
for any $t\in[-T,T]$, $x\in M\setminus\Sing(X)$ and $v\in\cN_x(\be_1)$.

\begin{Remark}\label{rmk:spf}
The differential of the sectional Poincar\'e flow $(\cP_t)_{t\in\RR}$ at the 0-section of $\cN$ is exactly the linear Poincar\'e flow $(\psi_t)_{t\in\RR}$. And the differential of $(\cP^*_t)_{t\in\RR}$ at the 0-section of $\cN$ is exactly the rescaled linear Poincar\'e flow $(\psi^*_t)_{t\in\RR}$.
\end{Remark}

\subsection{Plaque families}\label{sect:plaque-families}

Let $\La$ be any compact $\phi_t$-invariant set. Recall that $\cN_{\La}=\cup_{x\in\La\setminus\Sing(X)}\cN_x$. We assume that there is a {\em dominated splitting} $\cN_{\La}=G^{cs}\oplus G^{cu}$ of {\em index} $i$ with respect to the linear Poincar\'e flow $\psi_t$; more precisely, there exists $T>0$, such that for any $t\geq T$ and for any $x\in\La\setminus\Sing(X)$, $\dim G^{cs}(x)=i$,
\[\|\psi_t|_{G^{cs}(x)}\|\cdot\|\psi_{-t}|_{G^{cu}(\phi_t(x))}\|\leq 1/2.\]
One can immediately check that the inequality above is equivalent to the following:
\[\|\psi^*_t|_{G^{cs}(x)}\|\cdot\|\psi^*_{-t}|_{G^{cu}(\phi_t(x))}\|\leq 1/2.\]
In other words, a dominated splitting with respect to $\psi_t$ is a dominated splitting with respect to $\psi^*_t$, and {\itshape vice versa}. Therefore, one can simply say that $\cN_{\La}=G^{cs}\oplus G^{cu}$ is a dominated splitting. without specifying $\psi_t$ or $\psi^*_t$.

For each $x\in\La\setminus\Sing(X)$, let $G^{cs}(x;\al)$ and $G^{cu}(x;\al)$ be the open ball of radius $\al$ contained in $G^{cs}(x)$ and $G^{cu}(x)$ respectively.
Note that $(\psi^*_t)_{t\in\RR}$ is the differential of the local flow $(\cP^*_t)_{t\in\RR}$ at 0-section (Remark \ref{rmk:spf}), and $\cP^*_t$ is well-defined on $\cN_x(\be_1)$ for any $x\in M\setminus\Sing(X)$ and $t\in[-T,T]$ (Remark \ref{rmk:spf-uniform-size}).

Then one can generalize Hirsch-Pugh-Shub's plaque family theorem
\cite{HPS} to $\cP^*_T$ and obtain two continuous families of $C^1$ embeddings
\[\eta^{cs}_x: G^{cs}(x;\al_0)\to \cN_x,\quad \eta^{cu}_x: G^{cu}(x;\al_0)\to \cN_x,\]
with $x\in\La\setminus\Sing(X)$ and $\al_0>0$ small, such that the following properties hold (\cite[Lemma 2.14]{GY}):
\begin{itemize}
\item For any $x\in\La\setminus\Sing(X)$, one has $\eta^{cs/cu}_x(0_x)=0_x$ and $D\eta^{cs/cu}_x(0_x)=Id_{G^{cs/cu}(x)}$;
\item For any $0<\al\leq \al_0$, let
 \[\cW^{cs,*}_{\al}(x)=\eta^{cs}_x(G^{cs}(x;\al)), \quad\text{and}\ \cW^{cu,*}_{\al}(x)=\eta^{cu}_x(G^{cu}(x;\al)).\]
 Then the two plaque families $(\cW^{cs,*}_{\al_0}(x))_{x\in\La\setminus\Sing(X)}$ and $(\cW^{cu,*}_{\al_0}(x))_{x\in\La\setminus\Sing(X)}$ are locally invariant: for any $\varepsilon>0$ there exists $\de>0$ such that for any $x\in\La\setminus\Sing(X)$ and $t\in[-T,T]$, one has
 \[\cP^*_t(\cW^{cs,*}_{\de}(x))\subset \cW^{cs,*}_{\varepsilon}(\phi_t(x)),\quad\text{and}\ \cP^*_t(\cW^{cu,*}_{\de}(x))\subset \cW^{cu,*}_{\varepsilon}(\phi_t(x)).\]
\end{itemize}
We define for every $x\in\La\setminus\Sing(X)$ and $r\in(0,\al_0\|X(x)\|]$ the following
\[\cW^{cs}_{r}(x)=\|X(x)\|\cW^{cs,*}_{r/\|X(x)\|}(x), \quad \cW^{cu}_{r}(x)=\|X(x)\|\cW^{cu,*}_{r/\|X(x)\|}(x).\]
Then the plaque families $(\cW^{cs/cu}_{r}(x))_{x\in\La\setminus\Sing(X)}$ are locally invariant for $\cP_t$, {\itshape i.e.} for any $\varepsilon>0$, there is $\de>0$ such that
\[\cP_{t}\left(\cW^{cs/cu}_{\de\|X(x)\|}(x)\right)\subset \cW^{cs/cu}_{\varepsilon\|X(\phi_{t}(x))\|}(\phi_{t}(x)),\]
for any $x\in\La\setminus\Sing(X)$ and $t\in[-T,T]$.

\begin{Remark}
	The size of plaque families $(\cW^{cs/cu}_{r}(x))_{x\in\La\setminus\Sing(X)}$ is only uniform after a rescale of size $\|X(x)\|$. This is one of the main differences between singular and non-singular flows (diffeomorphisms).
\end{Remark}

\begin{Remark}\label{rmk:plaque-family}
For any $x\in\La\setminus\Sing(X)$ and $r\in(0,\al_0\|X(x)\|]$, one can define
\[W^{cs/cu}_{r}(x)=\exp_x(\cW^{cs/cu}_{r}(x)).\]
Then these are plaque families on the manifold $M$, and satisfy the following properties:
\begin{itemize}
\item $T_xW^{cs}_{r}(x)=G^{cs}(x)$ and $T_xW^{cu}_{r}(x)=G^{cu}(x)$;
\item $(W^{cs/cu}_{r}(x))_{x\in\La\setminus\Sing(X)}$ are locally invariant for the family of local holonomy maps $P_t$, i.e., for any $\varepsilon>0$, there is $\de>0$ such that
\[P_{t}\left(W^{cs/cu}_{\de\|X(x)\|}(x)\right)\subset W^{cs/cu}_{\varepsilon\|X(\phi_{t}(x))\|}(\phi_{t}(x)),\]
for any $x\in\La\setminus\Sing(X)$ and $t\in[-T,T]$.
\end{itemize}
For convenience, we will denote $W^{cs}(x)=W^{cs}_{\al_0\|X(x)\|}(x)$ and $W^{cu}(x)=W^{cu}_{\al_0\|X(x)\|}(x)$.
\end{Remark}

\subsection{Hyperbolic points and invariant manifolds}
The concept of hyperbolic times is prominent in studying the hyperbolic structure of diffeomorphisms and flows.
See~\cite{Lia89}, \cite{ABV00} and~\cite{PuS} for some of its applications. This section contains some basic properties of hyperbolic times and more importantly, its existence due to the Pliss lemma. Further discussion on this topic can be found in Section~\ref{s.7.2}.
\subsubsection{$\psi^*_t$-contracting/expanding points}

\begin{Definition}[{$\psi^*_t$-contracting/expanding \cite[Definition 2.17]{GY}}]
Let $E\subset\cN_{\La}$ be an invariant subbundle of the linear Poincar\'e flow $\psi_t$. For any $C>0$, $\eta>0$ and $T>0$, a point $x\in\La\setminus\Sing(X)$ is said to be $(C,\eta,T,E)$-$\psi^*_t$-contracting if there is a partition of time $0=t_0<t_1<\cdots<t_n<\cdots$ such that the following conditions hold:
\begin{itemize}
\item $t_{n+1}-t_n\leq T$ for any $n\in\NN$, and $\lim_{n\to\infty}t_n=\infty$;
\item $\prod_{i=0}^{n-1}\|\psi^*_{t_{i+1}-t_i}|_{E(\phi_{t_i}(x))}\|\leq C\e^{-\eta t_n},\quad \forall n\in\NN_+$.
\end{itemize}
The point $x$ is said to be $(C,\eta,T,E)$-$\psi^*_t$-expanding if it is $(C,\eta,T,E)$-$\psi^{-X,*}_t$-contracting for $-X$.
\end{Definition}

Suppose there is a dominated splitting $\cN_{\La}=G^{cs}\oplus G^{cu}$. Then it is well known that for a $(C,\eta,T,G^{cs})$-$\psi^*_t$-contracting point $x\in\La\setminus\Sing(X)$, the stable set of $\orb(x)$ contains an open disk in the plaque $W^{cs}(x)$. Let us recall some definitions. A strictly increasing homeomorphism $\theta:\RR\to\RR$ with $\theta(0)=0$ is called a {\em reparametrization}. For any $x$, we define
\begin{itemize}
\item $W^s(x)=\{y\in M: \exists \txtrm[reparametrization] \theta, \txtrm[s.t.] \limsup\limits_{t\to+\infty}\log d(\phi_{\theta(t)}(y),\phi_t(x))<0\}$;
\item $W^u(x)=\{y\in M: \exists \txtrm[reparametrization] \theta, \txtrm[s.t.] \limsup\limits_{t\to-\infty}\log d(\phi_{\theta(t)}(y),\phi_t(x))<0\}$;
\end{itemize}
where $d(\cdot,\cdot)$ is the distance function over $M$ induced by the Riemannian metric. Denote
\[W^s(\orb(x))=\bigcup_{y\in\orb(x)}W^s(y),\quad W^u(\orb(x))=\bigcup_{y\in\orb(x)}W^u(y).\]

\begin{Lemma}[{\cite[Lemma 2.18]{GY}}]
\label{lem:stable-manifolds}
Assume that $\cN_{\La}$ admits a dominated splitting $\cN_{\La}=G^{cs}\oplus G^{cu}$ of index $i$ with respect to the linear Poincar\'e flow $(\psi_t)_{t\in\RR}$. For any $C>0$, $\eta>0$ and $T>0$, there is $\de_0>0$ such that for any $x\in\La\setminus\Sing(X)$,
\begin{itemize}
\item if it is $(C,\eta,T,G^{cs})$-$\psi^*_t$-contracting, then $W^{cs}_{\de_0\|X(x)\|}(x)\subset W^s(\orb(x))$;
\item if it is $(C,\eta,T,G^{cu})$-$\psi^*_t$-expanding, then $W^{cu}_{\de_0\|X(x)\|}(x)\subset W^u(\orb(x))$.
\end{itemize}
\end{Lemma}

\begin{Remark}\label{rmk:stable-manifolds}
The proof of Lemma \ref{lem:stable-manifolds} uses a similar argument as in the proof of \cite[Corollary 3.3]{PuS}. The constant $\de_0$ is determined together with $\vep>0$ and $0<\eta'<\eta$ such that for any $x\in\La\setminus\Sing(X)$, if it is $(1,\eta,T,G^{cs})$-$\psi^*_t$-contracting with the partition of time $0=t_0<t_1<\cdots<t_n<\cdots$, then for all $n\in\NN$,
\[\mathrm{diam}(P_{t_n}W^{cs}_{\de_0\|X(x)\|}(x))\leq e^{-\eta't_n}\mathrm{diam}(W^{cs}_{\de_0\|X(\phi_{t_n}(x))\|}(\phi_{t_n}(x))),\]
and
\[P_{t_n}W^{cs}_{\de_0\|X(x)\|}(x)\subset W^{cs}_{\vep\|X(\phi_{t_n}(x))\|}(\phi_{t_n}(x)).\]
A similar result holds for the plaque family $W^{cu}$.
Moreover, the constant $\de_0$ can be chosen uniformly for vector fields in a $C^1$ neighborhood of $X$.
\end{Remark}
\subsubsection{Periodic orbits and a lemma of Pliss type}

Given a periodic orbit $\gamma$ for a vector field $X$, we denote by $\tau(\gamma)$ the minimal periodic of $\gamma.$ The goal of this subsection is to derive the existence of $\psi_t$-contracting/expanding points from a mild hyperbolic property of periodic orbits.

\begin{Definition}[$\psi_t$-contracting/expanding at period]
\label{def:contracting-at-period}
Given $C>0, \eta>0$ and $T>0$, let $\ga$ be a periodic orbit of $X\in\xX^1(M)$, $E\subset\cN_{\ga}$ be an invariant subbundle with respect to the linear Poincar\'e flow $\psi_t$. Assume that $\tau(\ga)\ge T$. $\ga$ is called $(C,\eta,T,E)$-$\psi_t$-contracting at period, if for any $x\in\ga$ and any partition of time $0=t_0<t_1<\cdots<t_n=\tau(\ga)$ with $t_{i+1}-t_i\geq T$ for all $0\leq i\leq n-1$, one has
\begin{equation}\label{eqn:contracting-at-period}
  \prod_{i=0}^{n-1}\|\psi_{t_{i+1}-t_i}|_{E(\phi_{t_i}(x))}\|\leq C\e^{-\eta \tau(\ga)}.
\end{equation}
The orbit $\ga$ is called $(C,\eta,T,E)$-$\psi_t$-expanding at period if it is $(C,\eta,T,E)$-$\psi_t$-contracting at period for $-X$.
\end{Definition}

\begin{Remark}
Since $\Phi_{\tau(\ga)} X(x) = X(\phi_{\tau(\ga)}(x)) = X(x)$ for all $x\in\ga$, we have
\[\prod_{i=0}^{n-1}\|\psi_{t_{i+1}-t_i}|_{E(\phi_{t_i}(x))}\|\leq C\e^{-\eta \tau(\ga)} \iff \prod_{i=0}^{n-1}\|\psi_{t_{i+1}-t_i}^*|_{E(\phi_{t_i}(x))}\|\leq C\e^{-\eta \tau(\ga)}.\]
	One can define $(C,\eta,T,E)$-$\psi_t^*$-contracting (or expanding) at period simply by replacing $\psi_t$ with $\psi_t^*$ in Definition \ref{def:contracting-at-period}. Then $\ga$ is $(C,\eta,T,E)$-$\psi_t$-contracting (or expanding) at period if and only if it is $(C,\eta,T,E)$-$\psi^*_t$-contracting (or expanding) at period.
	Therefore we may simply say that $\ga$ is $(C,\eta,T,E)$-contracting (or expanding) at period without referring to either $\psi$ or $\psi^*$.
	
\end{Remark}

\begin{Lemma}[{\cite[Lemma 2.22]{GY}}]\label{lem:pliss}
Given $X\in\mathscr{X}^1(M)$, $C>0$, $\eta>0$ and $T>0$. For any $\eta'\in(0,\eta)$, there exists $\tau_0=\tau_0(C,T,\eta,\eta')>0$ such that if $\ga$ is a periodic orbit with period $\tau(\ga)\geq\tau_0$, $E\subset\cN_{\ga}$ is a $\psi_t$-invariant subbundle, and if $\ga$ is $(C,\eta,T,E)$-$\psi_t$-contracting at period, then there exists a $(1,\eta',T,E)$-$\psi^*_t$-contracting point $x\in\ga$.
\end{Lemma}

\subsection{Fundamental sequence with dominated splitting}

Given $X\in\mathscr{X}^1(M)$ away from homoclinic tangencies, it is known that there is a neighborhood $\cU$ of $X$ with the property that for any $Y\in\cU$, any periodic orbit $\ga$ of $Y$, the normal bundle over $\ga$ admits a dominated splitting, and this domination is uniform in $\cU$ (see Lemma \ref{lem:BasicAwayTang}). Therefore, it is natural to consider a sequence of periodic orbits for vector fields $X_n$ converging to $X$ in the $C^1$ topology. The following definition is given by Liao \cite{Lia81}.

\begin{Definition}[Fundamental sequence]
\label{def:fund-seq}
Let $\{X_n\}$ be a sequence of $C^1$ vector fields on $M$ and suppose there is a periodic orbit $\ga_n$ of $X_n$ for each $n$. The sequence $\{(\ga_n,X_n)\}$ is called a  {\em fundamental sequence of $X$}, if $\|X_n-X\|_{C^1}\to 0$ and there exists a compact invariant set $\La$ of $X$, such that $\ga_n$ converges to $\La$ in the Hausdorff topology. The set $\La$ is called the {\em limit of the fundamental sequence}. Moreover, if $\ind(\ga_n)=i$ for all $n$, we call $\{(\ga_n,X_n)\}$ a {\em fundamental $i$-sequence}.
\end{Definition}

We use $(\ga_n,X_n)\to (\La,X)$ with the meaning that $\{(\ga_n,X_n)\}$ is a fundamental sequence of $X$ with the limit $\La$.

\begin{Definition}[Dominated splitting of fundamental sequence]
\label{def:dom-fund-seq}
Let $\{(\ga_n,X_n)\}$ be a fundamental sequence of $X$. The sequence is said to admit a dominated splitting of index $i$ if there exist $T>0$ and a $\psi^{X_n}_t$-invariant splitting $\cN_{\ga_n}=G^{cs}_n\oplus G^{cu}_n$ for each $n$, such that for $n$ large enough, for any $x\in\ga_n$, $t\geq T$, one has $\dim G^{cs}_n=i$ and
\[\|\psi^{X_n}_t|_{G^{cs}_n(x)}\|\cdot\|\psi^{X_n}_{-t}|_{G^{cu}_n(\phi^{X_n}_t(x))}\|\leq \frac{1}{2}.\]
The constant $T$ is called the {\em domination constant} and we call the splitting {\em $T$-dominated}.
\end{Definition}

\begin{Remark}
\label{rmk:fundamental-sequences}
For any fundamental sequence $\mathscr{F}=\{(\ga_n,X_n)\}$ of $X$, we define
\[\De(\mathscr{F})=\{L\in G^1:\ \exists x_n\in\ga_n, \txtrm[such that] \langle X_n(x_n)\rangle\to L\}.\]
Then $\De(\mathscr{F})$ is a compact $\hat{\Phi}^X_t$-invariant set. One can see that $\{(\ga_n,X_n)\}$ admits a dominated splitting of index $i$ if and only if $\De(\mathscr{F})$ admits a dominated splitting of index $i$ with respect to the extended linear Poincar\'{e} flow $\tilde{\psi}^X_t$. In fact, in view of Remark \ref{rmk:elp}, if $\cN_{\ga_n}=G^{cs}_n\oplus G^{cu}_n$ is a dominated splitting of index $i$, then the dominated splitting $\widetilde{\cN}_{\De(\mathscr{F})}=G^{cs}\oplus G^{cu}$ with respect to $\tilde{\psi}^X_t$ can be obtained by taking limits of the former; the other direction is just because a dominated splitting can be extended to neighborhoods, c.f. \cite{BDV05}.

Hence if $\mathscr{F}=\{(\ga_n,X_n)\}$ admits a dominated splitting of index $i$ and $(\ga_n,X_n)\to (\La,X)$, then since $\be(\De(\fF))=\La$, we have in particular that $\cN_{\La}$ admits a dominated splitting $G^{cs}\oplus G^{cu}$ of index $i$ with respect to the linear Poincar\'{e} flow $\psi^X_t$. Then there exist two continuous plaque families $(W^{cs}_{\al}(x))_{x\in\La\setminus\Sing(X)}$ and $(W^{cu}_{\al}(x))_{x\in\La\setminus\Sing(X)}$ tangent to $G^{cs}$ and $G^{cu}$ respectively, which are locally invariant for the family of local holonomy maps induced by the flow (Remark \ref{rmk:plaque-family}).
\end{Remark}

\subsection{Matching of dominated splittings near a singularity}
\label{sect:matching}

Let $\fF=\{(\ga_n,X_n)\}$ be a fundamental sequence, $(\ga_n,X_n)\to (\La, X)$. Suppose there is a hyperbolic singularity $\si\in\La\cap\Sing(X)$. Define
\[\De_{\si}(\fF)=\{L\in\De(\fF):\be(L)=\si\},\]
where $\De(\fF)$ is defined as in Remark \ref{rmk:fundamental-sequences}.
Then $\De(\fF)$ and $\De_{\si}(\fF)$ are nonempty compact $\hat{\Phi}_t$-invariant sets. As explained in Remark \ref{rmk:fundamental-sequences}, a dominated splitting over the fundamental sequence can be extended to $\De(\fF)$ and, in particular, to $\De_{\si}(\fF)$.

Since $\si$ can be accumulated by periodic orbits, it is a saddle, {\itshape i.e.}, the stable subspace $E^s(\si)$ and unstable subspace $E^u(\si)$ are both nontrivial.
The following lemma is contained in the proof of \cite[Lemma 4.1, 4.2]{LGW05}.

\begin{Lemma}\label{lem:elp-dom}
  Suppose there is a dominated splitting $\widetilde{\cN}_{\De_{\si}(\fF)}=G^{cs}\oplus G^{cu}$ with respect to the extended linear Poincar\'{e} flow, and the index $\dim{G^{cs}}<\ind(\si)$, then there is a dominated splitting $E^s(\si)=E^{ss}(\si)\oplus E^{c}(\si)$ with respect to the tangent flow, such that $\dim{E^{ss}(\si)}=\dim{G^{cs}}$.
  Moreover, suppose $E^{ss}(\si)\perp E^{cu}(\si)$ where $E^{cu}(\si)=E^c(\si)\oplus E^u(\si)$, then
  \begin{itemize}
    \item for any $L\in\De_{\si}(\fF)\cap E^{cu}(\si)$, one has $G^{cs}(L)=\{L\}\times E^{ss}(\si)$, and
    \item for any $L\in\De_{\si}(\fF)\cap E^{ss}(\si)$, one has $G^{cu}(L)\supset \{L\}\times E^u(\si)$.
  \end{itemize}
\end{Lemma}

The following lemma  is essentially contained in \cite[Lemma 4.3]{LGW05}, with similar ideas first appeared in \cite{MPP}.

\begin{Lemma}\label{lem:elp-dom2}
  Under the same assumptions as in Lemma \ref{lem:elp-dom}, one has
\[\De_{\si}(\fF)\cap (E^{ss}(\si)\oplus E^u(\si))\setminus(E^{ss}(\si)\cup E^u(\si))=\emptyset.\]
\end{Lemma}

\begin{Remark}\label{rmk:elp-dom}
	Symmetric results hold for Lemma \ref{lem:elp-dom} and Lemma \ref{lem:elp-dom2} when $\dim G^{cs}\geq \ind(\si)$, or equivalently, when $\dim G^{cu}<\dim M-\ind(\si)$. In particular, when $\dim G^{cs}=\ind(\si)$, there exists a dominated splitting $E^u(\si)=E^c(\si)\oplus E^{uu}(\si)$ with respect to the tangent flow such that $\dim E^c(\si)=1$ and
	\[\De_{\si}(\fF)\cap(E^{uu}(\si)\oplus E^s(\si))\setminus(E^{uu}(\si)\cup E^s(\si))=\emptyset.\]
\end{Remark}

Usually in practice, the dominated splitting $\widetilde{\cN}_{\De_{\si}(\fF)}=G^{cs}\oplus G^{cu}$ will be induced from that of the fundamental sequence $\fF$. Hence, informally, Lemma \ref{lem:elp-dom} tells how the splittings of periodic orbits should match with that of the singularity $\si$. Then Lemma \ref{lem:elp-dom2} tells how these periodic orbits with insufficient stable index (meaning that $\dim{G^{cs}}<\ind(\si)$) should behave in a small enough neighborhood of the singularity, or more precisely, in what directions these orbits can come in and get out of an arbitrarily small neighborhood of $\si$. One can see the set $(E^{ss}(\si)\oplus E^u(\si))\setminus(E^{ss}(\si)\cup E^u(\si))$ as a ``restricted area'' for the fundamental sequence $\fF$. Similarly, when the stable index is high (i.e., $\dim{G^{cs}}\ge\ind(\si)$), the restricted region is $(E^{uu}(\si)\oplus E^s(\si))\setminus(E^{uu}(\si)\cup E^s(\si))$ with $\dim E^{uu}(\sigma) = \dim G^{cu}$. 

These two lemmas play a key role in the local analysis around singularities.
In particular, Lemma \ref{lem:elp-dom2} will be used to give the critical contradiction for the proof of our main theorem. The general idea goes as the following: once the conclusion of Lemma \ref{lem:elp-dom2} holds, to obtain a contradiction one need only show existence of a fundamental sequence with proper index that steps into the ``restricted area''; this is carried out by a generic analysis under suitable assumptions; see for example, Lemma \ref{lem:gen_1}.

Let us explain briefly the idea behind Lemma \ref{lem:elp-dom2}. Suppose on the contrary that there exists $L\in\De_{\si}(\fF)$ in the ``restricted area'' $(E^{ss}(\si)\oplus E^u(\si))\setminus(E^{ss}(\si)\cup E^u(\si))$. Let $\al(L)$, $\om(L)$ be the $\al$-limit set and $\om$-limit set of $L$ under the extended tangent flow, respectively. It is easy to see that $\al(L)\subset E^{ss}(\si)$ and $\om(L)\subset E^u(\si)$. By Lemma \ref{lem:elp-dom}, for any $L'\in\al(L)$, $G^{cu}(L')\supset \{L'\}\times E^u(\si)$; and for any $L'\in\om(L)$, $G^{cs}(L')=\{L'\}\times E^{ss}(\si)$. Then one can show that any dominated splitting of the normal bundle over the orbit of $L$ will have a mismatch with either $\wt{\cN}_{\al(L)}=G^{cs}\oplus G^{cu}$ or  $\wt{\cN}_{\om(L)}=G^{cs}\oplus G^{cu}$. One can refer to \cite{LGW05} for details.

These results (Lemma \ref{lem:elp-dom}, Lemma \ref{lem:elp-dom2} and Remark \ref{rmk:elp-dom}) can be restated in a slightly different setting as we are going to show below.
Following from \cite{LGW05}, for any set $\La$ of $X$, we define
\[B(\La)=\{L\in G^1: \exists x_n\in\La, \txtrm[such that] \langle X(x_n)\rangle\to L\},\]
which is compact by definition. Moreover, for any singularity $\si\in\La$, the set
\[B_{\si}(\La)=\{L\in B(\La):\, \be(L)=\si\}\]
is a non-empty compact invariant subset of $B(\La)$. Recall that a compact invariant set $\La$ is called {\em chain transitive} if for any $x,y\in\La$, $y$ is chain attainable from $x$ in $\Lambda$, i.e., for any $\vep>0$, there is an $\vep$-chain $x=x_0,x_1,\ldots,x_n=y$ from $x$ to $y$, such that $x_i\in \La$ for all $i=0,\cdots,n$.
Every chain recurrence class is chain transitive.

\begin{Lemma}[\cite{LGW05}]\label{lem:matching}
	Let $\La$ be a non-trivial compact chain transitive set of $X$ that contains a hyperbolic singularity $\si$. Suppose there is a dominated splitting $\widetilde{\cN}_{B_{\si}(\La)}=G^{cs}\oplus G^{cu}$ with respect to the extended linear Poincar\'e flow with $\dim G^{cs}\ge\ind(\si)$, then there is a dominated splitting $E^u(\si)=E^{c}(\si)\oplus E^{uu}(\si)$ with respect to the tangent flow such that $\dim E^{uu}(\si)=\dim G^{cu}$ and
	\begin{equation*}
	B_{\si}(\La)\cap (E^{s}(\si)\oplus E^{uu}(\si))\setminus (E^{s}(\si)\cup E^{uu}(\si))=\emptyset.
	\end{equation*}
	A symmetric result holds when $\dim G^{cs}< \ind(\si)$.
\end{Lemma}


\subsection{A theorem of Liao}
\label{sect:Liao-thm}

The following theorem establishes the transverse intersection between the invariant manifold of a singularity $\sigma$ and $\psi_t^*$-contracting points close to $\sigma$.

\begin{Theorem}[Liao \cite{Lia89}. See also~{\cite[Lemma 3.12 and 3.13]{PYY2}}] \label{thm:Liao}
  Let $\{(\ga_n,X_n)\}$ be a fundamental sequence, $(\ga_n,X_n)\to(\La,X)$. Suppose that it admits a dominated splitting $\cN_{\ga_n}=G^{cs}_n\oplus G^{cu}_n$ of index $i$ and there exists a hyperbolic singularity $\si\in\La$ with $\ind(\si)\leq i$. If there are constants $C>0$, $\eta>0, T>0$, and a sequence of points $x_n\in\ga_n$, such that $x_n$ is $(C,\eta,T,G^{cs}_n)$-$\psi^{X_n,*}_t$-contracting and $x_n\to\si$, then
\[W^u(\si_{X_n},X_n)\pitchfork W^s(\ga_n,X_n)\neq\emptyset, \txtrm[for $n$ large.]\]
\end{Theorem}

Note that in the theorem, the size of stable manifold of a periodic point $x_n$ is proportional to $\|X_n(x_n)\|$ (Lemma \ref{lem:stable-manifolds}), which goes to 0 as $n\to\infty$. This makes it difficult to guarantee an intersection between the stable manifold of $\ga_n$ and the unstable manifold of $\si_{X_n}$. The proof of this theorem, which can be found in \cite[Section 3]{PYY2}, exploits the fact that hyperbolic points can only exist in an arbitrarily small cone of the local unstable manifold of $\si$. This involves a careful local analysis around the singularity. See also~\cite{SYY} for some related discussions.

The previous theorem is a generalization of the following result, obtained by Liao.
\begin{Corollary}[Liao \cite{Lia89}]\label{cor:Liao}
  Let $\si$ be a hyperbolic singularity of $X\in\mathscr{X}^1(M)$. If there are $C>0$, $\eta>0$, $T>0$, and a sequence of vector fields $X_n\to X$, each with a periodic orbit $\ga_n$, and a point $x_n\in\ga_n$, such that $x_n$ is $(C,\eta,T,\cN^{X_n})$-$\psi^{X_n,*}_t$-expanding and $x_n\to\si$, then
\[W^s(\si_{X_n},X_n)\pitchfork W^u(\ga_n,X_n)\neq\emptyset, \txtrm[for $n$ large.]\]
\end{Corollary}

\subsection{Generic results}
\label{sect:genericity}

We gather here some results on $C^1$ generic vector fields that will be used throughout the paper. Recall that a $C^1$ vector field $X$ is called {\em weak Kupka-Smale} if every critical element of $X$ is hyperbolic; the vector field $X$ is called {\em Kupka-Smale} if it is weak Kupka-Smale and, moreover, the stable manifold of each critical element intersects the unstable manifold of any other critical element transversely.


\begin{Lemma}\label{lem:gen_0}
For $C^1$ generic vector field $X\in \mathscr{X}^1(M)$, the following properties hold.
\begin{enumerate}
\item $X$ is Kupka-Smale.
\item There is a neighborhood $\cU$ of $X$, such that for any $Y\in\cU$, $Y$ has only finitely many singularities, each of which satisfies $\la+\la'\neq 0$, where $\la$ and $\la'$ are any two Lyapunov exponents of $\Phi^Y_t(\si)$.
\item Every nontrivial compact chain transitive set is the limit of a sequence of periodic orbits in the Hausdorff topology. \label{item:gen_0-chain-transitive-set}
\item If $\La$ is the limit of a fundamental $i$-sequence, then there exists a sequence of periodic orbits $\ga_n$ of $X$, such that $\ind(\ga_n)=i$ and $\ga_n\to\La$ in the Hausdorff topology. \label{item:gen_0-fund-seq}
\item If $C$ is a nontrivial chain recurrence class of $X$ and contains a periodic orbit $P$, then $C$ coincides with the homoclinic class of $P$. \label{item:gen_0-homoclinic-class}
\item For any $\si\in\Sing(X)$, the chain recurrence class $C(\si,X)$ varies continuously at $X$. \label{item:gen_0-continuity-chain-class}
\item For any $\si\in\Sing(X)$, if for any $n\geq 1$, there exists $X_n$ with a periodic orbit $\ga_n$ of $X_n$ such that $\|X_n-X\|_{C^1}<1/n$ and $W^u(\si_{X_n})\pitchfork W^s(\ga_n)\neq\emptyset$, then there is a periodic orbit $\ga$ of $X$, such that $W^u(\si)\pitchfork W^s(\ga)\neq\emptyset$. \label{item:gen_0-transverse-intersection}
\item For any critical elements $c$ and $c'$ of $X$, if $C(c,X)=C(c',X)$, then there is a neighborhood $\cU$ of $X$, such that for any $Y\in\cU$, one has $C(c_Y,Y)=C(c'_Y,Y)$.
\item For any critical element $c$ of $X$, if the chain recurrence class $C(c,X)$ is Lyapunov stable, then there is a neighborhood $\cU$ of $X$ such that for any weak Kupka-Smale vector field $Y\in\cU$, $C(c_Y,Y)$ is a quasi-attractor and, in particular, it is Lyapunov stable.\label{item:gen_0-last}
\end{enumerate}
\end{Lemma}

\begin{proof}
The first two items are classical results; item 3 is proven in \cite{Cro}; item 4 is proven in \cite{Wen04}; item 5 and 6 are consequences of the connecting lemma for pseudo-orbits \cite{BoC}; item 7 is standard as a transverse intersection of stable and unstable manifolds is robust under perturbations, for a rigorous proof one may use the same argument as in the proof of \cite[Lemma 3.2]{GW}; 
item 8 is given in \cite[Lemma 3.12]{GY}.
Let us prove the last item, which is a slightly stronger version of \cite[Lemma 3.14]{GY}. In fact, in their proof of \cite[Lemma 3.14]{GY}, it is shown that there is a neighborhood $\cU$ of $X$ such that for any weak Kupka-Smale $Y\in\cU$, one has $C(c_Y,Y)=W^{ch,u}(C(c_Y,Y))$, where $W^{ch,u}(C(c_Y,Y))$ is the chain-unstable set of $C(c_Y,Y)$. By definition, $W^{ch,u}(C(c_Y,Y))=\cap_{\varepsilon>0} W^{ch,u}_{\varepsilon}(C(c_Y,Y))$ and each $W^{ch,u}_{\varepsilon}(C(c_Y,Y))$ is an open neighborhood of $C(c_Y,Y)$. It follows that $C(c_Y,Y)$ is a quasi-attractor since $\phi^Y_1(\Cl(W^{ch,u}_{\varepsilon}(C(c_Y,Y)))\subset W^{ch,u}_{\varepsilon}(C(c_Y,Y))$.
\end{proof}

\begin{Remark}
  Although some of the results (item 3-6) listed in Lemma \ref{lem:gen_0} are originally given for diffeomorphisms, they also hold for flows. This is because in the proof of these results, we either use the connecting lemma for pseudo-orbits {\em outside} a neighborhood of singularities (as in item 3, item 5-6), or rely on the persistency of hyperbolic periodic orbits (as in item 4). Such operations or arguments apply also to flows.
\end{Remark}

Let $\si$ be a hyperbolic singularity of a vector field $X$. Suppose there is a dominated splitting of the unstable subspace $E^u(\si)=E^c(\si)\oplus E^{uu}(\si)$ (possibly as a consequence of Lemma~\ref{lem:matching}). Let $W^{uu}(\si)$ be the strong unstable manifold tangent to $E^{uu}(\si)$ at $\si$. A {\em fundamental domain} of $W^{uu}(\si)$ is a cross-section of $X$ restricted to $W^{uu}(\si)\setminus\{\si\}$, such that it intersects every orbit in $W^{uu}(\si)\setminus\{\si\}$. Note that the subspace $E^c(\si)$ can be trivial, in which case $W^{uu}(\si)$ is just the unstable manifold of $\si$.
Recall that we have defined for any set $\La$ of $X$ the following compact subset of $G^1$:
\[B(\La)=\{L\in G^1: \exists x_n\in\La, \txtrm[such that] \langle X(x_n)\rangle\to L\}.\]


\begin{Lemma}\label{lem:gen_1}
For a $C^1$ generic vector field $X\in\mathscr{X}^1(M)$, let $\si\in\Sing(X)$ be any hyperbolic singularity, $E^u(\si)=E^c(\si)\oplus E^{uu}(\si)$ be any dominated splitting and $W^{uu}(\si)$ be the strong unstable manifold tangent to $E^{uu}(\si)$. Then we have the following properties.
\begin{enumerate}
\item If $D^{uu}(\si)$ is a fundamental domain of $W^{uu}(\si)$, then there exists a residual subset $R$ in $D^{uu}(\si)$ such that $\Cl(\orb^+(x,X))$ is Lyapunov stable for any $x\in R$.
\item If the chain recurrence class $C(\si)$ is nontrivial and $W^{uu}(\si)\cap C(\si)\setminus\{\si\}\neq \emptyset$, then $\Cl(W^{uu}(\si))\supset C(\si)$.
\item If the chain recurrence class $C(\si)$ is nontrivial and Lyapunov stable, then
\[B(C(\si))\cap(E^s(\si)\oplus E^{uu}(\si))\setminus(E^s(\si)\cup E^{uu}(\si))\neq \emptyset.\]
\end{enumerate}
\end{Lemma}

\begin{proof}
The first item is given in \cite[Proposition 2.7]{CMP} if $E^c(\si)$ is a trivial subspace (see also \cite[Lemma A.1]{ArP}). Their proof can also be applied to the case when $E^c(\si)$ is nontrivial, with obvious modifications. The second item follows from lower semi-continuity of $\Cl(W^{uu}(\si,X))$ and the connecting lemma. To be precise, by semi-continuity, we can assume that $X$ is a continuity point of $\Cl(W^{uu}(\si,X))$. Moreover, we can assume that $X$ has the weak shadowing property for pseudo-orbits \cite{Cro}. Now suppose there is $z\in C(\si,X)\setminus \Cl(W^{uu}(\si,X))$, then there exist a neighborhood $\cU$ of $X$ and a neighborhood $U$ of $z$ such that for any $Y\in\cU$, $\Cl(W^{uu}(\si_Y,Y))\cap U=\emptyset$. Fix a point $x\in W^{uu}(\si,X)\setminus\{\si\}$. Since $z\in C(\si,X)$, by the weak shadowing property for pseudo-orbits, there is an orbit segment $\{\phi_t(y):\ 0\leq t\leq t_0\}$ with $y$ arbitrarily close to $x$ and $\phi_{t_0}(y)$ arbitrarily close to $z$. In particular, $\phi_{t_0}(y)\in U$. Then applying the connecting lemma of Wen-Xia \cite{WeX}, one can connect $x$ to $\phi_{t_0}(y)$ without altering the local strong unstable manifold containing $x$, obtaining a vector field $Y\in\cU$ such that $\phi_{t_0}(y)\in W^{uu}(\si_Y,Y)\cap U$, a contradiction.
Finally, the last item is given in \cite[Lemma 2.5]{Zhe21}.
\end{proof}

\section{Vector fields away from homoclinic tangencies}
\label{sect:away_HT}
%
In this section, we will consider vector fields away from homoclinic tangencies in general, with the main goal of proving Theorem C and also the first two items of Theorem B.

Let us begin by presenting the following lemma (\cite[Lemma 2.9]{GY}), which is a flow version of a result of Wen concerning diffeomorphisms away from homoclinic tangencies \cite{Wen02,Wen04}. Let us denote by $\mathcal{HT}$ the subset of $\xX^1(M)$ having a homoclinic tangency and by $\Cl(\mathcal{HT})$ its closure.

\begin{Lemma}[\cite{Wen02,Wen04,GY}]\label{lem:BasicAwayTang}
For any $X\in\mathscr{X}^1(M)\setminus\Cl(\mathcal{HT})$, there exist a neighborhood $\mathcal{U}$ of $X$ and constants $C>0, \eta>0, \delta>0, T>0$, such that for any periodic orbit $\gamma$ of $Z\in\mathcal{U}$ with $\tau(\gamma)\geq T$, there is an invariant splitting $\mathcal{N}_{\gamma}=G^s\oplus G^c\oplus G^u$ for the linear Poincar\'{e} flow $\psi_t^Z$, satisfying the following properties.
  \begin{enumerate}
  \item {\em Domination}:
    either $G^c=0$, $\lambda^s\leq-\delta<\delta\leq\lambda^u$, and for any $x\in\gamma,
    t\geq T$,
    \begin{equation*}
      \|\psi_t^Z|_{G^s(x)}\|\cdot\|\psi_{-t}^Z|_{G^u(\phi_t^Z(x))}\|\leq\frac{1}{2},
    \end{equation*}
    or $\dim{G^c}=1$, $\lambda^s<-2\delta<-\delta<\lambda^c<\delta<2\delta<\lambda^u$,
    and for any $x\in\gamma, t\geq T$,
    \begin{equation*}
    \begin{aligned}
      &\|\psi_t^Z|_{G^s(x)}\|\cdot\|\psi_{-t}^Z|_{G^c\oplus G^u(\phi_t^Z(x))}\|\leq\frac{1}{2},\\
      &\|\psi_t^Z|_{G^s\oplus G^c(x)}\|\cdot\|\psi_{-t}^Z|_{G^u(\phi_t^Z(x))}\|\leq\frac{1}{2},\\
    \end{aligned}
    \end{equation*}
    where $\la^s$ is the largest Lyapunov exponent for $\psi^Z_t|_{G^s}$, $\la^u$ is the smallest Lyapunov exponent
    for $\psi^Z_t|_{G^u}$, and $\la^c$ is the Lyapunov exponent for $\psi^Z_t|_{G^c}$.
  \item {\em Hyperbolicity at period}:
    for any $x\in\gamma$ and any time partition $0=t_0<t_1<\cdots<t_n=\tau(\gamma)$
    with $t_{i+1}-t_i\geq T$ for each $i\in\{0,1,\cdots, n-1\}$, one has
    \begin{equation*}
    \begin{split}
      &\prod_{i=0}^{n-1}\|\psi^Z_{t_{i+1}-t_i}|_{G^s(\phi^Z_{t_i}x)}\|\leq C\e^{-\eta\tau(\gamma)},\\
      &\prod_{i=0}^{n-1}\|\psi^Z_{t_i-t_{i+1}}|_{G^u(\phi^Z_{t_{i+1}}x)}\|\leq C\e^{-\eta\tau(\gamma)}.\\
    \end{split}
    \end{equation*}
  \end{enumerate}
\end{Lemma}

\begin{Remark}\label{rmk:BasicAwayTang}
	Note that the constant $\eta>0$ can be chosen arbitrarily small without altering the neighborhood $\cU$ of $X$ and other constants. Since the central bundle $G^c$ is either trivial or one-dimensional, the lemma still holds when the constant $\de>0$ is chosen arbitrarily small (which may cause a change in the splitting) as long as $\eta>0$ is chosen correspondingly small, without altering the neighborhood $\cU$, the constants $C$ and $T$. One may refer to \cite{Wen02, Wen04} for more details on the relation of these constants.
\end{Remark}

The properties revealed by Lemma \ref{lem:BasicAwayTang} are profound for dynamics away from homoclinic tangencies. For non-singular flows (and diffeomorphisms), the result can be used together with Liao's selecting lemma \cite{Lia81, Wen08} to ensure the existence of a partially hyperbolic splitting with one-dimensional center on some minimally non-hyperbolic set, see \cite{Wen04,Cro10,XZ}.

The following result is a direct consequence of Lemma \ref{lem:BasicAwayTang}.

\begin{Corollary}\label{cor:BasicAwayTang}
For any $X\in\xX^1(M)\setminus\Cl(\mathcal{HT})$, there exist $C>0$, $\eta>0$ and $T>0$ with the following property: if $\fF=\{(\ga_n,X_n)\}$ is a fundamental $i$-sequence of $X$, then $\fF$ admits a $T$-dominated splitting $\cN_{\ga_n}=G^{cs}_n\oplus G^{cu}_n$ of index $i$. Moreover,  for every $n$ large, either $\ga_n$ is $(C,\eta,T,G^{cs}_n)$-$\psi^{X_n}_t$-contracting at period or there is a $T$-dominated splitting $G^{cs}_n=G^s_n\oplus G^c_n$ such that $\dim G^c_n=1$ and $\ga_n$ is $(C,\eta,T,G^s_n)$-$\psi^{X_n}_t$-contracting at period.
\end{Corollary}

\subsection{Singularities are Lorenz-like}
In this section, we prove the following theorem, which is the main part of Theorem C. Recall that $C^1$ generically we have $\sv(\si)\ne0$ for all $\si\in\Sing(X)$, thanks to Lemma~\ref{lem:gen_0}.

\begin{Theorem}\label{thm:lorenz-like}
For any $C^1$ generic $X\in\mathscr{X}^1(M)\setminus\Cl(\mathcal{HT})$ and any $\si\in\Sing(X)$, suppose the chain recurrence class $C(\si)$ is nontrivial, then
\begin{itemize}
	\item if $\sv(\si)>0$, then $\ind(\si)>1$ and $\si$ is Lorenz-like;
	\item if $\sv(\si)<0$, then $\ind(\si)<\dim M-1$ and $\si$ is reversed Lorenz-like.
\end{itemize}
\end{Theorem}

For the proof, let us begin with a simpler case.
A regular orbit $\Ga$ is called a {\em homoclinic orbit} related to a hyperbolic singularity $\si$ if $\Ga\subset W^s(\si)\cap W^u(\si)$. Note that if $\Ga$ is a homoclinic orbit related to $\si$, then $\Cl(\Ga)=\Ga\cup\{\si\}$.

\begin{Lemma}\label{lem:lorenz-like}
Let $X\in\xX^1(M)\setminus\Cl(\mathcal{HT})$ and $\si$ be a hyperbolic singularity of $X$. Suppose there exists a homoclinic orbit $\Ga$ related to $\si$ and $\sv(\si)>0$. Then there is a dominated splitting $E^s(\si)=E^{ss}(\si)\oplus E^{c}(\si)$ with $\dim E^{c}(\si)=1$ and $W^{ss}(\si)\cap \Ga=\emptyset$, where $W^{ss}(\si)$ is the strong stable manifold of $\si$ tangent to $E^{ss}(\si)$. A symmetric result holds for the case $\sv(\si)<0$.
\end{Lemma}
\begin{proof}
We will only prove the case when $\sv(\si)>0$, as the case $\sv(\si)<0$ can be proven symmetrically.
By changing the Riemannian metric in a neighborhood of $\si$, we can assume that $E^s(\si)\perp E^u(\si)$.

Fix two points $x,y\in\Ga$ on the local stable and local unstable manifolds of $\si$ respectively. With a $C^1$-small perturbation in arbitrarily small neighborhoods of $x$ and $y$, one can obtain a sequence of vector fields $X_n$ having a periodic orbit $\ga_n$, such that  $X_n\to X$ in $C^1$ topology and $\ga_n\to \Cl(\Ga)$ in Hausdorff topology. Moreover, $\tau(\ga_n)\to \infty$ as $n\to\infty$. In other words, we obtain a fundamental sequence $\fF=\{(\ga_n,X_n)\}$ such that $(\ga_n,X_n)\to(\Cl(\Ga),X)$.

\nibf[Claim.]
There is a dominated splitting $\wt{\cN}_{\De(\fF)}=G^{cs}\oplus G^{cu}$ with $\dim G^{cs}=\ind(\si)-1$.

\begin{proof}[Proof of Claim.]

By Lemma \ref{lem:BasicAwayTang}, there exist positive constants $C,\eta,\de$ and $T$ such that for any $n$ large enough, the normal bundle $\cN_{\ga_n}$ admits a dominated splitting $\cN_{\ga_n}=G^s_n\oplus G^c_n\oplus G^u_n$ for the linear Poincar\'e flow $\psi^{X_n}_t$ satisfying the properties listed there. By choosing a subsequence, we may assume that the splittings converge as $n\to\infty$; in particular, there are integers $s,c$ and $u$, such that $\dim G^s_n=s$, $\dim G^c_n=c$, and $\dim G^u_n=u$, for all $n$ large. Moreover, Lemma \ref{lem:BasicAwayTang} implies either $c=0$ or $c=1$.

One can show that either $s=\ind(\si)-1$ or $s+c=\ind(\si)-1$ by using a similar argument as in \cite[Lemma 4.1]{SGW}. We only sketch their proof.
Let $A:E\to E'$ be any linear map between finite dimensional vector spaces $E$ and $E'$. Denote by $\wedge^2 E$ (resp. $\wedge^2 E'$) the second exterior power of $E$ (resp. of $E'$). Then the map $A$ induces a linear map $\wedge^2 A:\wedge^2 E\to \wedge^2 E'$.
For each $n$ large enough, let $E_n=G^s_n\oplus\langle X_n(\ga_n)\rangle$. Then $E_n$ is a $\Phi^{X_n}_t$-invariant subbundle and $\dim E_n=s+1$.
Following from the second item of Lemma \ref{lem:BasicAwayTang}, one can show that for every $n$ large enough and for any $x\in\ga_n$,
\[\prod_{i=0}^{[\tau(\ga_n)/T]-1}\|\wedge^2\Phi^{X_n}_T|_{E_n(\phi^{X_n}_{iT}(x))}\|\leq \e^{-\eta\tau(\ga_n)/2}.\]
Then by Pliss lemma \cite[Lemma 11.8]{Man87}, there exists $x_n\in\ga_n$, $k_n\in \NN$ with $k_n\to \infty$ such that
\[\prod_{i=0}^{j-1}\|\wedge^2\Phi^{X_n}_T|_{E_n(\phi^{X_n}_{iT}(x_n))}\|\leq \e^{-j\eta T/4},\quad \forall j=1,\ldots,k_n.\]
By taking a subsequence, we may assume that $x_n\to z\in\Cl(\Ga)$, and $E_n(x_n)\to E(z)$. Noting that $\phi^X_t(z)\to \si$ as $t\to+\infty$, one can further obtain a subspace $E(\si)\subset T_{\si}M$ with $\dim E(\si)=s+1$, such that
\[\prod_{i=0}^{j-1}\|\wedge^2\Phi^{X}_T|_{\Phi^X_{iT}E(\si)}\|\leq \e^{-j\eta T/4 },\quad \forall j=1,2,\ldots.\]
This implies that $E(\si)$ is sectionally contracting. Since $\sv(\si)>0$, there is a sectionally expanding subspace $F(\si)\subset T_{\si}M$ with $\dim F(\si)=\dim M-\ind(\si)+1$. Hence
\[1\geq \dim E(\si)\cap F(\si)\geq\dim E(\si)+\dim F(\si)-\dim M=s+2-\ind(\si).\]
It follows that $s\leq \ind(\si)-1$. Symmetrically, an argument as above applied to the subbundles $G^u_n\oplus\langle X_n(\ga_n)\rangle$ shows that $u\leq \dim M-\ind(\si)$. This implies $s\geq\ind(\si)-1-c$ since we have $s+c+u=\dim M-1$. As $c=0$ or $1$, one deduces that either $s=\ind(\si)-1$ or $s+c=\ind(\si)-1$.

Then for each $n$ large enough, we define a splitting $\cN_{\ga_n}=F^{cs}_n\oplus F^{cu}_n$ of index $(\ind(\si)-1)$ in the following way: if $s=\ind(\si)-1$, let $F^{cs}_n=G^s_n$ and $F^{cu}_n=G^c_n\oplus G^u_n$; if $s+c=\ind(\si)-1$, let $F^{cs}_n=G^s_n\oplus G^c_n$ and $F^{cu}_n=G^u_n$. By Lemma \ref{lem:BasicAwayTang}, the splitting $\cN_{\ga_n}=F^{cs}_n\oplus F^{cu}_n$ is dominated (uniformly in $n$) with domination constant $T$, hence induces a dominated splitting to its limit (Remark \ref{rmk:fundamental-sequences}): there is a dominated splitting $\wt{\cN}_{\De(\fF)}=G^{cs}\oplus G^{cu}$ with $\dim G^{cs}=\ind(\si)-1$.
\end{proof}

We continue the proof of Lemma \ref{lem:lorenz-like}.
Define $\De^u_{\si}=\{L\in\De(\fF):L\subset E^u(\si)\}$. Then $\De^u_{\si}$ is a nonempty compact $\hat{\Phi}^X_t$-invariant set. Restricted to $\De^u_{\si}$, the dominated splitting $\widetilde{\cN}_{\De}=G^{cs}\oplus G^{cu}$ gives a dominated splitting $\widetilde{\cN}_{\De^u_{\si}}=G^{cs}\oplus G^{cu}$ such that $\dim{G^{cs}}=\ind(\si)-1$. Then by Lemma \ref{lem:elp-dom}, there is a dominated splitting $E^s(\si)=E^{ss}(\si)\oplus E^{c}(\si)$ with $\dim{E^{c}(\si)}=1$.

Let $W^{ss}(\si)$ be the strong stable manifold tangent to $E^{ss}(\si)$. We prove $W^{ss}(\si)\cap\Ga=\emptyset$ by contradiction. Suppose $W^{ss}(\si)\cap\Ga\neq\emptyset$, i.e., $\Ga\subset W^{ss}(\si)$. With possibly a small perturbation, we can assume that $X$ is smooth enough and is $C^1$-linearizable around $\si$. Then fixing a linearizable neighborhood $U_{\vep}(\si)$ of $\si$, we can assume that $X$ is linear in this neighborhood and hence $W^{ss}_{\vep}(\si),W^u_{\vep}(\si)$ coincide with $E^{ss}(\si)$ and $E^u(\si)$ respectively. Fix $x'\in W^{ss}_{\vep}(\si)\cap\Ga$, $y'\in W^u_{\vep}(\si)\cap\Ga$. Let $x_n\to x'$, $y_n\to y'$ be two sequences of points such that for some $t_n>0$, $\phi_{t_n}(x_n)=y_n$ and the orbit segments $\{\phi_t(x_n):~0\leq t\leq t_n\}\subset W^{ss}_{\vep}(\si)\oplus W^u_{\vep}(\si)$. Connecting $x_n$ to $x'$ and $y_n$ to $y'$ by small pushing in neighborhoods of $x'$ and $y'$ respectively, we obtain a fundamental sequence $(\ga'_n,X'_n)\to (\Cl(\Ga),X)$ with $x_n\in\ga'_n$. Choose $p_n\in\{\phi_t(x_n):~0\leq t\leq t_n\}$ such that $p_n\to\si$. By construction, the limit of any convergent subsequence of $\{\langle X'_n(p_n)\rangle\}$ is contained in $E^{ss}(\si)\oplus E^u(\si)$. Since $X$ is linear in $U_{\vep}(\si)$, choose carefully the sequence $\{p_n\}$ and a convergent subsequence $\{p_{n_k}\}$, we can assume that $\langle X'_{n_k}(p_{n_k})\rangle\to L\in(E^{ss}(\si)\oplus E^u(\si))\setminus(E^{ss}(\si)\cup E^u(\si))$. But this will lead to a contradiction since one can make exactly the same analysis for $(\ga'_n,X'_n)_{n\in\NN}$ as for $\{(\ga_n,X_n)\}$, and find that Lemma \ref{lem:elp-dom2} holds.
\end{proof}

\begin{Corollary}\label{cor:lorenz-like}
For any $X\in\mathscr{X}^1(M)\setminus\Cl(\mathcal{HT})$, suppose there is a homoclinic orbit $\Ga$ related to a hyperbolic singularity $\si\in\Sing(X)$ and $\sv(\si)>0$. Then $\Cl(\Ga)=\Ga\cup\{\si\}$ is singular-hyperbolic: there is a dominated splitting $T_{\Cl(\Ga)}M=E^{ss}\oplus E^{cu}$, $\dim E^{ss}=\ind(\si)-1$, such that $E^{ss}$ is contracting and $E^{cu}$ is sectionally expanding. A symmetric result holds for the case $\sv(\si)<0$.
\end{Corollary}
\begin{proof}
As in the proof of Lemma \ref{lem:lorenz-like}, there is a fundamental sequence $\fF=\{(\ga_n,X_n)\}$ such that $(\ga_n,X_n)\to (\Cl(\Ga),X)$ and  it admits a dominated splitting $\wt{\cN}_{\De(\fF)}=G^{cs}\oplus G^{cu}$ of index $(\ind(\si)-1)$. Moreover, there is a partially hyperbolic splitting $T_{\si}M=E^{ss}(\si)\oplus E^{cu}(\si)$ such that $\dim E^{ss}(\si)=\ind(\si)-1$ and $W^{ss}(\si)\cap \Cl(\Ga)=\{\si\}$, which implies $\De_{\si}(\fF) \subset E^{cu}(\si)$.

By changing the Riemannian metric around $\si$, we may assume that $E^{ss}(\si)\perp E^{cu}(\si)$. By Lemma \ref{lem:elp-dom}, for any $L\in\De_{\si}(\fF)$, one has $G^{cs}(L)=\{L\}\times E^{ss}(\si)$. Then as the Dirac measure $\de_{\si}$ is the only ergodic measure supported on $\Cl(\Ga)$, this implies that $(G^{cs},\tilde{\psi}_t)$ is dominated by $(\cP,\wt{\Phi}_t)$ on $\De(\fF)$, i.e., there exists $T>0$ such that
\[\|\tilde{\psi}_t|_{G^{cs}}\|\cdot m(\wt{\Phi}_{t}|_{\cP})^{-1}\leq\frac{1}{2},\quad \forall t\geq T.\]
Then a standard argument as in \cite[Lemma 5.5, 5.6]{LGW05} gives a dominated splitting $T_{\Cl(\Ga)}M = E^{ss}\oplus E^{cu}$, $\dim E^{ss}=\ind(\si)-1$. Finally, since the Dirac measure $\de_{\si}$ is the only ergodic measure supported on $\Cl(\Ga)$ and $\sv(\si)>0$, the splitting $T_{\Cl(\Ga)}M = E^{ss}\oplus E^{cu}$ is singular-hyperbolic.
\end{proof}

As a direct consequence of Corollary \ref{cor:lorenz-like}, one has the following result.

\begin{Corollary}\label{cor:lorenz-like2}
For any $X\in\mathscr{X}^1(M)\setminus\Cl(\mathcal{HT})$, suppose there is a homoclinic orbit $\Ga$ related to a hyperbolic singularity $\si\in\Sing(X)$ and $\sv(\si)>0$. Then for any fundamental sequence $(\ga_n,X_n)\to (\Cl(\Ga), X)$, $\ga_n$ is a hyperbolic periodic orbit of index $(\ind(\si)-1)$ for $n$ large enough. Moreover, these periodic orbits admit a partially hyperbolic splitting for the linear Poincar\'e flow, $\cN_{\ga_n}=N^{ss}_n\oplus N^{cu}_n$, such that $N^{ss}_n$ is contracting, and $N^{cu}_n$ is expanding at period, uniformly in $n$. A symmetric result holds for the case $\sv(\si)<0$.
\end{Corollary}

Now let us prove Theorem \ref{thm:lorenz-like}.

\begin{proof}[Proof of Theorem \ref{thm:lorenz-like}.]
Consider the case $\sv(\si)>0$. Since $X$ is generic, by item 5 of Lemma \ref{lem:gen_0} we can assume that $X$ is a continuity point of $C(\si,X)$.
Since $C(\si,X)$ is nontrivial, by item 3 of Lemma \ref{lem:gen_0} we can assume that there is a sequence of periodic orbits converging to $C(\si,X)$ in the Hausdorff topology. Applying the connecting lemma of Wen-Xia \cite{WeX}, there is $Y$ arbitrarily close to $X$ and a neighborhood $V$ of $\si$ such that $Y=X$ on $V$ and $Y$ has a homoclinic orbit $\Ga$ related to $\si$. We may assume that $Y$ is away from homoclinic tangencies since $X$ is away from homoclinic tangencies. Perturbing $Y$, we obtain a fundamental sequence $(\ga_n,Y_n)\to(\Cl(\Ga),Y)$.

Suppose on the contrary that $\ind(\si)=1$. By Corollary \ref{cor:lorenz-like2}, the orbit $\ga_n$ is a periodic source for $n$ large enough. Then one can apply Lemma \ref{lem:pliss} and Corollary \ref{cor:Liao} to obtain a transverse intersection between $W^s(\si,Y_n)$ and $W^u(\ga_n,Y_n)$. Since this intersection is transverse, we can further assume that $Y_n$ is generic. Then the first item of Lemma \ref{lem:gen_1} (applied to $-Y_n$) implies that $C(\si,Y_n)=\{\si\}$. This is a contradiction to the fact that $C(\si,X)$ varies continuously at $X$, since $Y_n\to Y$ and $Y$ can be chosen arbitrarily close to $X$. Therefore, we have $\ind(\si)>1$.

By Lemma \ref{lem:lorenz-like}, there is a dominated splitting $E^s(\si,Y)=E^{ss}(\si,Y)\oplus E^{c}(\si,Y)$, where $\dim E^{c}(\si,Y)=1$, and $W^{ss}(\si,Y)\cap\Ga=\emptyset$. Since $X=Y$ in $V$, we have the same dominated splitting $E^s(\si,X)=E^{ss}(\si,X)\oplus E^{c}(\si,X)$ for $X$. Moreover, one must have $W^{ss}(\si,X)\cap C(\si,X)\setminus\{\si\}=\emptyset$, since otherwise the homoclinic orbit $\Ga$ of $Y$ can be obtained such that $\Ga\cap W^{ss}(\si,Y)\neq\emptyset$ in the first place (see the proof of Theorem D of \cite{WeX} for more details).

Hence we have shown that $\si$ is Lorenz-like and $\ind(\si)>1$ if $\sv(\si)>0$. A symmetric argument as above gives a proof for the case $\sv(\si)<0$.
\end{proof}

\subsection{Lyapunov stable chain recurrence classes}

In this section, we will further assume that $C(\sigma)$ is Lyapunov stable and obtain results on the indices of critical elements contained in the class.
\subsubsection{Indices of singularities}

\begin{Proposition}\label{prop:homogeneous-index}
Given a $C^1$ generic $X\in\mathscr{X}^1(M)\setminus\Cl(\mathcal{HT})$ and $\si\in\Sing(X)$, if $C(\si)$ is a nontrivial Lyapunov stable chain recurrence class, then for any $\rho\in\Sing(X)\cap C(\si)$, it satisfies that $\rho$ is Lorenz-like and $\ind(\rho)=\ind(\si)$.
\end{Proposition}

\begin{proof}
Let $\rho$ be any singularity in $C(\si)$. Since $C(\si)$ is nontrivial, it follows from Theorem \ref{thm:lorenz-like} that the singularity $\rho$ is either Lorenz-like or reversed Lorenz-like. As $C(\si)$ is Lyapunov stable, it contains the whole unstable manifold of $\rho$. This implies that $\rho$ is Lorenz-like, since a reversed Lorenz-like singularity does not contain its strong unstable manifold (see Definition \ref{def:lorenz-like}).
In the following we prove by contradiction that $\ind(\rho)=\ind(\si)$. Without loss of generality, let us suppose on the contrary that $\ind(\rho)>\ind(\si)$.

A topological cycle $\Ga$ on the manifold is called a {\em heteroclinic cycle} related to $\si$ and $\rho$ if $\si,\rho\in\Ga$ and $\Ga\setminus\{\si,\rho\}$ is composed of two regular orbits $\Upsilon_1$ and $\Upsilon_2$, with $\Upsilon_1\subset W^u(\si)\cap W^s(\rho)$ and $\Upsilon_2\subset W^u(\rho)\cap W^s(\si)$.
Following the argument in \cite[Sublemma 4.6]{SGW}, there is a vector field $Y$ arbitrarily $C^1$ close to $X$ with a heteroclinic cycle $\Ga$ related to $\si_Y$ and $\rho_Y$.
Then by Corollary \ref{cor:lorenz-like2} and following the argument in \cite[Lemma 4.5]{SGW}, there is a fundamental sequence $(\ga_n, Y_n)\to (\Ga, Y)$ with dominated splitting $\cN_{\ga_n}=G^s_n\oplus G^{cu}_n$, and $\dim G^{s}_n=\ind(\rho)-1\geq\ind(\si)$.
As a consequence of Lemma~\ref{lem:elp-dom}, there is a dominated splitting $E^u(\si)=E^c(\si)\oplus E^{uu}(\si)$ such that $\dim E^c(\si)=\ind(\rho)-\ind(\si)$.
Furthermore, denoting $\fF=(\ga_n,Y_n)_{n\in\NN}$,  Lemma~\ref{lem:elp-dom2} gives
\begin{equation}\label{e.p.4.7}
	\De(\fF)\cap (E^{s}(\si)\oplus E^{uu}(\si))\setminus (E^{s}(\si)\cup E^{uu}(\si))=\emptyset.
\end{equation}

On the other hand, since $C(\si)$ is Lyapunov stable, $W^{uu}(\si)\subset C(\si)$. Therefore, by connecting the strong unstable manifold of $\si$ to the stable manifold of $\rho$ in the first place when applying the argument in \cite[Sublemma 4.6]{SGW}, we can assume that the heteroclinic cycle $\Ga$ has a part $\Upsilon_1$ contained in the strong unstable manifold $W^{uu}(\si,Y)$. Then as in the proof of \cite[Lemma 4.5]{SGW}, a fundamental sequence $(\ga_n,Y_n)\to (\Ga,Y)$ can be chosen such that $\De(\fF)\cap (E^{s}(\si)\oplus E^{uu}(\si))\setminus(E^{s}(\si)\cup E^{uu}(\si))\neq\emptyset$.
This is, however, a contradiction to \eqref{e.p.4.7}.
\end{proof}
\subsubsection{Indices of periodic orbits}

\begin{Proposition}\label{prop:indices-of-periodic-orbits-0}
For any $C^1$ generic $X\in\mathscr{X}^1(M)\setminus\Cl(\mathcal{HT})$, $\si\in\Sing(X)$, suppose $C(\si)$ is a Lyapunov stable chain recurrence class and $\ga$ is a periodic orbit contained in $C(\si)$, then $\ind(\ga)\leq\ind(\si)-1$.
\end{Proposition}

\begin{proof}
	Suppose on the contrary that there exists a periodic orbit $\ga\subset C(\si)$ with $\ind(\ga)\geq\ind(\si)$. By genericity, $C(\si)$ is the homoclinic class of $\ga$ (item \ref{item:gen_0-homoclinic-class} of Lemma \ref{lem:gen_0}). From the main result of \cite{ABCDW}, the set of periodic orbits with index equal to $\ind(\ga)$ is dense in $C(\si)$. Let us denote
	\[B(C(\si))=\{L\in G^1: \exists x_n\in C(\si), \txtrm[such that] \langle X(x_n)\rangle\to L\}.\]
	By Corollary \ref{cor:BasicAwayTang} and Remark \ref{rmk:fundamental-sequences}, there is a dominated splitting $\wt{\cN}_{B(C(\si))}=G^{cs}\oplus G^{cu}$ with respect to the extended linear Poincare flow, such that $\dim G^{cs}=\ind(\ga)$.
	By Lemma \ref{lem:matching}, there is a nontrivial dominated splitting of the unstable subspace $E^u(\si)=E^c(\si)\oplus E^{uu}(\si)$. Furthermore, one has
	\[B(C(\si))\cap (E^s(\si)\oplus E^{uu}(\si))\setminus(E^s(\si)\cup E^{uu}(\si))= \emptyset.\]
	On the other hand, since $C(\si)$ is Lyapunov stable, item 3 of Lemma \ref{lem:gen_1} ensures that
	\[B(C(\si))\cap (E^s(\si)\oplus E^{uu}(\si))\setminus(E^s(\si)\cup E^{uu}(\si))\neq \emptyset.\]
	A contradiction.
\end{proof}

As a concluding remark, Theorem C is now proven as it is contained in Theorem \ref{thm:lorenz-like} and Proposition \ref{prop:homogeneous-index}. Also, we have proven the first two items of Theorem B in Proposition \ref{prop:homogeneous-index} and Proposition \ref{prop:indices-of-periodic-orbits-0}.

\section{Non-singular chain transitive set}
\label{sect:non-singular-set}

This section contributes mainly to Proposition \ref{prop:non-singular-set}, which is a special case of Theorem A. This section also contains the proof of Theorem~B and it ends with a proof of Corollary~\ref{cor:index-completeness} assuming that Theorem~A holds. The proof of Theorem~A occupies Section~6 and onward.

\begin{Proposition}\label{prop:non-singular-set}
For a $C^1$ generic vector field $X\in\xX^1(M)\setminus\Cl(\mathcal{HT})$, let $\La$ be a Lyapunov stable chain recurrence class of $X$. If $\La$ contains a non-singular compact invariant subset $\La_0$, then $\La$ is a homoclinic class.
\end{Proposition}

\begin{Remark}
	In \cite{Yan07} it was proven that any Lyapunov stable chain recurrence class of a $C^1$ generic diffeomorphism away from homoclinic tangencies is a homoclinic class.  Proposition~\ref{prop:non-singular-set} can be seen as a generalization of this result to non-singular flows; however, the proof requires significant modification comparing to~\citealp{Yan07} as we will see in Remark~\ref{r.5.1}.
\end{Remark}

The proof of Proposition \ref{prop:non-singular-set} will make use of the central model of Crovisier \cite{Cro10}. Central model for non-singular flows has been studied in \cite{XZ}, in their proof of the $C^1$ Palis conjecture for non-singular flows.

The following lemma acts as a starting point for building a central model. Recall that an invariant set $\Gamma$ is said to be minimally non-hyperbolic, if $\Gamma$ is non-hyperbolic and yet every invariant subset of $\Gamma$ is hyperbolic.
\begin{Lemma}\label{lem:non-singular-set}
Suppose that the assumptions of Proposition \ref{prop:non-singular-set} hold, and let $\Lambda_0$ be a non-singular compact invariant subset of $\Lambda$. Then one has
\begin{itemize}
\item either $\La_0$ intersects a nontrivial homoclinic class, which is $\La$ in this case,
\item or $\La_0$ contains a minimally non-hyperbolic (non-singular) set $K_0$ that admits a partially hyperbolic splitting $\cN_{K_0}=G^{ss}\oplus G^c\oplus G^{uu}$ with $\dim G^c=1$.
\end{itemize}
\end{Lemma}
\begin{proof}
This result is well-known, so we only sketch the proof here as one can refer to \cite[Corollary 4.4]{Cro10} or \cite[Theorem 2.1]{XZ} for details. Assume that the first item is not satisfied, then $\La_0$ can not be hyperbolic. By Zorn's lemma, $\La_0$ contains a transitive minimally non-hyperbolic set $K_0$. By the connecting lemma of pseudo-orbits and genericity, $K_0$ can be accumulated by periodic orbits (item \ref{item:gen_0-chain-transitive-set} of Lemma \ref{lem:gen_0}).
Since $X$ is away from homoclinic tangencies, there is a fundamental sequence $\ga_n$ that admits a dominated splitting $\cN_{\ga_n}=G^{cs}_n\oplus G^{cu}_n$ satisfying Corollary \ref{cor:BasicAwayTang} and further induces a dominated splitting of the normal bundle over $K_0$, namely $\cN_{K_0}=G^{cs}\oplus G^{cu}$. Liao's selecting lemma \cite{Lia81, Wen08} can be used to show that either $K_0$ intersects a nontrivial homoclinic class, or one of the subbundles is uniform, see \cite[Proposition 4.3]{Cro10}. The first case can not happen here as we have assumed that $\La_0\supset K_0$ does not intersects a nontrivial homoclinic class. We thus assume without loss of generality that $G^{cu}$ is uniformly expanding. Consider the fundamental sequence $\ga_n$, by Corollary \ref{cor:BasicAwayTang} there are two cases for the subbundle $G^{cs}_n$: either $G^{cs}_n$ is contracting at period, or there is a further splitting $G^{cs}_n=G^{ss}_n\oplus G^{c}_n$ with $\dim G^{c}=1$. If the first case holds, one can apply Pliss' lemma (to $G^{cs}_n$) to show that $K_0$ intersects a nontrivial homoclinic class, obtaining a contradiction. So we are left with the second case, in which one obtains a dominated splitting $\cN_{K_0}=G^{ss}\oplus (G^c\oplus G^{cu})$ by taking limits of the splitting $\cN_{\ga_n}=G^{ss}_n\oplus (G^c_n\oplus G^{cu}_n)$. If $(G^c\oplus G^{cu})$ is uniformly expanding, then one can apply again Pliss's lemma (to $G^{ss}_n$) to show that $K_0$ intersects a nontrivial homoclinic class, a contradiction. Therefore $G^c\oplus G^{cu}$ is not uniformly expanding. As $K_0$ cannot intersects a nontrivial homoclinic class, one uses Liao's selecting lemma again to show that $G^{ss}$ is uniformly contracting. This gives the second item of the lemma.
\end{proof}

\begin{Remark}\label{r.5.1}
	In \cite{Yan07}, for building a central model the author started by choosing a minimally non-hyperbolic set with a further global condition: it is the limit of a fundamental sequence with the smallest possible index. This global condition is due to the observation that any fundamental sequence with limit in the chain recurrence class has a uniform contracting stable bundle of dimension at least as large as the minimal possible index. This condition was proven to be convenient.
	
	However, this global condition will not be possible for general singular flows as such a set may contain singularities. In our case, one could try to use a local version of the aforementioned condition by requiring that the fundamental sequences to be considered all have their limit in $\La_0$. But then this will not be of much help for our proof.
\end{Remark}

\subsection{Central model}

Recall that a {\em central model} is a pair $(\hat{K},\hat{f})$, where $\hat{K}$ is a compact metric space, $\hat{f}$ is a continuous map from $\hat{K}\times[0,1]$ to $\hat{K}\times[0,+\infty)$, such that
\begin{itemize}
	\item $\hat{f}(\hat{K}\times\{0\})=\hat{K}\times\{0\}$;
	\item $\hat{f}$ is a local homeomorphism in a small neighborhood of $\hat{K}\times\{0\}$;
	\item $\hat{f}$ is a skew product, {\itshape i.e.} there exist $\hat{f}_1:\hat{K}\to\hat{K}$ and $\hat{f}_2:\hat{K}\times[0,1]\to[0,+\infty)$ such that
	\[\hat{f}(x,t)=(\hat{f}_1(x),\hat{f}_2(x,t)),\quad \forall (x,t)\in\hat{K}\times[0,1].\]
\end{itemize}

As $\hat{f}$ is a skew product, the set $\hat{K}$ (naturally identified with $\hat{K}\times\{0\}$) is called the {\em base} of the central model. The central model $(\hat{K},\hat{f})$ has a {\em chain-recurrent central segment} if it contains a nontrivial segment $I=\{x\}\times[0,a]$ that is contained in a chain-transitive set of $\hat{f}$. A chain-recurrent central segment is contained in the maximal invariant set in $\hat{K}\times[0,1]$. A {\em trapping strip} for $\hat{f}$ is an open set $S\subset \hat{K}\times [0,1]$ satisfying $\hat{f}(\Cl(S))\subset S$, and moreover, $S\cap(\{x\}\times[0,+\infty))$ is a nontrivial interval containing $(x,0)$ for any $x\in \hat{K}$. As $\hat{f}$ is a local homeomorphism, its inverse $\hat{f}^{-1}$ is well-defined in a neighborhood of $\hat{K}\times\{0\}$. A trapping strip for $\hat{f}^{-1}$ is called an {\em expanding strip} for $\hat{f}$. Accordingly, we say that $(\hat{K},\hat{f})$ has a trapping (resp. expanding) strip $S$ if the set $S$ is a trapping (resp. expanding) strip for $\hat{f}$. The following lemma establishes a dichotomy between existence of chain-recurrent central segments and existence of arbitrarily small trapping (or expanding) strips.

\begin{Lemma}[{\cite[Proposition 2.5]{Cro10}}]\label{lem:central-model}
	Let $(\hat{K},\hat{f})$ be a central model with a chain-transitive base, then
	\begin{itemize}
		\item either there exists a chain-recurrent central segment,
		\item or there exist some trapping strips $S$ in $\hat{K}\times[0,1]$ either for $\hat{f}$ or for $\hat{f}^{-1}$ in arbitrarily small neighborhoods of $\hat{K}\times\{0\}$.
	\end{itemize}
\end{Lemma}

Given a non-singular compact $\phi_t$-invariant set $K$ admitting a partially hyperbolic splitting $\cN_K=G^{s}\oplus G^c\oplus G^{u}$ with respect to the linear Poincar\'e flow, with center dimension $\dim G^c=1$, let $T>0$ be a domination constant of the splitting and $P_T$ the local holonomy map defined by $P_T|_{N_x(r_x)}=P_{x,\phi_T(x)}$ for each $x\in M\setminus\Sing(X)$ (Section \ref{sect:sectional-Poinc-flow}). Note that since $K$ is non-singular and compact, there is $\de_0>0$ such that $P_T$ is well-defined for all $y\in N_x(\de_0)$ with $x\in K$. A central model $(\hat{K},\hat{f})$ is called a {\em central model for $(K,P_T)$} if there exists a continuous map $\pi:\hat{K}\times[0,+\infty)\to M$ such that
\begin{itemize}
	\item $\pi$ semi-conjugates $\hat{f}$ and $P_T$: $P_T\circ\pi=\pi\circ\hat{f}$ on $\hat{K}\times[0,1]$;
	\item $\pi(\hat{K}\times\{0\})=K$;
	\item the collection of maps $t\mapsto \pi(\hat{x},t)$ is a continuous family of $C^1$-embeddings of $[0,+\infty)$ into $M$, parameterized by $\hat{x}\in\hat{K}$;
	\item for any $\hat{x}\in\hat{K}$, the curve $\pi(\hat{x},[0,+\infty))$ is tangent to $G^c(x)$ at the point $x=\pi(\hat{x},0)$, and $\pi(\hat{x},[0,+\infty))\subset N_x(\de_0)$.
\end{itemize}
We sometimes write $(\hat{K},\hat{f};\pi)$ as the central model for $(K,P_T)$, to emphasize the role played by the map $\pi$.
By rescaling, one can require that a central model for $(K,P_T)$ is ``small'' in the sense that for any prefixed constant $\de\in(0,\de_0]$ and small cone field $Cone(K)$ of the center bundle $G^c$, for any $\hat{x}\in\hat{K}$, one has $\pi(\hat{x},[0,1])\subset N_{\pi(\hat{x})}(\de)$, and it is tangent to $Cone(K)$.

The construction of central models for $(K,P_T)$ depends on whether $P_T$ preserves an orientation of the center bundle $G^c$ or not. One needs to consider the following two cases:
\begin{itemize}
	\item The {\em orientable} case: the bundle $G^c$ is orientable; 
	\item The {\em non-orientable} case: the bundle $G^c$ is not orientable;
\end{itemize}
Note that in the orientable case, $P_T$ preserves an orientation of $G^c$ as it is homotopic to the identity.

\begin{Theorem}[{\cite[Section 3.2]{Cro10}, \cite[Section 5.2]{XZ}}]\label{thm:central-model}
	Suppose there is a non-singular compact $\phi_t$-invariant set $K$ admitting a partially hyperbolic splitting $\cN_K=G^{ss}\oplus G^c\oplus G^{uu}$ with respect to the linear Poincar\'e flow such that $\dim G^c=1$. Let $T>0$ be the domination constant for the splitting, then there exist a central model for $(K,P_T)$. Moreover,
	\begin{itemize}
		\item In the orientable case, one can obtain two central models for $(K, P_T)$, which are denoted by $(\hat{K}_+,\hat{f}_+;\pi_+)$ and $(\hat{K}_-,\hat{f}_-;\pi_-)$, such that
		\begin{itemize}
			\item the maps $\pi_{\iota}:\hat{K}_{\iota}\times \{0\}\to K$ are both homeomorphisms, for $\iota=+,-$;
			\item for any point $x\in K$, $\iota\in\{+,-\}$, let $\hat{x}^{\iota}=\pi^{-1}_{\iota}|_{\hat{K}_{\iota}\times \{0\}}(x)$, then $x$ is contained in (the interior of) a central curve which is the union of the two half central curves $\pi_{\iota}(\hat{x}^{\iota},[0,+\infty))$, for $\iota=+,-$.
		\end{itemize}
		
		\item In the non-orientable case, the central model $(\hat{K},\hat{f};\pi)$ for $(K,P_T)$ satisfies the following:
		\begin{itemize}
			\item $\pi:\hat{K}\times\{0\}\to K$ is two-to-one: any point $x\in K$ has exactly two preimages $\hat{x}^+$ and $\hat{x}^-$ in $\hat{K}$;
			\item any point $x\in K$ is contained in (the interior of) a central curve which is the union of the two half central curves $\pi(\hat{x}^{\iota},[0,+\infty))$, for $\iota=+,-$.
		\end{itemize}
		
		\item If $K$ is chain-transitive (resp. minimal), then in both cases and for any of the central models, the base is also chain-transitive (resp. minimal).
	\end{itemize}
\end{Theorem}

\subsection{An intersection lemma}

The following lemma plays a central role in the proof of Proposition \ref{prop:non-singular-set}.

\begin{Lemma}\label{lem:intersection}
	Let $\{\ga_n\}$ be a sequence of periodic orbits of $X\in\xX^1(M)$, with Hausdorff limit $\Ga$
	satisfying the following properties:
\begin{enumerate}
	\item[(1)] All singularities in $\Ga$ are Lorenz-like;
	\item[(2)] There exist constants $C,\eta,T>0$ and a $T$-dominated splitting $\cN_{\ga_n}=G^{cs}_n\oplus G^{cu}_n$ such that $\ga_n$ is $(C,\eta,T,G^{cs}_n)$-$\psi_t$-contracting at period for all $n$.
\end{enumerate}	
	Then either $G^{cs}_n$ is uniformly contracting with constants uniform in $n$, or one has $W^u(\Ga)\pitchfork W^s(\ga_n)\neq\emptyset$, for $n$ large.
\end{Lemma}
\begin{proof}
	Let $\cN_{\Ga}=G^{cs}\oplus G^{cu}$ be the dominated splitting induced by the dominated splitting $\cN_{\ga_n}=G^{cs}_n\oplus G^{cu}_n$, see Remark \ref{rmk:fundamental-sequences}. Denote $i$ as the index of the dominated splitting, {\itshape i.e.} $i=\dim G^{cs}_n$.

	Let $0<\eta'\leq\min\{\eta/2, \frac{1}{T}\log\frac{3}{2}\}$, and $H_n(\eta')$ be the set of $(1,\eta',T,G^{cs}_n)$-$\psi^{*}_t$-contracting points on $\ga_n$. By assumption (2) and Lemma \ref{lem:pliss}, $H_n(\eta')$ is not empty for every $n$ large. Now suppose that $G^{cs}_n$ is not uniformly contracting, or that the constants are not uniform in $n$. In both cases, there exist $x_n\in H_n(\eta')$ and $l_n\in\NN$ for $n$ large enough, such that the following properties hold:
	\begin{itemize}
		\item $\phi_{-iT}(x_n)\notin H_n(\eta')$, for all $i=1,2,\ldots,l_n$;
		\item $l_n\to\infty$ as $n\to\infty$.
	\end{itemize}
	Since $x_n\in H_n(\eta')$, the first item implies that
	\begin{equation}\label{eqn:intersection-1}
		\prod_{i=1}^{j}\|\psi^{*}_T|_{G^{cs}_n(\phi_{-iT}(x_n))}\|>\e^{-j\eta'T},\quad \forall 1\leq j\leq l_n.
	\end{equation}
	Since $G^{cs}_n\oplus G^{cu}_n$ is $T$-dominated, one has
	\begin{equation}\label{eqn:intersection-2}
		\prod_{i=0}^{j-1}\|\psi^{*}_{-T}|_{G^{cu}_n(\phi_{-iT}(x_n))}\|\leq \e^{j\eta'T}/2^j\leq \left(\frac{3}{4}\right)^j,\quad\forall 0\leq j\leq l_n-1.
	\end{equation}
	By choosing a subsequence, we can assume that $x_n\to x\in \Ga$.
		
	If $x$ is a regular point, by inequality \eqref{eqn:intersection-2}, $x$ is $(1,\eta'',T,G^{cu})$-$\psi^*_t$-expanding, where $\eta''=(\log 4-\log 3)/T$. By Lemma \ref{lem:stable-manifolds}, there is $\de>0$ such that $W^{cu}_{\de\|X(x)\|}(x)\subset W^u(\orb(x))$. Again, by Lemma \ref{lem:stable-manifolds}, since $x_n\in H_n(\eta')$, there is $\de'>0$ such that $W^{cs}_{n,\de'\|X(x_n)\|}\subset W^s(\ga_n)$. As $\cN_{\ga_n}=G^{cs}_n\oplus G^{cu}_n$ converges to $\cN_{\Ga}=G^{cs}\oplus G^{cu}$, $W^{cu}_{\de\|X(x)\|}(x)\pitchfork W^{cs}_{n,\de'\|X(x_n)\|}\neq\emptyset$, i.e., $W^u(\orb(x))\pitchfork W^s(\ga_n)\neq\emptyset$, for $n$ large enough.

	If $x$ is a singularity, then by assumption (1) it is Lorenz-like. We have the following claim.
	\begin{Claim}
		If $\eta'>0$ is small enough, then $\ind(x)\leq i$.
	\end{Claim}
	\begin{proof}[Proof of the claim]
		To prove the claim, suppose on the contrary that $\ind(x)>i$.  By Lemma \ref{lem:elp-dom} there is a dominated splitting of the stable subspace $E^s(x)=E^{ss}(x)\oplus E^c(x)$ such that $\dim E^{ss}(x)=i$. By changing the Riemannian metric around these singularities, we can assume that $E^{ss}(x)\perp E^{cu}(x)$ for any such $\si$, where $E^{cu}(x)=E^c(x)\oplus E^u(x)$. By domination, there exists $\eta_0>0$ such that
		\begin{equation}\label{eqn:intersection-3}
			\|\Phi_T|_{E^{ss}(x)}\|\cdot m(\Phi_T|_{E^{cu}(x)})^{-1}\leq e^{-\eta_0 T},
		\end{equation}
		where $m(\cdot)$ stands for the minimal norm. Since there are only finitely many singularities in $\Ga$, we can assume that the constant $\eta_0$ is uniform for all singularity $\si\in\Ga$ with $\ind(\si)>i$.
		
		On the other hand, note that $\Ga$ is a chain transitive set; consequently, we have $\Ga\subset C(x)$. It follows that $\Ga\cap W^{ss}(x)\setminus\{x\}=\emptyset$ as $x$ is Lorenz-like. Applying the argument as in the proof of \cite[Lemma 4.4]{LGW05}, one can show that $\De_{x}\subset E^{cu}(x)$, where $E^{cu}(x)=E^c(x)\oplus E^u(x)$ and
		\[\De_x=\{L\in G^1\cap \be^{-1}(x):\ \exists x_n\in\ga_n, \txtrm[such that] \langle X_n(x_n)\rangle\to L\}.\]
		Then by Lemma \ref{lem:elp-dom}, for any $L\in \De_x$,  one has $G^{cs}(L)=\{L\}\times E^{ss}(x)$. This implies that there is a neighborhood $U$ of the singularity $x$, such that for any point $z_n\in\ga_n\cap U$, $\langle X(z_n)\rangle$ is so close to $E^{cu}(x)$ and $G^{cs}_n(z_n)$ is so close to $E^{ss}(x)$ that by~\eqref{eqn:intersection-3},
		\[\|\psi^*_T|_{G^{cs}_n(z_n)}\|
		\leq e^{\eta_0T/2}\cdot\|\Phi_T|_{E^{ss}(x)}\|\cdot m(\Phi_T|_{E^{cu}(x)})^{-1}
		\leq e^{-\eta_0T/2}.\]
		If $\eta'\leq\eta_0/2$, we obtain a contradiction to inequality \eqref{eqn:intersection-1}.
	\end{proof}
	
	Since $\eta'>0$ can be arbitrarily small, we can assume that $\eta'$ has been chosen small enough such that by the previous claim it holds $\ind(x)\leq i$. Then one can apply Theorem \ref{thm:Liao} to show that $W^u(x)\pitchfork W^s(\ga_n)\neq\emptyset$ for $n$ large. The proof is now complete.
\end{proof}

\subsection{Proof of Proposition \ref{prop:non-singular-set}}

Suppose that $\La$ is not a homoclinic class, which is equivalent to being aperiodic by genericity (item \ref{item:gen_0-homoclinic-class} of Lemma \ref{lem:gen_0}). By Lemma \ref{lem:non-singular-set}, there is a minimally non-hyperbolic set $K_0\subset\La_0$ that admits a partially hyperbolic splitting $\cN_{K_0}=G^{ss}\oplus G^c\oplus G^{uu}$ with $\dim G^c=1$. Let $T>0$ be the domination constant for the splitting, and $P_T$ the local holonomy map between normal manifolds $N_x(\de_0)$ for $x\in K_0$. It is easy to see that $K_0$ is chain-transitive, and in fact, is transitive. We remark that the domination constant $T$ can be fixed by Lemma \ref{lem:BasicAwayTang}, since by genericity $K_0$ is accumulated by periodic orbits (item \ref{item:gen_0-chain-transitive-set} of Lemma \ref{lem:gen_0}).

Let us first consider the orientable case.

One fixes a small neighborhood $U$ of $K_0$ so that the maximal invariant set in $U$, which we denote by $K$, also admits a partially hyperbolic splitting $\cN_{K}=G^{ss}\oplus G^c\oplus G^{uu}$ with one-dimensional center and the center bundle $G^c$ has an orientation preserved by $P_T$. By Lemma \ref{lem:gen_0} (item \ref{item:gen_0-chain-transitive-set}), $K$ contains periodic orbits arbitrarily close to $K_0$. By Theorem \ref{thm:central-model}, one can obtain two central models $(\hat{K}_+,\hat{f}_+;\pi_+)$ and $(\hat{K}_-,\hat{f}_-;\pi_-)$ for $(K,P_T)$, with properties listed there. By considering the subset of the base $\hat{K}_+$ that projects by $\pi_+$ to $K_0$, one obtains a central model $(\hat{K}_{0,+},\hat{f}_+;\pi_+)$ for $(K_0,P_T)$. Similarly, one can also obtain a central model $(\hat{K}_{0,-},\hat{f}_-;\pi_-)$ for $(K_0,P_T)$. By Theorem \ref{thm:central-model}, since $K_0$ is chain-transitive, the central models $(\hat{K}_{0,\iota},\hat{f}_{\iota};\pi_{\iota})$ ($\iota\in\{+,-\}$) are both chain-transitive.

Then according to Lemma \ref{lem:central-model}, we have the following subcases regarding the two central models $(\hat{K}_{0,\iota},\hat{f}_{\iota};\pi_{\iota})$ ($\iota\in\{+,-\}$):
\begin{enumerate}[(1)]
\item One of the central models has a chain-recurrent central segment;
\item Both of the central models have arbitrarily small trapping strips;
\item One has arbitrarily small trapping strips, and the other has arbitrarily small expanding strips;
\item Both of the central models have arbitrarily small expanding strips.
\end{enumerate}

\nibf[Subcase (1).]
We assume without loss of generality that the central model $(\hat{K}_{0,+},\hat{f}_+;\pi_+)$ has a chain-recurrent central segment $I=\{x\}\times[0,a]$. Then the center curve $\ga_x=\pi_+(I)$ is contained in a chain-transitive subset of $\La$. By Lemma \ref{lem:central-model}, the existence of chain-recurrent central segment is a local property of $\hat{K}_{0,+}$. In other words, $I$ can be assumed small enough so that $\ga_x$ is contained in the maximal invariant set of an arbitrarily small neighborhood of $K_0$. Hence, we can assume that $\ga_x$ is contained in a chain-transitive subset $K'\subset K\cap\La$.
By Lemma \ref{lem:gen_0} (item \ref{item:gen_0-chain-transitive-set}), $K$ contains periodic orbits arbitrarily close to an interior point of $\ga_x$. Since each orbit in $K$ has a stable manifold (resp. unstable manifold) with uniform size, tangent to $G^{ss}\oplus\langle X\rangle$ (resp. $G^{uu}\oplus\langle X\rangle$), the  stable manifold of some periodic orbit $Q$ will intersect with the union of unstable manifolds of points in $\ga_x$. As $\La$ is Lyapunov stable and $\ga_x\subset\La$, it contains the periodic orbit $Q$, a contradiction.

\nibf[Subcase (2).]
Let $S^{\iota}_0$ be arbitrarily small trapping strips for $(\hat{K}_{0,\iota},\hat{f}_{\iota})$ for $\iota\in\{+,-\}$. By continuity of $\hat{f}_{\iota}$, we can assume that $(\hat{K}_{\iota},\hat{f}_{\iota})$ also has a trapping strip $S^{\iota}$, by letting $U$ to be small enough (so that $\hat{K}_{\iota}$ is close enough to $\hat{K}_{0,\iota}$).

For any $x\in K$, $\iota\in\{+,-\}$, let $\hat{x}^{\iota}$ be the point in $\hat{K}_{\iota}$ such that $\pi_{\iota}(\hat{x}^{\iota},0)=x$, and denote by
\[\ga^{\iota}_x=\pi_{\iota}(S^{\iota}\cap (\{\hat{x}^{\iota}\}\times [0,1])).\]
Then $\ga_x=\ga^+_x\cup\ga^-_x$ is a center curve with $x$ contained in its interior.
For any periodic orbit $Q\subset K$ and a point $q\in Q$, it can be shown that $\ga_q$ is contained in the union of stable manifolds of finitely many hyperbolic periodic points in $\ga_q$, see \cite[Proposition 5.4]{XZ}. Moreover, restricted to $N_q(\de_0)$, the union of stable manifolds of the periodic points in $\ga_q$ forms an $(s+1)$-dimensional $C^1$ disk $D_q$ tangent to $G^{ss}\oplus G^c$ at $q$ and with $q$ contained in its interior, where $s=\dim G^{ss}$. Let $Q_n\subset K$ be a sequence of periodic orbits with its limit in $K_0$, and suppose $q_n\in Q_n$ are points such that $q_n\to x\in K_0$, as $n\to\infty$. Then the unstable manifold $W^u(\orb(x))$ (tangent to $G^{uu}\oplus\langle X\rangle$) will intersect with $D_{q_n}$ for $n$ large, which implies that $W^u(\orb(x))$ will intersect with the stable manifold of some periodic orbit. Since $x\in K_0\subset \La$ and $\La$ is Lyapunov stable, the last statement further implies that $\La$ contains periodic orbits, a contradiction.

\nibf[Subcase (3).]
Suppose $(\hat{K}_{0,+},\hat{f}_+)$ has an expanding strip $S^+_0$ and $(\hat{K}_{0,-},\hat{f}_-)$ has a trapping strip $S^-_0$. By continuity of $\hat{f}_+$ and $\hat{f}_-$, we can also assume that $(\hat{K}_+,\hat{f}_+)$ has an expanding strip $S^+$ and $(\hat{K}_-,\hat{f}_-)$ has a trapping strip $S^-$. Let us denote
\[S^+_{\infty}=\bigcap_{i=0}^{\infty}\hat{f}_+^{-i}(S^+).\]
For any $x\in K$, let $\hat{x}^+$ be the point in $\hat{K}_+$ such that $\pi_+(\hat{x}^+,0)=x$, and
\[\ga^+_x=\pi_+(S^+\cap(\{\hat{x}^+\}\times[0,1])),\quad \Ga^+_x=\pi_+(S^+_{\infty}\cap(\{\hat{x}^+\}\times[0,1])).\]
Let $p_n\in K$ be a sequence of periodic points such that $p_n\to x_0\in K_0$. Let $len(\Ga^+_{p_n})$ be the length of $\Ga^+_{p_n}$. Then as $n\to\infty$, we must have $len(\Ga^+_{p_n})\to 0$, since otherwise a similar argument as for subcase (2) would imply that $\La$ contains periodic orbits, leading to a contradiction. This fact implies that $\ga^+_{x_0}$ is contained in the chain-unstable set of $K_0$ and hence contained in $\La$, see \cite[Remark 6.2]{Yan07} and the claim right before it.
In fact, one can show that $\ga^+_{x_0}\subset W^u(K_0)$, see the claim right after Remark 6.2 in \cite{Yan07}. Then we can assume that $\{P_{-nT}(\ga^+_{x_0})\}_{n\geq 0}$ is contained in an arbitrarily small neighborhood of $K_0$, so that every point in $\ga^+_{x_0}$ has a strong unstable manifold of dimension equal to $\dim G^{uu}$ (tangent to a well-defined bundle $G^{uu}$), and has a uniform size.

\begin{Remark}
	Up to this point, the proof is largely similar to the proof of the main lemma in \cite[Section 6]{Yan07}. In the rest of the proof, some new arguments will be used to overcome the difficulty caused by singularities.
\end{Remark}

Fix a point $y\in \ga^+_{x_0}\setminus\{x_0\}$. By Lemma 3.2 (the technique lemma) in \cite{Yan07} (see also \cite[Proposition 1.1]{Wang}), there exists a sequence of periodic points $(q_n)_{n\in\NN}$ such that
\begin{itemize}
	\item $q_n\to y$ as $n\to\infty$,
	\item $\{y\}\cup K_0$ is contained in the Hausdorff limit of the orbits $Q_n=\orb(q_n)$, which is denoted by $K_1$, and
	\item moreover, let $\mu_n$ be the uniform distribution over the periodic orbit $Q_n$, then $\mu_n$ converges in the weak* topology to an invariant measure $\mu$ with support in $K_0$.
\end{itemize}

Let $C,\eta,\de,T$ be the constants given by Lemma \ref{lem:BasicAwayTang} for $X$. Then we can assume that the sequence $Q_n$ admits a dominated splitting $\cN_{Q_n}=H^s_n\oplus H^c_n\oplus H^u_n$ with respect to the linear Poincar\'e flow, satisfying the properties listed as in Lemma \ref{lem:BasicAwayTang}. By choosing a subsequence, we assume that $s'=\dim H^s_n$ for all $n$. Let $s=\dim G^{ss}$.

\begin{Claim}
One has $s'\geq s-1$. Moreover, if $s'=s-1$, then there exist $C',\eta'>0$ such that $Q_n$ is $(C',\eta',T,H^{cs}_n)$-$\psi_t$-contracting at period, where $H^{cs}_n=H^s_n\oplus H^c_n$ and $\dim H^{cs}_n=s$, for $n$ large.
\end{Claim}
\begin{proof}[Proof of the Claim.]
Suppose on the contrary that $s'<s-1$. Since $K_1$ is the Hausdorff limit of $Q_n$, the dominated splitting $\cN_{Q_n}=H^s_n\oplus H^c_n\oplus H^u_n$ induces a dominated splitting of the normal bundle over $K_1$, $\cN_{K_1}=H^s\oplus H^c\oplus H^u$. Then $\dim H^s=s'<s-1$ and, by Lemma \ref{lem:BasicAwayTang}, $\dim H^c\in\{0,1\}$. Let $H_n$ be the set of $(1,\eta/2,T,H^u_n)$-$\psi^*_t$-expanding points on $Q_n$, and let $H=\cap_{n=1}^{\infty}\Cl(\cup_{k\geq n}H_k)$. Then $H\subset K_1$, and any point in $H\setminus\Sing(X)$ is $(1,\eta/2,T,H^u)$-$\psi^*_t$-expanding. By~\cite[Proposition 3.5]{ABV00} (see also Lemma \ref{lem:pliss-for-measure} and Remark \ref{rmk:pliss-for-measure}), there is $\theta>0$ such that $\mu_n(H_n)\geq \theta$, for $n$ large. Then
\[\mu(H)=\lim_{n\to\infty}\mu(\Cl(\cup_{k\geq n}H_k))\geq\lim_{n\to\infty}\limsup_j\mu_j(\Cl(\cup_{k\geq n}H_k))\geq \theta>0.\]
In particular, there exists a $(1,\eta/2,T,H^u)$-$\psi^*_t$-expanding point $x\in H\cap\supp(\mu)\subset K_0$. This leads to a contradiction since $H^u\cap G^{ss}$ is nontrivial restricted to the set $K_0$, and $G^{ss}$ is uniformly contracting. Hence we have shown $s'\geq s-1$.

If $s'=s-1$, the same argument as above shows that $H^c_n$ is nontrivial for $n$ large, which implies $H^c$ is nontrivial. Also, restricted to $K_0$, we have $H^c\subset G^{ss}$. Then the fact that $G^{ss}$ is contracting and $\mu_n$ converges to $\mu$ with $\supp(\mu)\subset K_0$, implies that there exists $C',\eta'>0$ such that $Q_n$ is $(C',\eta',T,H^c_n)$-$\psi_t$-contracting at period. As $H^s_n$ is dominated by $H^c_n$, one concludes that $Q_n$ is also $(C',\eta',T,H^{cs}_n)$-$\psi_t$-contracting at period, with $H^{cs}_n=H^s_n\oplus H^c_n$.
\end{proof}

By the previous claim, either $s'\geq s$ and $Q_n$ is $(C,\eta,T,H^s_n)$-$\psi_t$-contracting at period, or $\dim H^{cs}_n=s$ and $Q_n$ is $(C',\eta',T,H^{cs}_n)$-$\psi_t$-contracting at period. We will consider only the first case, as the other case can be handled in a similar way.

Assuming $s'\geq s$. If $H^s_n$ is uniformly contracting, then the stable manifolds of $Q_n$ will intersect with the union of strong unstable manifolds of points in $\ga^+_{x_0}$, for $n$ large. Since $\La$ is Lyapunov stable and $\ga^+_{x_0}\subset\La$, one would conclude that $\La$ contains periodic orbits, leading to a contradiction. If $H^s_n$ is not uniformly contracting, one can apply Lemma \ref{lem:intersection} and Proposition \ref{prop:homogeneous-index} to show that $W^u(K_1)\pitchfork W^s(Q_n)\neq\emptyset$ for $n$ large. Note that $K_1$ is a chain-transitive set containing $K_0$, we have $K_1\subset\La$. Hence by Lyapunov stability of $\La$, it contains periodic orbits, a contraction again. This finishes the proof for subcase (3).

\nibf[Subcase (4).] This case can be easily reduced to Subcase (3); see \cite[Subcase A.3]{Yan07} for details.

\vskip 1.5em

Finally, when there is no orientation of $G^c$ preserved by $P_T$, we have essentially one central model $(\widehat{K},\widehat{f})$. One need only consider three subcases: (a) with a chain-recurrent central segment; (b) with arbitrarily small trapping strips for $\hat{f}$; and (c) with arbitrarily small trapping strips for $\hat{f}^{-1}$. Each subcase can be proven similarly as the corresponding one for the orientation preserving case. The proof of Proposition \ref{prop:non-singular-set} is now complete.

\begin{Remark}\label{rmk:non-singular-set}
	The proof also gives information about the indices of periodic orbits in the homoclinic class. Precisely, under the assumptions of Proposition \ref{prop:non-singular-set}, suppose there exists a minimally non-hyperbolic non-singular set $K_0\subset\La_0\subset \La$ admitting a partially hyperbolic splitting $\cN_{K_0}=G^{ss}\oplus G^c\oplus G^{uu}$ with one-dimensional center bundle. Then $\La$ contains periodic orbits of index $i\geq\dim G^{ss}$.
\end{Remark}

\subsection{Proof of Theorem~B}

The first and the second items of Theorem~B have been proven in Proposition \ref{prop:homogeneous-index} and Proposition \ref{prop:indices-of-periodic-orbits-0} respectively.
As for the last item, we have the following result.
\begin{Proposition}\label{prop:indices-of-periodic-orbits-1}
For any $C^1$ generic $X\in\mathscr{X}^1(M)\setminus\Cl(\mathcal{HT})$, $\si\in\Sing(X)$, suppose $C(\si)$ is a Lyapunov stable chain recurrence class. Then there exist a neighborhood $\cU$ of $X$ and a neighborhood $U$ of $C(\si)$ such that for any $Y\in\cU$ and any periodic orbit $\ga\subset U$ of $Y$, one has $\ind(\ga)\leq \ind(\si)$.
\end{Proposition}

\begin{proof}
	Suppose on the contrary that there exist $i>\ind(\si)$ and a fundamental $i$-sequence with the limit $\La_0\subset C(\si)$. Since $X$ is $C^1$ generic, by item \ref{item:gen_0-fund-seq} of Lemma \ref{lem:gen_0} there is a sequence of periodic orbits $\ga_n$ of $X$, such that $\ind(\ga_n)=i$ and $\ga_n\to\La_0$ in the Hausdorff topology. By Corollary \ref{cor:BasicAwayTang}, there is a dominated splitting $\cN_{\ga_n}=G^{cs}_n\oplus G^{cu}_n$  such that its index $i'=i$ or $i-1$, and $\ga_n$ is $(C,\eta,T,G^{cs}_n)$-$\psi_t$-contracting at period, where the constants $C,\eta,T>0$ are given by Lemma \ref{lem:BasicAwayTang}.
	
	We claim that $G^{cs}_n$ is uniformly contracting with constants uniform in $n$. Otherwise, suppose $G^{cs}_n$ is not uniformly contracting with constants uniform in $n$, Lemma \ref{lem:intersection} ensures that $W^u(\La_0)\pitchfork W^s(\ga_n)\neq\emptyset$ for $n$ large. As $C(\si)$ is Lyapunov stable, $\ga_n\subset C(\si)$ for $n$ large. This is a contradiction to Proposition \ref{prop:indices-of-periodic-orbits-0}.
	
	If $\La_0$ contains a singularity $\rho$, we have $\ind(\rho)=\ind(\si)$ by Proposition \ref{prop:homogeneous-index}. Then as is shown that $G^{cs}_n$ is uniformly contracting with constant uniform in $n$ and $\dim G^{cs}_n=i'\geq\ind(\rho)$, it follows from Theorem~\ref{thm:Liao} that $W^u(\rho)\pitchfork W^s(\ga_n)\neq\emptyset$ for $n$ large. Again, this is a contradiction to Proposition \ref{prop:indices-of-periodic-orbits-0}.
	
	Therefore, we must have $\La_0\cap\Sing(X)=\emptyset$. Since $G^{cs}_n$ is uniformly contracting with constants uniform in $n$, the dominated splitting $\cN_{\ga_n}=G^{cs}_n\oplus G^{cu}_n$ converges to a partially hyperbolic splitting $\cN_{\La_0}=G^s\oplus G^{cu}$, with $G^s$ uniformly contracting and $\dim G^s=i'\geq\ind(\si)$.
	
	If $\La_0$ is a hyperbolic set, it is easy to see that $C(\si)$ contains periodic orbits of index at least as large as $i'$. If $\La_0$ is not hyperbolic, Proposition \ref{prop:non-singular-set} shows that $C(\si)$ also contains periodic orbits, and in fact, it contains periodic orbits of index at least as large as $i'$ (Remark \ref{rmk:non-singular-set}). Since $i'\geq\ind(\si)$, we have shown in both cases that $C(\si)$ contains a periodic orbit of index at least $\ind(\si)$, which contradicts Proposition \ref{prop:indices-of-periodic-orbits-0}.
\end{proof}

The proof of Theorem~B is now complete.

\subsection{Proof of Corollary \ref{cor:index-completeness} assuming Theorem~A}
Assuming that Theorem~A holds, Corollary~\ref{cor:index-completeness} is then a simple consequence of Theorem~B (item 2) and the following result.

\begin{Proposition}\label{prop:index-completeness}
Under the same assumption as in Proposition \ref{prop:indices-of-periodic-orbits-1} and suppose $C(\si)$ contains a periodic orbit of index $\alpha<\ind(\si)-1$, then it contains a periodic orbit of index $\alpha+1$.
\end{Proposition}

\begin{proof}
Let $\ga$ be a periodic orbit in $C(\si)$ with index $\alpha<\ind(\si)-1$. Note that we have $\dim W^u(\ga)+\dim W^s(\si)>\dim M$, hence $W^u(\ga)\pitchfork W^s(\si)\neq \emptyset$ by genericity. Moreover, by Lyapunov stability of $C(\si)$ and Lemma \ref{lem:gen_1}, the unstable manifold $W^u(\si)$ is dense in $C(\si)$. In particular, $W^u(\si)$ accumulates on the stable manifold of $\ga$. This allows us to apply the connecting lemma of Wen-Xia \cite{WeX} in an arbitrarily small neighborhood of $\ga$, so that we can obtain a vector field $Y$ arbitrarily $C^1$ close to $X$ having a heteroclinic cycle related to $\si$ and $\ga$. Note that $\ind(\ga)=\alpha<\ind(\si)-1$. By Corollary \ref{cor:lorenz-like2} and following the argument in \cite[Lemma 4.5]{SGW}, there exists a vector $Z$ arbitrarily $C^1$ close to $Y$ (and hence arbitrarily close to $X$) with a periodic orbit of index $\alpha+1$ arbitrarily close to the heteroclinic cycle in the Hausdorff topology. 
Thus there exists a fundamental $(\alpha+1)$-sequence of $X$ with $\ga$ contained in its limit $\Lambda\subset C(\si)$. Since $X$ is $C^1$ generic, by \cite[Lemma 3.5]{Wen04} there exists a sequence of hyperbolic periodic orbits $\ga_n$ of index $\alpha+1$ of $X$ itself that converges to $\Lambda$ in the Hausdorff topology.

By Corollary \ref{cor:BasicAwayTang}, there exists a dominated splitting $\cN_{\ga_n}=G^{cs}_n\oplus G^{cu}_n$ such that its index $i=\alpha$ or $\alpha+1$, and $\ga_n$ is $(C,\eta,T,G^{cs}_n)$-$\psi_t$-contracting at period, where the constants $C,\eta,T>0$ are given by Lemma \ref{lem:BasicAwayTang}. By Lemma \ref{lem:intersection}, one has either $W^u(\Lambda)\pitchfork W^s(\ga_n)\neq\emptyset$ for $n$ large or $G^{cs}_n$ is uniformly contracting with constants uniform in $n$. Note that $\ga\subset \Lambda$ and $\ind(\ga)=\alpha$. If $G^{cs}_n$ is uniformly contracting with constants uniform in $n$, we take $x_n\in\gamma_n$ that converge to a point $x\in\gamma$. Note that $W^s(x_n)$ has size at least $\left(\delta_0\liminf_n\|X(x_n)\|\right)$ for some $\delta_0>0$ which is uniform in $n$, and $\dim W^u(\gamma) + \dim W^s(\gamma_n) = \dim M+1$. It follows that $W^u(\gamma)\pitchfork W^s(\gamma_n)\ne\emptyset$. Since $C(\si)$ is Lyapunov stable, one concludes in both cases that $\ga_n\subset C(\si)$ for $n$ large.
\end{proof}

\section{A reduction of the main theorem}
\label{sect:reduction}

The aim of this section is to reduce our main theorem (Theorem A) to Proposition \ref{prop:main-proposition}. This reduction can be viewed as a direct generalization of the main steps in Gan-Yang's proof \cite{GY} of the Weak Palis Conjecture for $C^1$ singular flows on 3-manifolds.

Throughout this section, $X$ is a $C^1$ generic vector field away from tangencies, and $C(\si)$ is a non-trivial Lyapunov stable chain recurrent class associated to a singularity $\si$.
\subsection{The difficulty and general ideas for the proof of Theorem A}
\label{sect:first-reduction}

We shall give some general ideas as how the main theorem will be proven. In particular, we will show where the main difficulty lies when dealing with general dimensions and how it will be overcome.

Proving by contradiction, let us assume that $C(\si)$ is aperiodic.
Then one can show that the chain recurrence class admits a partially hyperbolic splitting $T_{C(\si)}M=E^{ss}\oplus F$ with respect to the tangent flow, where $E^{ss}$ is contracting and $\dim E^{ss}=\ind(\si)-1$ (Proposition \ref{prop:ph-or-horseshoe}). Note that the flow direction is contained in the bundle $F$, {\itshape i.e.,} $X(x)\in F(x)$ for any $x\in C(\si)$, and one must have $\dim F\geq 2$. The core of our main theorem then lies in the following proposition.

\begin{Proposition}[The Main Proposition]\label{prop:main-proposition}
For a $C^1$ generic vector field $X\in\mathscr{X}^1(M)\setminus\Cl(\mathcal{HT})$ and $\si\in\Sing(X)$, if $C(\si)$ is a nontrivial Lyapunov stable chain recurrence class with a partially hyperbolic splitting $T_{C(\si)}M=E^{ss}\oplus F$, such that $\Phi_t|_{E^{ss}}$ is  uniformly contracting and $\dim E^{ss}=\ind(\si)-1$, then $C(\si)$ contains a periodic orbit.
\end{Proposition}

In the case when $\dim F=2$, Proposition~\ref{prop:main-proposition} can be proved by generalizing a result of \cite{GY}, see Theorem \ref{thm.gxyz}. However, their approach relies essentially on the codimension-two condition, hence cannot be used to the case when $\dim F\geq 3$. A new approach is needed here.

As we discussed in Section~\ref{sect:preliminaries}, Li-Gan-Wen \cite{LGW05} discovered a useful mechanism (Lemma \ref{lem:matching}) governing the way how periodic orbits with a given dominated splitting can accumulate on a singularity. One would like to show that, assuming $C(\si)$ contains no periodic orbits, then there exists a sequence of periodic orbits accumulating on a singularity in $C(\si)$ that violates the mechanism mentioned above.

Let us be more precise.
Let $\{\ga_n\}$ be a sequence of periodic orbits of $X$, such that $\ga_n\to \La$ in the Hausdorff topology, where $\La\subset C(\si)$ is a compact invariant set of $X$. Suppose the sequence $\{\ga_n\}$ can be obtained such that it admits a dominated splitting of index $\ind(\si)$ with respect to the linear Poincar\'e flow, and $\si\in \La$. The mechanism in \cite{LGW05} (Lemma \ref{lem:matching}) then asserts that there is a dominated splitting of the unstable subspace $E^u(\si)=E^{c}(\si)\oplus E^{uu}(\si)$ such that $\dim E^{c}=1$; moreover, the flow direction of any sequence of points $x_n\in \ga_n$ can not approximate a 1-dimensional subspace in the ``restricted area'' $(E^s(\si)\oplus E^{uu}(\si))\setminus (E^s(\si)\cup E^{uu}(\si))$, where $E^s(\si)$ is the stable subspace of $T_{\si}M$. However, this is not a strong restriction to periodic orbits accumulating $\si$, as the sequence of periodic orbits may just bypass the ``restricted area'' and still admits a dominated splitting of index $i=\ind(\si)$.

Therefore, the main difficulty is to show the existence of a sequence of periodic orbits that admits a dominated splitting of some appropriate index ({\itshape i.e.,} $\ind(\si)$) and, at the same time, ``travels into'' the corresponding ``restricted area''. Note that this is not a problem when $\dim F=2$ as in such a case one would have $\dim E^u(\si)=1$ (recall that $\sigma$ must be Lorenz-like, see Proposition~\ref{prop:homogeneous-index}).

Since $C(\si)$ is Lyapunov stable, the unstable manifold of any singularity is contained in the class. Thus, one may get the wrong impression that if there is a dominated splitting $E^u(\si)=E^{c}(\si)\oplus E^{uu}(\si)$ with $\dim E^{c}=1$ and somehow one can obtain a sequence of periodic orbits $\{\gamma_n\}$ admitting an index $i=\ind(\si)$ dominated splitting w.r.t. the linear Poincar\'e flow, such that $\{\gamma_n\}$ accumulates on certain part of the strong unstable manifold $W^{uu}(\si)$ tangent to $E^{uu}(\si)$, then $\{\gamma_n\}$ will naturally ``travel into'' the ``restricted area'' $(E^s(\si)\oplus E^{uu}(\si))\setminus (E^s(\si)\cup E^{uu}(\si))$. This may not be true in general, even if the vector field can be linearized in a neighborhood of $\si$.
Nonetheless, as $X$ is $C^1$ generic, one can show that if $\{\gamma_n\}$ approximates the entire class $C(\si)$, then it will ``travel into'' the ``restricted area'' (see Lemma \ref{lem:gen_1}).

We therefore try to obtain a sequence $\{\gamma_n\}$ that admits an index $i=\ind(\si)$ dominated splitting and, at the same time, approximates the whole class $C(\si)$. Note that we may consider periodic orbits of vector fields arbitrarily $C^1$ close to $X$, as long as they admit a uniform index $i$ dominated splitting. For such a purpose, we will consider a special sequence of $C^2$ vector fields $X_n$ that converges to $X$ in the $C^1$ topology, and locate for each $X_n$ a special invariant measure $\mu_n$ (physical-like measure, see Section \ref{sect:physical-like-measure}). Then we will show that $\{\supp(\mu_n)\}$ converges in the Hausdorff topology to $C(\si)$. Moreover, the support of each $\mu_n$ can be approximated by periodic orbits of $X_n$ admitting a dominated splitting of index $i=\ind(\si)$ with respect to the linear Poincar\'e flow\footnote{In fact, we will consider a measure $\mu'_n$ that is the restriction of $\mu_n$ to a $\phi^{X_n}_t$-invariant set $\La_n$ with $\mu_n(\La_n)\to 1$ as $n\to\infty$. See Section \ref{sect:hyperbolic-pts}.}, which is uniform in $n$.

The reason $C^2$ vector fields are considered is that for such a vector field, an ergodic measure satisfying Pesin's entropy formula is {\em $u$-saturated} in the sense that its support contains the unstable manifold of almost every points~\cite{LeY85}.
This constitutes one of the primitive components of our proof.
We will come back to this in Section \ref{sect:proof-of-main-prop}.

\subsection{The reduction and proof of Theorem A}\label{sect:reduction-main}

Our main theorem (Theorem A) is a simple consequence of Proposition~\ref{prop:main-proposition} and the following proposition.

\begin{Proposition}\label{prop:ph-or-horseshoe}
Given $X$ and $C(\sigma)$ as stated in the beginning of this section, one of the following properties holds for $C(\si)$:
\begin{itemize}
\item there is a partially hyperbolic splitting $T_{C(\si)}M=E^{ss}\oplus F$ w.r.t. the tangent flow $\Phi^X_t$, s.t. $E^{ss}$ is contracting and $\dim E^{ss}=\ind(\si)-1$;
\item $C(\si)$ is a homoclinic class.
\end{itemize}
\end{Proposition}

Assuming Proposition~\ref{prop:main-proposition} and Proposition~\ref{prop:ph-or-horseshoe}, now let us prove the main theorem.

\begin{proof}[Proof of Theorem A]
Let $\La$ be any nontrivial Lyapunov stable chain recurrence class of $X$. If $\La\cap\Sing(X)=\emptyset$, then Proposition \ref{prop:non-singular-set} implies that $\La$ is a homoclinic class. If $\La$ contains a singularity $\si$ of $X$, then by Proposition \ref{prop:ph-or-horseshoe} either it is a homoclinic class, or there is a partially hyperbolic splitting $T_{\La}M=E^{ss}\oplus F$ with $E^{ss}$ contracting and $\dim E^{ss}=\ind(\si)-1$. In the latter case, we apply Proposition~\ref{prop:main-proposition} which shows that $\Lambda$ is a homoclinic class.
\end{proof}

The proof of Proposition~\ref{prop:ph-or-horseshoe} occupies the next subsection.

\subsection{Partial hyperbolicity or homoclinic intersections}
\label{sect:ph-or-horseshoe}
In this subsection, we prove Proposition \ref{prop:ph-or-horseshoe}. We rely on the following result by \cite{GY}.

\begin{Lemma}[{\cite[Proposition 4.9]{GY}}]\label{lem:mixingdominated}
	Let $\Lambda$ be a compact invariant set of $X\in \xX^1(M)$ verifying the following properties:
	\begin{itemize}
		\item $\Lambda\setminus{\rm Sing}(X)$ admits a dominated splitting
			${\cal N}_{\Lambda}=G^{cs}\oplus G^{cu}$ in the normal bundle w.r.t. the linear Poincar\'e flow $\psi_t$, with index $i$.
		\item Every singularity $\sigma\in\Lambda$ is hyperbolic and $\ind(\sigma)>i$. Moreover, $T_\sigma M$ admits a partially hyperbolic splitting $T_{\sigma}M=E^{ss}\oplus F$ with respect to the tangent flow, where $\dim E^{ss}=i$ and for the corresponding strong stable manifolds $W^{ss}(\sigma)$, one has $W^{ss}(\sigma)\cap \Lambda=\{\sigma\}$.
		\item For every $x\in\Lambda$, one has $\omega(x)\cap {\rm Sing}(X)\neq\emptyset$.
	\end{itemize}
	Then one has
	\begin{itemize}
		\item either $\Lambda$ admits a partially hyperbolic splitting $T_\Lambda M=E^{ss}\oplus F$ with respect to the tangent flow $\Phi_t$, where $\dim E^{ss}=i$,
		\item or, there is a sequence of hyperbolic periodic orbits $\{\gamma_n\}_{n\in\NN}$ of index $i$ such that
		\begin{itemize}
			\item $\tau(\gamma_n)\to\infty$, as $n\to\infty$,
			\item $H(\gamma_n)\cap\Lambda\neq\emptyset$ for each $n\in\NN$,
			\item There is $T>0$ such that for $x_n\in\gamma_n$,
				\[\lim_{n\to\infty}\frac{1}{\lfloor\tau(\gamma_n)/T\rfloor}\sum_{i=0}^{\lfloor\tau(\gamma_n)/T\rfloor-1}\log\|\psi_T|_{G^s(x_n)}\|=0,\]
				where $G^s(\ga_n)$ is the stable subbundle of $\cN_{\ga_n}$.
		\end{itemize}
	\end{itemize}
\end{Lemma}

We will need another lemma for the proof of Proposition \ref{prop:ph-or-horseshoe}.
\begin{Lemma}\label{lem:lorenz-like3}
Let $X\in\mathscr{X}^1(M)\setminus\Cl(\mathcal{HT})$ be a $C^1$ generic vector field and $\si\in\Sing(X)$ a hyperbolic singularity with $\sv(\si)>0$. Assume that the chain recurrence class $C(\si)$ is nontrivial, then there is a dominated splitting $\cN_{C(\si)}=G^{cs}\oplus G^{cu}$, such that $\dim G^{cs}=\ind(\si)-1$.
\end{Lemma}
\begin{proof}
By Corollary \ref{cor:BasicAwayTang} and Remark \ref{rmk:fundamental-sequences}, we need only show that every point of $C(\si)$ is contained in the limit of a fundamental $(\ind(\si)-1)$-sequence. Using the connecting lemma for pseudo-orbits \cite{BoC}, we can show that $C(\si)\subset \Cl(W^u(\si))\cap\Cl(W^s(\si))$ (see item 2 of Lemma \ref{lem:gen_1}). Now for any $x\in C(\si)$, applying the connecting lemma of Wen-Xia \cite{WeX}, there is a vector field $Y$ arbitrarily $C^1$ close to $X$ such that $\si_Y$ has a homoclinic orbit passing by an arbitrarily small neighborhood of $x$. Then in view of Corollary \ref{cor:lorenz-like2}, one can obtain a fundamental $(\ind(\si)-1)$-sequence with $x$ contained in its limit.
\end{proof}

\begin{proof}[Proof of Proposition \ref{prop:ph-or-horseshoe}]
Let us check that the conditions in Lemma \ref{lem:mixingdominated} are satisfied by $C(\si)$:
\begin{itemize}
\item there is a dominated splitting $\cN_{C(\si)}=G^{cs}\oplus G^{cu}$, $\dim G^{cs}=\ind(\si)-1$: this is true by Proposition \ref{prop:homogeneous-index} and Lemma \ref{lem:lorenz-like3};
\item every singularity $\rho\in C(\si)$ is Lorenz-like and $\ind(\rho)=\ind(\si)$: by Proposition \ref{prop:homogeneous-index}.
\end{itemize}

Then either there exists $x\in C(\si)$ such that $\omega(x)\cap\Sing(X)=\emptyset$ and $C(\si)$ is a homoclinic class by Proposition \ref{prop:non-singular-set}, or Lemma \ref{lem:mixingdominated} can be applied and the conclusion also follows.
\end{proof}

\section{Preliminaries on invariant measures}\label{s.7}
\label{sect:preliminaries-measure}
In this section we give some preliminaries regarding invariant measures that will be used in the proof of the general case of  main proposition (Proposition \ref{prop:main-proposition}).
We will consider Borel probability measures on $M$, the set of which is denoted by $\mathfrak{M}$ and endowed with the weak* topology.

\subsection{Ergodic closing lemma and Lyapunov exponents}

Given $X\in\xX^1(M)$, an ergodic measure $\mu$ for $\phi^X_t$ is called {\em nontrivial} if it is not supported on a critical element.
A point $x\in M\setminus\Sing(X)$ is called {\em strongly closable} if for any $C^1$ neighborhood $\cU$ of $X$ and for any $\vep>0$, there is a vector field $Y\in\cU$, a point $y\in M$ and constants $\tau>0$, $L>0$, such that
\begin{itemize}
	\item $\phi^Y_{\tau}(y)=y$;
	\item $d(\phi^X_t(x),\phi^Y_t(y))<\vep$, for all $0\leq t\leq \tau$;
	\item $Y(z)=X(z)$, for all $z\in M\setminus B_{\vep}(\phi^X_{[-L,0]}(x))$, where $\phi^X_{[-L,0]}(x)=\{\phi^X_t(x):-L\leq t\leq 0\}$.
\end{itemize}
The set of strongly closable points of $X$ is denoted by $\Si(X)$. The following lemma is the flow version of Ma\~n\'e's Ergodic Closing Lemma \cite{Man82}.

\begin{Lemma}[\cite{Wen96}] \label{lem:ergodic-closing}
	For any $T>0$ and any $\phi^X_T$-invariant Borel probability measure $\mu$ over $M$, one has $\mu(\Si(X)\cup\Sing(X))=1$.
\end{Lemma}

This lemma enables us to approximate a nontrivial ergodic measure by periodic measures of nearby systems. In fact, for any nontrivial ergodic measure $\mu$ of $X$, one can choose $T>0$ such that $\mu$ is ergodic for the time-$T$ map $\phi_T$ (\cite{PuSh}). Fixing such a $T$, let $B(\mu)$ be the set of generic points of $\mu$, $\mu(B(\mu))=1$. Here, $x$ is called a {\em generic point} of $\mu$ if the sequence
\[\frac{1}{nT}\sum_{i=0}^{n-1}\int_0^T(\phi_t)_*\de_{iT}(x)dt\]
converges to $\mu$ in the weak* topology, as $n\to\infty$.
Since $\mu$ is not the Dirac measure supported on a singularity, Lemma \ref{lem:ergodic-closing} implies that $\mu(\supp(\mu)\cap B(\mu)\cap\Si(X))=1$. Then one can take a point $x\in \supp(\mu)\cap B(\mu)\cap\Si(X)$, and by definition of $B(\mu)$ and $\Si(X)$, there is a fundamental sequence $\{(\ga_n,X_n)\}$ such that $(\ga_n,X_n)\to (\supp(\mu), X)$ and $\de_{\ga_n}\to\mu$ in the weak* topology.

One would like to know the connections between Lyapunov exponents of $\gamma_n$ and that of $\mu$.
Let $\la_1\leq\la_2\leq\cdots\leq\la_{\dim M}$ be the Lyapunov exponents (counting multiplicity) of $\mu$. Note that there is always a zero exponent for $\mu$ corresponding to the flow direction. We define the {\em index} of $\mu$ as the number of negative exponents, counting multiplicity.
If $\mu$ has no zero Lyapunov exponent other than the one corresponding to flow direction, then $\mu$ is called {\em hyperbolic}; otherwise, it is called {\em non-hyperbolic}.

\begin{Remark}
There is a result for $C^1$ generic diffeomorphisms \cite{ABC} stating that an ergodic measure $\mu$ can be approximated by measures $\de_{\ga_n}$ supported on periodic orbits $\ga_n$ in the weak* topology such that the set of Lyapunov exponents of $\de_{\ga_n}$ converges to that of $\mu$. In particular, the index of $\ga_n$ is strongly related to $\mu$. It is possible that this result can be generalized to $C^1$ generic flows. We will not do so, however, as we will deal with vector fields that are not $C^1$ generic (namely $C^2$).
\end{Remark}

Before proceeding to the technical results of this subsection, we remark that similar discussions for diffeomorphisms can be found in \cite[Section 4]{LVY}.

\begin{Lemma}\label{lem:le-away-tang}
Let $X\in\xX^1(M)\setminus\Cl(\mathcal{HT})$, $\mu$ be a nontrivial ergodic measure for $X$. Then there is a fundamental $\ind(\mu)$-sequence $\{(\ga_n,X_n)\}$ such that $(\ga_n,X_n)\to(\supp(\mu),X)$ and $\de_{\ga_n}\to \mu$ in the weak* topology. Moreover, if $\mu$ is non-hyperbolic, then $\mu$ has exactly two zero Lyapunov exponents, and one can find a fundamental $(\ind(\mu)+1)$-sequence $\{\hat\gamma_n\}$ with $(\hat\ga_n,X_n)\to(\supp(\mu),X)$ and $\de_{\hat\ga_n}\to \mu$.
\end{Lemma}
\begin{proof}

By Lemma \ref{lem:ergodic-closing} and the argument followed, there exists a fundamental sequence $\fF=\{(\ga_n,X_n)\}$ such that $(\ga_n,X_n)\to(\supp(\mu),X)$ and $\de_{\ga_n}\to \mu$ in the weak* topology. Since $\mu$ is nontrivial, $\tau(\ga_n)\to\infty$ as $n\to\infty$. By Lemma \ref{lem:BasicAwayTang}, there exist $C>0$, $\de>0$, $\eta>0$ and $T>0$ such that the fundamental sequence $\fF$ admits a $T$-dominated splitting $\cN_{\ga_n}=G^s_n\oplus G^c_n\oplus G^u_n$, where Lyapunov exponents corresponding to $G^s_n$ (resp. $G^u_n$) are smaller than $-\de$ (resp. larger than $\de$), and $\dim G^c_n\leq1$. Moreover, $\ga_n$ is $(C,\eta,T,G^s_n)$-$\psi^{X_n}_t$-contracting at period and $(C,\eta,T,G^u_n)$-$\psi^{X_n}_t$-expanding at period, for $n$ large.

The dominated splitting over $\fF$ induces a $T$-dominated splitting $\cN_{\supp(\mu)}=G^s\oplus G^c\oplus G^u$ (Remark \ref{rmk:fundamental-sequences}). Since $\ga_n$ is $(C,\eta,T,G^s_n)$-$\psi^{X_n}_t$-contracting at period, one has
\[\frac{1}{k_n}\sum_{i=0}^{k_n-1}\log\|\psi^{X_n}_{T}|_{G^s_n(\phi^{X_n}_{iT}(x))}\|\leq -\eta T+\frac{C'}{\tau_n}, \quad \forall x\in\ga_n,\]
where $\tau_n=\tau(\ga_n)$, $k_n=\lfloor\tau_n/T\rfloor$ and $C'>0$ is a constant that is uniform for all $n$. In other words, if we denote
\[\mu_{n,x}=\frac{1}{k_n}\sum_{i=0}^{k_n-1}\de_{\phi^{X_n}_{iT}(x)},\quad \xi_n^s=\log\|\psi^{X_n}_T|_{G^s_n}\|,\]
then one has $\mu_{n,x}(\xi_n^s)\leq -\eta T+O(1/\tau_n)$ for any $x\in\ga_n$. Letting $\mu_n=\frac{1}{T}\int_0^T(\phi^{X_n}_t)_*(\mu_{n,x})dt$ with $x\in\ga_n$, then
\[\int\xi_n^s d\mu_n=\frac{1}{T}\int_0^T(\phi^{X_n}_t)_*\mu_{n,x}(\xi_n^s)dt=\frac{1}{T}\int_0^T\mu_{n,\phi^{X_n}_t(x)}(\xi_n^s)dt\leq -\eta T+O(1/\tau_n).\]

Since $\|\psi^X_T(x)\|$ is uniformly bounded away from 0 and $\infty$, $\xi^s=\log\|\psi^X_T|_{G^s}\|$ is uniformly bounded, i.e., there exists $K_0>0$, such that $|\xi^s(x)|<K_0$ whenever it is defined. As $\|X_n-X\|_{C^1}\to 0$ and $G^s_n\to G^s$ as $n\to\infty$, we can assume that $|\xi^s_n(x)|=|\log\|\psi^{X_n}_T|_{G^s_n(x)}\||<K_0$ for all $n$ and for any $x\in\ga_n$.
It is easy to see that
\[\de_{\ga_n}=\frac{k_nT}{\tau_n}\mu_n+\frac{1}{\tau_n}\int_{-\vep_n}^0(\phi^{X_n}_t)_*\de_xdt,\]
where $\vep_n=\tau_n-k_nT<T$. It follows that
\begin{equation}\label{eqn:le-gs}
\int\xi^s_nd\de_{\ga_n}\leq -\eta T+O(1/\tau_n).
\end{equation}

Since $\mu$ is nontrivial, for any $\vep>0$, there is an open neighborhood $W$ of $\Sing(X)$, such that $\mu(\Cl(W))<\vep$. One can extend $\cN_{\supp(\mu)}=G^s\oplus G^c\oplus G^u$ in a continuous way to a neighborhood $U$ of $\supp(\mu)\setminus W$, so that $\xi^s$ is well-defined and continuous on $U$. Then $\xi^s_n$ converges uniformly to $\xi^s$ on $U$, and in particular, there is $N_0>0$ such that for any $n\geq N_0$, we have $|\xi^s_n(x)-\xi^s(x)|<\vep$ for all $x\in\ga_n\cap U$. Therefore,
\begin{align*}
\left|\int\xi^s_nd\de_{\ga_n}-\int\xi^sd\mu\right| &\leq\left|\int_{W}\xi^s_nd\de_{\ga_n}-\int_{W}\xi^sd\mu\right|+\left|\int_{U}\xi^s_nd\de_{\ga_n}-\int_{U}\xi^sd\mu\right|\\
&\leq K_0(\de_{\ga_n}(W)+\mu(W))+\left|\int_{U}(\xi^s_n-\xi^s)d\de_{\ga_n}\right|+\left|\int_{U}\xi^sd(\de_{\ga_n}-\mu)\right|\\
&\leq K_0(\de_{\ga_n}(W)+\vep)+\vep +\left|\int_{U}\xi^sd(\de_{\ga_n}-\mu)\right|.
\end{align*}

As $\de_{\ga_n}\to\mu$, there is $N_1\geq N_0$, such that for any $n\geq N_1$, one has $\de_{\ga_n}(W)< 2\vep$ and $|\int_{U}\xi^sd(\de_{\ga_n}-\mu)|<\vep$. Then for any $n\geq N_1$,
\[\left|\int\xi^s_nd\de_{\ga_n}-\int\xi^sd\mu\right|<(3K_0+2)\vep.\]
In view of inequality \eqref{eqn:le-gs}, one has
\[\int\xi^sd\mu=\lim_{n\to\infty}\int\xi^s_nd\de_{\ga_n}\leq-\eta T.\]
Let $\la^s$ be the largest Lyapunov exponent of $\mu$ corresponding to $G^s$. By definition of $\xi^s$, $\la^s\leq \frac{1}{T}\int\xi^sd\mu\leq -\eta$. Similarly, let $\la^u$ be the smallest Lyapunov exponent of $\mu$ corresponding to $G^u$. Then $\la^u\geq\eta$. Moreover, if $G^c$ is nontrivial, letting $\la^c$ be the Lyapunov exponent of $\mu$ corresponding to $G^c$, then $\la^c=\lim_{n}\la^c_n$, where $\la^c_n$ is the Lyapunov exponent of $\de_{\ga_n}$ corresponding to $G^c_n$. This concludes the first part of the lemma. Now if $\mu$ is non-hyperbolic, one can apply Franks' lemma (see {\it e.g.} \cite[Theorem A.1]{BGV}) to the fundamental sequence $\fF$ to obtain a fundamental $\ind(\mu)$-sequence and also a fundamental $(\ind(\mu)+1)$-sequence, with all other properties unchanged.
\end{proof}

Note that the constants $\eta$ and $T$ in the proof above are uniform for a neighborhood of $X$. From the proof of Lemma \ref{lem:le-away-tang}, one has the following corollary.

\begin{Corollary}\label{cor:le-away-tang}
For any $X\in\xX^1(M)\setminus\Cl(\mathcal{HT})$, let the neighborhood $\cU$ of $X$ and the constants $\eta>0$, $\de>0$, $T>0$ be given by Lemma \ref{lem:BasicAwayTang}. Then for any nontrivial ergodic measure $\mu$ of $Z\in\cU$,
\begin{itemize}
\item there is a fundamental sequence $(\ga_n,X_n)\to(\supp(\mu),X)$ admitting a $T$-dominated splitting $\cN_{\ga_n}=G^s_n\oplus G^c_n\oplus G^u_n$, 	which induces a $T$-dominated splitting $\cN_{\supp(\mu)}=G^s\oplus G^c\oplus G^u$ with respect to the linear Poincar\'e flow $\psi^Z_t$ such that $\dim G^c\leq 1$;
\item $\int\log\|\psi^Z_T|_{G^s}\|d\mu\leq -\eta T$ and $\int\log\|\psi^Z_{-T}|_{G^u}\|d\mu\leq -\eta T$;
\item let $\la^s$ (resp. $\la^u$) be the largest (resp. smallest) Lyapunov exponent of $\mu$ corresponding to $G^s$ (resp. $G^u$), then $\la^s\leq -\eta<\eta\leq\la^u$; consequently, $G^c$ is nontrivial if $\mu$ has a central Lyapunov exponent $\la^c\in(-\eta,\eta)$ other than the zero exponent corresponding to the flow direction, in which case, $\la^c$ is the Lyapunov exponent corresponding to $G^c$;
\item $G^c$ is trivial if $\mu$ has no Lyapunov exponent in $[-\de,\de]$ other than the zero exponent corresponding to the flow direction.
\end{itemize}
\end{Corollary}

\begin{Remark}\label{rmk:le-away-tang}
	Note that the neighborhood $\cU$ and the constants are given by Lemma \ref{lem:BasicAwayTang}. As remarked before (see Remark \ref{rmk:BasicAwayTang}), $\eta$ can be chosen arbitrarily small and $\de$ can also be chosen arbitrarily small (as long as $\eta$ is small enough).
\end{Remark}

\subsection{A lemma of Pliss type for ergodic measures}\label{s.7.2}
Let $X$ and $\mu$ be as in the previous section. Then there is a fundamental sequence $\{(\ga_n,X_n)\}$ with a dominated splitting $\cN_{\ga_n}=G^s_n\oplus G^c_n\oplus G^u_n$ that converges to a dominated splitting $\cN_{\supp(\mu)}=G^s\oplus G^c\oplus G^u$ over $\supp(\mu)$, as given by Corollary \ref{cor:le-away-tang}. Then one can consider hyperbolic points on $\ga_n$ with respect to $G^{s/u}_n$, and also hyperbolic points on $\supp(\mu)$ with respect to $G^{s/u}$. Following~\cite{ABV00}, we have the following lemma of Pliss type.

\begin{Lemma}\label{lem:pliss-for-measure}
For any $X\in\xX^1(M)\setminus\Cl(\mathcal{HT})$, there is a neighborhood $\cU$ of $X$ and positive numbers $\eta>0$, $T>0$ satisfying the following property: for any $\eta'\in (0,\eta)$, there is $\theta>0$ such that, for any nontrivial ergodic measure $\mu$ of $Z\in\cU$, $\mu(H^s(\eta'))\geq \theta$ and $\mu(H^u(\eta'))\geq \theta$ whenever the corresponding subbundle is nontrivial, where
\begin{itemize}
\item $\cN_{\supp(\mu)}=G^s\oplus G^c\oplus G^u$ is the dominated splitting given by Corollary \ref{cor:le-away-tang},
\item $H^s(\eta')$ is the set of $(1,\eta',T,G^s)$-$\psi^{Z,*}_t$-contracting points, and
\item $H^u(\eta')$ is the set of $(1,\eta',T,G^u)$-$\psi^{Z,*}_t$-expanding points.
\end{itemize}
\end{Lemma}
\begin{proof}
Let $\cU\ni X$, $C>0$, $\de>0$, $\eta>0$ and $T>0$ be given by Lemma \ref{lem:BasicAwayTang}. By the uniform continuity of $\psi^{Z,*}_t$ (see, for instance, \cite[Lemma 2.1]{SYY}), we have
\[C(\cU,T)=\sup\{\|\psi^{Z,*}_t\|: t\in[-T,T], Z\in\cU\}<\infty,\]
where $\|\psi^{Z,*}_t\|=\sup\{|\psi^{Z,*}_t(v)|: v\in\cN^Z, \txtrm[and] \|v\|=1\}$.

Let $C_0=\log C(\cU,T)$. By the Pliss lemma \cite{Man87,Pli}, for any $\eta'\in (0,\eta)$, there exist $N_0=N_0(\eta,\eta',C_0)>0$ and $\th=\th(\eta,\eta',C_0)>0$, such that for any $N$ real numbers $a_0,\, \cdots,\, a_{N-1}$, with $N\geq N_0$, satisfying
\begin{gather*}
\sum_{i=0}^{N-1}a_i\leq -N(\eta+\eta')T/2,\\
|a_i|\leq C_0,\quad n=0,\ldots,N-1,
\end{gather*}
there exist $l\geq\th N$, $0\leq i_1<\cdots<i_l\leq N-1$ such that
\begin{equation}\label{eqn:pliss-point}
\sum_{k=i_j}^{i-1}a_k\leq -(i-i_j)\eta'T,
\end{equation}
for every $j=1,\ldots,l$ and $i_j<i\leq N$.

From the proof of Lemma \ref{lem:le-away-tang} (and hence Corollary \ref{cor:le-away-tang}), the neighborhood $\cU$ and constants $\eta$, $T$ satisfy the properties in Corollary \ref{cor:le-away-tang}. Then for any $Z\in\cU$, any nontrivial ergodic measure $\mu$ of $Z$, let $\{(\ga_n,X_n)\}$ be a fundamental sequence with a dominated splitting $\cN_{\ga_n}=G^s_n\oplus G^c_n\oplus G^u_n$ that induces the dominated splitting $\cN_{\supp(\mu)}=G^s\oplus G^c\oplus G^u$ as in Corollary \ref{cor:le-away-tang}, $\ga_n\to\supp(\mu)$ and $\de_{\ga_n}\to\mu$. We assume $X_n\in\cU$ for all $n$. Also, we may assume that $\tau_n=k_n T/l_n$ for some positive integers $k_n$ and $l_n$. Since $\mu$ is nontrivial, one has $\tau_n\to\infty$, and it follows that $k_n\to\infty$ as $n\to\infty$.

In the following, we only prove the lemma when $G^s$ is nontrivial, as the other case can be proven similarly. Define $H^s_n(\eta')$, for each $n$, as the set of $(1,\eta',T,G^s_n)$-$\psi^{X_n,*}_t$-contracting points. Note that by Lemma \ref{lem:BasicAwayTang}, $\ga_n$ is $(C,\eta,T,G^{s}_n)$-$\psi^{X_n,*}_t$-contracting at period. Then for any $x\in\ga_n$, we have
\[\sum_{i=0}^{k_n-1}\log\|\psi^{X_n,*}_T|_{G^s(\phi^{X_n}_{iT}(x))}\|\leq -k_n\eta T+\log C.\]
Since $k_n\to\infty$, we can assume that
\[\sum_{i=0}^{k_n-1}\log\|\psi^{X_n,*}_T|_{G^s(\phi^{X_n}_{iT}(x))}\|\leq -k_n(\eta+\eta')T/2.\]
Note that $|\log \|\psi^{X_n,*}_T|_{G^s(\phi^{X_n}_{iT}(x))}\||\leq C_0$ for all $i$. By inequality \eqref{eqn:pliss-point} and a standard argument for diffeomorphisms, there exists $l\geq\theta k_n$, $0\leq i_1<\cdots<i_l\leq k_n-1$, such that $\phi^{X_n}_{i_lT}(x)$ are all $(1,\eta',T,G^s)$-$\psi^{X_n,*}_t$-contracting. In other words, for any $x\in\ga_n$,  $\mu_{n,x}(H^s_n(\eta'))\geq \theta$, where
\[\mu_{n,x}=\frac{1}{k_n}\sum_{i=0}^{k_n-1}\de_{\phi^{X_n}_{iT}(x)},\]
It follows that
\[\de_{\ga_n}(H^s_n(\eta'))=\frac{1}{T}\int_0^T(\phi^{X_n}_t)_*\mu_{n,x}(H^s_n(\eta'))dt=\frac{1}{T}\int_0^T\mu_{n,\phi^{X_n}_t(x)}(H^s_n(\eta'))dt\geq\theta.\]

Let $H^s_0=\cap_{m\geq 0}\Cl(\cup_{n\geq m}H^s_n(\eta'))$. Then $H^s_0\subset H^s(\eta')\cup\Sing(Z)$. Since $\de_{\ga_n}\to \mu$ in weak* topology, one can show that
\[\mu(H^s_0)=\lim_{m\to\infty}\mu(\Cl(\cup_{n\geq m}H^s_n(\eta')))\geq\lim_{m\to\infty}\limsup_{k\to\infty}\de_{\ga_k}(\Cl(\cup_{n\geq m}H^s_n(\eta')))\geq \theta.\]
Since $\mu$ is ergodic and nontrivial, it follows that $\mu(H^s(\eta'))\geq\mu(H^s_0)\geq\theta$.
\end{proof}

\begin{Remark}\label{rmk:pliss-for-measure}
Following the same argument, the conclusion of the lemma also holds for periodic measure over periodic orbits with large period.
\end{Remark}

\subsection{Hyperbolic measure implies homoclinic intersection}
The main purpose of this section is the following result.

\begin{Proposition}\label{prop:hyp-measure}
Let $X\in\mathscr{X}^1(M)\setminus\Cl(\mathcal{HT})$ and $\mu$ be a nontrivial hyperbolic ergodic measure for $\phi^X_t$,
then there is a hyperbolic periodic orbit $P$ of index $\ind(\mu)$, such that the homoclinic class $H(P)$ of $P$ is nontrivial and $\supp(\mu)\subset H(P)$.
\end{Proposition}

Similar results are already obtained for star vector fields \cite{SGW} and sectional-hyperbolic Lyapunov stable chain recurrence classes \cite{PYY}.
The following lemma concerning hyperbolic measures is contained in the proof of \cite[Theorem 5.6]{SGW}. See also~\cite[Lemma 4.5]{PYY}.
\begin{Lemma}\label{lem:hyp-meas}
Let $X\in\mathscr{X}^1(M)$, and $\mu$ be a nontrivial hyperbolic ergodic measure for $\phi^X_t$. Suppose there is a dominated splitting $\cN_{\supp(\mu)}=G^{cs}\oplus G^{cu}$ with respect to the linear Poincar\'{e} flow such that the following inequalities hold for some $\varepsilon>0$ and $T>0$,
\begin{gather}
  \int\log\|\psi^X_T|_{G^{cs}}\|d\mu<-\varepsilon,\label{eqn:hyp-1}\\
  \int\log\|\psi^X_{-T}|_{G^{cu}}\|d\mu<-\varepsilon,\label{eqn:hyp-2}
\end{gather}
then there is a hyperbolic periodic orbit $P$ of index $\ind(\mu)$, such that the homoclinic class $H(P)$ of $P$ is nontrivial and $\supp(\mu)\subset H(P)$.
\end{Lemma}

\begin{proof}[Proof of Proposition \ref{prop:hyp-measure}.]
By Corollary \ref{cor:le-away-tang}, there exist $\eta>0$ and $T>0$ such that $\supp(\mu)$ admits a $T$-dominated splitting $\cN_{\supp(\mu)}=G^s\oplus G^c\oplus G^u$ with
\begin{gather*}
\int\log\|\psi^X_T|_{G^s}\|d\mu\leq -\eta T,\\
\int\log\|\psi^X_{-T}|_{G^u}\|d\mu\leq -\eta T.
\end{gather*}
If $G^c$ is trivial, then the conclusion follows already from Lemma \ref{lem:hyp-meas}. Otherwise, since $\mu$ is hyperbolic, the central Lyapunov exponent $\la^c$ corresponding to $G^c$ is nonzero. Without loss of generality, let us assume that $\la^c>0$. Let $G^{cu}=G^c\oplus G^u$, then by domination,
\[\int\log\|\psi^X_{-T}|_{G^{cu}}\|d\mu=\int\log\|\psi^X_{-T}|_{G^c}\|d\mu=-\la^c T.\]
The conclusion again follows from Lemma \ref{lem:hyp-meas}.
\end{proof}

\begin{Corollary}\label{cor:hyp-measure}
Let $X\in\mathscr{X}^1(M)\setminus\Cl(\mathcal{HT})$ be weak Kupka-Smale and $C(\si)$ be a nontrivial chain recurrence class associated with $\si\in\Sing(X)$. If $\mu$ is a nontrivial hyperbolic ergodic measure supported on $C(\si)$ with $\ind(\mu)=\ind(\si)$, then there is a $C^1$ vector field $Y$ arbitrarily close to $X$ and a periodic orbit $\ga$, such that $W^u(\si_Y,Y)\pitchfork W^s(\ga,Y)\neq \emptyset$.
\end{Corollary}
\begin{proof}
By Proposition \ref{prop:hyp-measure}, there is a hyperbolic periodic orbit $P$ of $X$, such that $H(P)\subset C(\si)$ and $\ind(P)=\ind(\mu)$. Applying the connecting lemma for pseudo-orbits, there is a vector field $Y$ arbitrarily close to $X$ such that $W^u(\si_Y,Y)\cap W^s(P_Y,Y)\neq\emptyset$. Since $\dim W^{s}(P)+\dim W^u(\si)=(\ind(P)+1)+(\dim M-\ind(\mu))=\dim M+1$, the intersection can be made transverse by yet another arbitrarily small perturbation.
\end{proof}

\subsection{Properties for physical-like measures and Gibbs \texorpdfstring{$F$}{F}-states}\label{sect:physical-like-measure}

We will consider time-$T$ maps of flows, which are $C^1$ diffeomorphisms on $M$.
Let $f\in\diff^1(M)$ and $\mathfrak{M}(f)$ be the set of $f$-invariant Borel probability measures.
Given $x\in M$, let $\tau_f(x)$ be the set of limit points of the sequence $\{\frac{1}{n}\sum_{i=0}^{n-1}\de_{f^i(x)}\}_{n\in\NN}$ in the weak* topology.
It is well known that $\tau_f(x)\subset\mathfrak{M}(f)$ (see {\itshape e.g.} \cite{Man87}).
Following the work of Catsigeras and Enrich \cite{CE}, a measure $\mu\in\mathfrak{M}(f)$ is called {\em physical-like} if for any $\vep>0$, the set
\[B_{\vep}(\mu)=\{x\in M: d^*(\tau_f(x),\mu)<\vep\},\]
has positive Lebesgue measure, where $d^*:\mathfrak{M}\times \mathfrak{M}\to\RR$ is a metric that gives the weak* topology.
We denote the set of physical-like measures of $f$ by $\PhL(f)$.

We are mostly interested in attracting sets.
Let $U$ be any open subset of $M$ and suppose $f(\Cl(U))\subset U$. Then we can consider the dynamics of $f$ restricted to $U$. Let us denote by $\PhL(f,U)$ the set of physical-like measures of $f$ such that $\supp(\mu)\subset U$.

The diffeomorphism $f$ is said to admit a {\em dominated splitting over $U$} if there is an invariant splitting $T_UM=E\oplus F$ and constants $K>0$, $0<\la<1$, such that
\[\|Df^n|_{E(x)}\|\cdot\|Df^{-n}|_{F(f^n(x))}\|\leq K\la^n,\quad \forall n\geq 1, x, \cdots, f^n x\in U.\]
If such a dominated splitting exists, we say that a measure $\mu\in\mathfrak{M}(f)$ is a {\em Gibbs $F$-state} if $\supp(\mu)\subset U$ and
\[h_{\mu}(f)\geq \int \log|\det(Df|_{F(x)})|d\mu(x),\]
where $h_{\mu}(f)$ is the measure-theoretic entropy of $\mu$. The set of Gibbs $F$-states is denoted as $\Gibbs^F(f,U)$.  Table~\ref{table.1}, taken from~\cite{GYYZ}, summaries the properties of physical-like measures and Gibbs $F$-states.

\begin{table}
\centering
\renewcommand{\arraystretch}{1.4}
    \begin{tabular}{lcc}
    \toprule[1pt]
    {\hspace{2mm}}  {\hspace{2mm}} & {\hspace{1cm}} $\PhL(f,U)$ {\hspace{1cm}} &   $\Gibbs^F(f,U) $  \\\hline
    Existence & True & True \\
    Convexity & False & True    \\
    Compactness & True  & {\hspace{.5mm}} If $h_{ \boldsymbol{\cdot}}(f)$ is upper semi-continuous {\hspace{.5mm}} \\
    Semi-continuity &  Theorem~\ref{thm:semi-conti}  & If $h_\mu( \boldsymbol{\cdot})$ is upper semi-continuous \\
    \bottomrule[1pt]\\
    \end{tabular}
    \caption{Properties of physical-like measures and Gibbs $F$-states}\label{tab:phl-gibbs}
    \label{table.1}
\end{table}

\begin{Remark}\label{rmk:gibbs-pef}
	If for some $\mu\in \Gibbs^F(f,U)$ it satisfies all Lyapunov exponents along the $F$ bundle (at almost every point) are non-negative and all other Lyapunov exponents are negative, then combined with Ruelle's inequality \cite{Ru}, one can see that $\mu$ satisfies {\em Pesin's entropy formula}:
\[h_{\mu}(f)=\int \sum_i \max\{\la_i(x),0\}d\mu(x),\]
where $\{\la_i(x)\}_{i=1}^{\dim M}$ are the Lyapunov exponents (counting multiplicity) given by Oseledets' theorem \cite{Ose68}.
\end{Remark}

The following result is obtained in \cite{CE,CCE}, see also \cite{GYYZ}:
\begin{Theorem}\label{thm:physical-like}
Given $f\in\diff^1(M)$ and $U\subset M$ an open subset such that $f(\Cl(U))\subset U$, the set of physical-like measures $\PhL(f,U)$ is a nonempty compact subset of $\mathfrak{M}(f)$. Moreover, suppose $f$ admits a dominated splitting $T_UM=E\oplus F$ over $U$, then one has $\PhL(f,U)\subset\Gibbs^F(f,U)$.
\end{Theorem}

As a simple consequence, one has

\begin{Corollary}\label{cor:physical-like}
Every quasi-attractor of a $C^1$ diffeomorphism supports a physical-like measure.
\end{Corollary}

The following result gives upper semi-continuity for the physical-like measures, which is crucial for our approximation argument using $C^2$ vector fields.
\begin{Theorem}[\cite{GYYZ}]\label{thm:semi-conti}
Let $f\in \diff^1(M)$ and $U\subset M$ be an open subset with $f(\Cl(U))\subset U$. Assume that $f$ admits a dominated splitting $T_UM=E\oplus F$ over $U$. Assume that $\{f_n\}$ is a sequence of $C^1$ diffeomorphisms converging to $f$ in the $C^1$ topology and $\mu_n$ a physical-like measure of $f_n$ with $\supp(\mu_n)\subset U$.
Then any weak* limit of $\mu_n$ is a Gibbs $F$-state of $f$.
\end{Theorem}
\begin{Remark}
	The semi-continuity of $\Gibbs^F(f,U)$ depends on the continuity of the metric entropy as a function of the diffeomorphism, which holds for diffeomorphisms away from tangencies~\cite{LVY}. However, such result is lacking for singular flows due to the possibility of entropy vanishing near singularities. See~\cite[Theorem G]{SYY} for some related discussion. To overcome this, we use
	Theorem \ref{thm:semi-conti} which establishes certain semi-continuity of the subset $\PhL(f,U)\subset\Gibbs^F(f,U)$.
\end{Remark}

\subsection{A characterization of measures satisfying Pesin's entropy formula}

Now suppose $f$ is a $C^2$ diffeomorphism on $M$ and $\mu\in\mathfrak{M}(f)$. For $\mu$-a.e. $x$, let us denote by $\chi_1(x)<\chi_2(x)<\cdots<\chi_{r(x)}(x)$ the distinctive Lyapunov exponents at $x$ given by Oseledets' theorem \cite{Ose68}. Let $T_xM=E_1(x)\oplus\cdots\oplus E_{r(x)}(x)$ be the corresponding splitting such that for every non-zero $v\in E_i(x)$, it satisfies
\[\lim_{n\to\pm\infty}\frac{1}{n}\log\|Df^n_x(v)\|=\chi_i(x).\]
Let us define
\[E^u(x)=\oplus_{\chi_i(x)>0}E_i(x).\]
Then there is a $C^2$ immersed submanifold $W^u(x)$, called the unstable manifold, tangent to $E^u(x)$ at $x$ (see \cite{BaP07}). Moreover, the unstable manifold $W^u(x)$ is characterized as follows:
\[W^u(x)=\{y\in M: \limsup_{n\to\infty}\frac{1}{n}\log d(f^{-n}(x),f^{-n}(y))<0\}.\]

Given any immersed submanifold $W\subset M$, denote by $m_W$ the Lebesgue measure on $W$ induced by the Riemannian metric of $M$. Let $\xi$ be any $\mu$-measurable partition of $M$. Then for $\mu$-a.e. $x$, there is a probability measure $\mu_x^{\xi}$ defined on $\xi(x)$, the element of $\xi$ containing $x$. The partition $\xi$ is said to be {\em subordinate to $W^u$} if for $\mu$-a.e. $x$, it satisfies that $\xi(x)\subset W^u(x)$ and $\xi(x)$ contains a neighborhood of $x$, open in the submanifold topology of $W^u(x)$. We say that $\mu$ has {\em $u$-absolutely continuous conditional measures} if for every $\mu$-measurable partition $\xi$ subordinate to $W^u$, $\mu_x^{\xi}$ is absolutely continuous with respect to $m^u_x=m_{W^u(x)}$ for $\mu$-a.e. $x$.

\begin{Theorem}
  [\cite{LeY85}, Theorem A and Corollary 6.1.4]\label{thm:LeY}
  Let $f:M\to M$ be a $C^2$ diffeomorphism, $\mu\in\mathfrak{M}(f)$. Then $\mu$ has $u$-absolutely continuous conditional measures if and only if it satisfies Pesin's entropy formula. Moreover, when $\mu$ satisfies Pesin's entropy formula, there is a $\mu$-measurable partition $\xi$ subordinate to $W^u$ such that for $\mu$-a.e. $x$, the density function $d\mu^{\xi}_x/dm^u_x$ is a strictly positive $C^1$ function on $\xi(x)$.
\end{Theorem}

\section{Proof of the main proposition}
\label{sect:proof-of-main-prop}

In this section, we prove Proposition \ref{prop:main-proposition}.
Throughout this section, let $X\in\xX^1(M)\setminus\Cl(\mathcal{HT})$ be a $C^1$ generic vector field, $\si\in\Sing(X)$, and $C(\si)$ a nontrivial Lyapunov stable chain recurrence class.
Assume that there is a partially hyperbolic splitting $T_{C(\si)}M = E^{ss}\oplus F$, such that $E^{ss}$ is contracting and $\dim E^{ss}=\ind(\si)-1$.

By Proposition \ref{prop:homogeneous-index}, all singularities in $C(\sigma)$ are Lorenz-like. The following theorem proves Proposition \ref{prop:main-proposition} in the case $\dim F=2$.
\begin{Theorem}[{\cite[Theorem 1.5]{GXYZ}}]\label{thm.gxyz}
  Let $X\in\xX^1(M)$ be a $C^1$ generic vector field. Suppose $C(\sigma)$ is a non-trivial chain recurrence class of a singularity $\sigma$ satisfying the following conditions:
  \begin{itemize}
    \item all singularities in $C(\sigma)$ are Lorenz-like;
    \item there is a partially hyperbolic splitting $T_{C(\sigma)}M=E^{s}\oplus F$ with respect to the tangent flow, where $E^{s}$ is contracting and $\dim F=2$.
  \end{itemize}
  Then $C(\sigma)$ is a homoclinic class of a periodic orbit.
\end{Theorem}

Therefore, in the rest of this section, we will assume
$\dim F\geq 3$, which implies $\dim E^u(\si)\geq 2$.

\subsection{Idea of the proof}\label{sect:idea-of-pf}

As discussed in Section \ref{sect:first-reduction}, we will use the mechanism (Section \ref{sect:matching}) developed in \cite{LGW05} for the proof of our main proposition. Let us explain how this can be done. Suppose there is a fundamental sequence $\fF=\{(\ga_n,X_n)\}$ admitting a dominated splitting $\cN_{\ga_n}=G^{cs}_n\oplus G^{cu}_n$ of index $i=\ind(\si)$, such that $(\ga_n,X_n)\to (C(\si),X)$, then the following results hold:
\begin{itemize}
\item $B(C(\si))\subset \De(\fF)$, as $C(\si)$ is the limit of $\fF$.
\item There is a dominated splitting $\wt{\cN}_{\De(\fF)}=G^{cs}\oplus G^{cu}$ with respect to the extended linear Poincar\'e flow, and $\dim G^{cs}=\ind(\si)$ (see Remark \ref{rmk:fundamental-sequences}).
\item There is a dominated splitting $E^u(\si)=E^c(\si)\oplus E^{uu}(\si)$ with respect to the tangent flow with $\dim E^c(\si)=1$ (see Remark \ref{rmk:elp-dom}).
\item Since $C(\si)$ is nontrivial and Lyapunov stable, it follows from Lemma \ref{lem:gen_1} that $B(C(\si))\cap (E^s(\si)\oplus E^{uu}(\si))\setminus (E^s(\si)\cup E^{uu}(\si))\neq\emptyset$. By the first item above, $B(C(\si))\subset \De(\fF)$, hence
\[\De(\fF)\cap (E^s(\si)\oplus E^{uu}(\si))\setminus (E^s(\si)\cup E^{uu}(\si))\neq\emptyset,\]
which is a contradiction to Lemma \ref{lem:elp-dom2} (see also Remark \ref{rmk:elp-dom}), or to Lemma \ref{lem:matching}.
\end{itemize}

Therefore, for the proof of Proposition \ref{prop:main-proposition}, it is sufficient to obtain such a fundamental sequence assuming that $C(\si)$ is aperiodic. Note that in this argument, a key requirement is that the limit of the fundamental sequence is the entire chain recurrence class $C(\si)$.

We remark that, assuming $C(\si)$ is aperiodic, it is not difficult to show the existence of a fundamental sequence $\fF_0$ of $X$ admitting a dominated splitting of index $i=\ind(\si)$ and having a limit $\La\subset C(\si)$ (but possibly a proper subset). In fact, assuming $C(\si)$ is aperiodic, it can not be {\em locally star}\footnote{A chain recurrence class $\La$ of $X$ is called {\em locally star} if there is a neighborhood $\cU$ of $X$ and a neighborhood $U$ of $\La$ such that for any $Y\in\cU$, any periodic orbit of $Y$ in $U$ is hyperbolic.}. Otherwise, since all singularities in $C(\si)$ are Lorenz-like and have the same index (Proposition \ref{prop:homogeneous-index}), then \cite[Theorem 3.7]{SGW} implies that $C(\si)$ is singular-hyperbolic, and by \cite[Corollary E]{PYY}, it is a homoclinic class, which is a contradiction. As $X$ is not locally star, there exists a fundamental sequence $\fF_0=\{(\ga_n,X_n)\}$ of $X$ with a limit $\La\subset C(\si)$, such that $\ga_n$ is non-hyperbolic for all $n$. By partial hyperbolicity and Proposition \ref{prop:indices-of-periodic-orbits-1}, we have $\ind(\ga_n)=\ind(\si)-1$. By Lemma \ref{lem:BasicAwayTang}, $\fF_0$ admits a dominated splitting of index $i=\ind(\si)$. Moreover, by Proposition \ref{prop:non-singular-set}, $\La$ contains a singularity, say $\si$. However, the problem is that we have no control over how the fundamental sequence approaches $\si$. More precisely, we can not be sure that the sequence will step into the ``restricted area'' $(E^s(\si)\oplus E^{uu}(\si))\setminus (E^s(\si)\cup E^{uu}(\si))$ and as such, the mechanism of \cite{LGW05} (Lemma \ref{lem:elp-dom2}) can hardly be applied.

The key observation is that, since $X$ is $C^1$ generic and $C(\si)$ is Lyapunov stable, from Lemma \ref{lem:gen_1} we know that there is a residual subset $R$ of the local unstable manifold $W^u_{loc}(\si)$ such that for any $x\in R$, the forward orbit of $x$ is dense in $C(\si)$. Moreover, if there is a dominated splitting $E^u(\si)=E^c(\si)\oplus E^{uu}(\si)$, there is also a residual subset $R'$ of the local strong unstable manifold $W^{uu}_{loc}(\si)$ such that for any $x\in R'$, $\orb^+(x)$ is dense in $C(\si)$. In particular, if the limit $\La$ of a fundamental sequence of $X$ contains an open set in $W^{u}_{loc}(\si)$ (or contains an open set in $W^{uu}_{loc}(\si)$), then $\La\supset C(\si)$.

To obtain a fundamental sequence of $X$ such that its limit contains an open subset of $W^u_{loc}(\si)$ (or $W^{uu}_{loc}(\si)$), we use an argument involving $C^2$ vector fields and physical-like measures.
Let us be more precise. Assume that $C(\si)$ is aperiodic. We consider a sequence of $C^2$ vector fields $\{X_n\}$ that converges to $X$ in the $C^1$ topology. For each $X_n$, we show that the continuation of $C(\si)$ for $X_n$ supports a nontrivial physical-like measure $\mu_n$ that satisfies Pesin's entropy formula. Since $X_n$ is $C^2$, Theorem \ref{thm:LeY} implies that the measure $\mu_n$ is ``$u$-saturated'' in the sense that for $\mu_n$-a.e. $x$, the unstable manifold at $x$ is contained in the support of $\mu_n$. We will make use of this property of $\mu_n$ and show that the Hausdorff limit of $\supp(\mu_n)$ contains an open set in $W^{u}_{loc}(\si)$ (or an open set in $W^{uu}_{loc}(\si)$), and hence contains $C(\si)$. Moreover, we show that each $\supp(\mu_n)$ can be approximated by periodic orbits admitting a dominated splitting of index $i=\ind(\si)$ (uniform in $n$) with respect to the linear Poincar\'e flow\footnote{In fact, we will consider a measure $\mu'_n$ that is the restriction of $\mu_n$ to a $\phi^{X_n}_t$-invariant set $\La_n$ with $\mu_n(\La_n)\to 1$ as $n\to\infty$. See Section \ref{sect:hyperbolic-pts}.}. It follows that each $\supp(\mu_n)$ and hence $C(\si)$ admit a dominated splitting of index $i=\ind(\si)$ with respect to the linear Poincar\'e flow.
Then a contradiction will be obtained by applying the mechanism of \cite{LGW05} (Lemma \ref{lem:matching}).

\subsection{Properties for invariant measures}\label{sect:inv-measures}

Since $X$ is away from homoclinic tangencies, there is an open neighborhood $\cU$ of $X$ and numbers $C>0$, $\eta>0$, $\de>0$, $T>0$ given by Lemma \ref{lem:BasicAwayTang}, such that any periodic orbit $\ga$ of any vector field $Z\in \cU$ with $\tau(\ga)\geq T$ admits a $T$-dominated splitting $\cN_{\ga}=G^s\oplus G^c\oplus G^u$ with respect to $\psi^Z_t$, and satisfies other properties listed there.

For simplicity, let us assume that $T=1$.
Moreover, we assume that $\cU$ is small enough so that the following properties are satisfies for every vector field $Z$ in $\cU$:
\begin{enumerate}[(P1)]
\item the continuation of every singularity $\rho\in C(\si)$ is well-defined and Lorenz-like with the same index as that of $\rho$, which is equal to $\ind(\si)$ (Proposition \ref{prop:homogeneous-index}); \label{P:lorenz-like}
\item the continuation of $C(\si)$ is well-defined and is close to $C(\si)$ in the Hausdorff topology, i.e., $C(\si_Z,Z)\to C(\si)$ as $\|Z-X\|_{C^1}\to 0$ (item \ref{item:gen_0-continuity-chain-class} of Lemma \ref{lem:gen_0}); \label{P:continuity}
\item there is a partially hyperbolic splitting $T_{C(\si_Z,Z)}M = E^{ss}_Z\oplus F_Z$, $\dim E^{ss}_Z=\ind(\si)-1$. \label{P:ph}
\item there is an open neighborhood $U$ of $C(\si)$ (independent of $Z$) such that $\phi^Z_1(\Cl(U))\subset U$ and every singularity in $U$ is a continuation of some singularity in $C(\si)$; moreover, the partially hyperbolic splitting $T_{C(\si_Z,Z)}M = E^{ss}_Z\oplus F_Z$ can be extended to $U$ so that $\phi^Z_1$ admits a partially hyperbolic splitting over $U$, {\it i.e.} a dominated splitting $T_UM= E^{ss}_Z\oplus F_Z$ such that $E^{ss}_Z$ is uniformly contracting (\cite[Theorem 6]{BGW07}). \label{P:ph-nbhd}
\end{enumerate}

For any $Z\in\cU$ and any invariant measure $\mu$ of $\phi^Z_1$, let $\la_1(\mu,Z)\leq\la_2(\mu,Z)\leq\cdots\leq\la_{\dim M}(\mu,Z)$ be the Lyapunov exponents of $\mu$ for the diffeomorphism $\phi^Z_1$. That is,
\begin{equation}\label{eq:int-le}
\la_i(\mu,Z)=\int \la_i(x,\phi^Z_1)d\mu(x),
\end{equation}
where $\la_1(x,\phi^Z_1)\leq \cdots\leq\la_i(x,\phi^Z_1)\leq\cdots\leq\la_{\dim M}(x,\phi^Z_1)$ are the Lyapunov exponents at $x$, given by Oseledets' Theorem \cite{Ose68}.
Assume that $\supp(\mu)\subset U$, then by partial hyperbolicity \ref{P:ph-nbhd}, there exists $c_0<0$ such that $\la_{\ind(\si)-1}(\mu,Z)<c_0<0$. We will call $\la_{\ind(\si)+1}(\mu,Z)$ the {\em central exponent} and denote it by $\la^c(\mu,Z)$, {\it i.e.}
\[\la^c(\mu,Z) = \la_{\ind(\si)+1}(\mu,Z).\]
Correspondingly, we also denote $\la^c(x,\phi^Z_1)=\la_{\ind(\si)+1}(x,\phi^Z_1)$ and call it the central exponent at $x$.
The property of the central exponent will be clear shortly in the following sections. But first of all, we remark that for the Dirac measure $\de_{\rho}$ supported at a singularity $\rho\in C(\si_Z,Z)$, one has $\la^c(\de_{\rho},Z)>0$ and $\la_{\ind(\si)}(\de_{\rho},Z)+\la^c(\de_{\rho},Z)=\sv(\rho)>0$ because by \ref{P:lorenz-like} the singularity $\rho$ is Lorenz-like and $\ind(\rho)=\ind(\si)$.
\subsubsection{The central exponent is non-negative}
For ergodic measures, we have the following result.
\begin{Lemma}\label{lem:ergodic-measure}
	For $\si\in\Sing(X)$, if $C(\si)$ is aperiodic, then there exists a $C^1$ neighborhood $\cU_1\subset\cU$ of $X$ with the following property: for any weak Kupka-Smale $Z\in\cU_1$ and any ergodic measure $\mu$ of $\phi^Z_1$ with $\supp(\mu)\subset C(\si_Z,Z)$ and $\mu(\Sing(Z))=0$, it satisfies $\la^c(\mu,Z)\geq \la_{\ind(\si)}(\mu,Z)=0$ and $\la_j(\mu,Z)>0$ for $j=\ind(\si)+2,\ldots,\dim M$.
\end{Lemma}
\begin{proof}
Firstly, note that by partial hyperbolicity \ref{P:ph} one has $\la_i(\mu,Z)<0$ for all $1\leq i\leq\ind(\si)-1$. We show that $\la_{\ind(\si)}(\mu,Z)$ is zero. Then following from Lemma \ref{lem:le-away-tang} we must have $\la_j(\mu,Z)>0$ for $j=\ind(\si)+2,\ldots,\dim M$.

Suppose the conclusion of the lemma does not hold, {\it i.e.} there exists a weak Kupka-Smale vector field $Z\in\cU$ arbitrarily $C^1$ close to $X$ with an ergodic measure $\mu$ such that $\supp(\mu)\subset C(\si_Z,Z)$ and $\la_{\ind(\si)}(\mu,Z)<0$. If $\mu$ is supported on a periodic orbit, say $\ga'$, Proposition \ref{prop:indices-of-periodic-orbits-1} implies that $\ga'$ is hyperbolic and $\ind(\ga')=\ind(\si)$ (assuming $Z$ is close enough to $X$). In other words, $C(\si_Z, Z)$ contains a hyperbolic periodic orbit of index $\ind(\si)$.
	If $\mu$ is nontrivial, Proposition \ref{prop:indices-of-periodic-orbits-1} and Lemma \ref{lem:le-away-tang} imply that $\mu$ is hyperbolic and $\ind(\mu)=\ind(\si)$. Then by Proposition \ref{prop:hyp-measure}, it is also true that $C(\si_Z,Z)$ contains a hyperbolic periodic orbit of index $\ind(\si)$. Since $Z$ is weak Kupka-Smale, one can apply the connecting lemma for pseudo-orbits \cite{BoC} to obtain an intersection between the unstable manifold of the singularity $\si_Z$ and the stable manifold of the periodic orbit. Moreover, the intersection can be made transverse by an arbitrarily small perturbation as the periodic orbit has index $\ind(\si)$. Since $X$ is generic and $Z$ can be chosen arbitrarily close to $X$, by item \ref{item:gen_0-transverse-intersection} of Lemma \ref{lem:gen_0} there is a periodic orbit $\ga$ of $X$ such that $W^s(\ga,X)\pitchfork W^u(\si,X)\neq\emptyset$. As $C(\si)$ is Lyapunov stable, the chain recurrence class $C(\si)$ contains the periodic orbit $\ga$, a contradiction to the aperiodicity of $C(\si)$.
\end{proof}
\subsubsection{The central exponent is small}\label{sect:small-central-le}

\begin{Lemma}\label{lem:invariant-measure}
	Assuming $C(\si)$ is aperiodic. Let $\{X_n\}$ be a sequence of weak Kupka-Smale vector fields in $\cU$ that converges to $X$ in the $C^1$ topology. Let $\{\mu_n\}$ be a sequence of probability measures such that
	\begin{itemize}
		\item $\mu_n$ is $\phi^{X_n}_1$-invariant (not necessarily ergodic) and $\supp(\mu_n)\subset C(\si_{X_n},X_n)$;
		\item $\mu_n$ converges to a $\phi^X_1$-invariant measure $\mu$ in the weak* topology;
		\item $\mu_n(\Sing(X_n))=\mu(\Sing(X))=0$.
	\end{itemize}
	Then $\la^c(\mu_n,X_n)\to \la^c(\mu,X)=0$ as $n\to \infty$. Consequently, there exists a $\phi^{X_n}_t$-invariant set $\La_n\subset\supp(\mu_n)$ for each $n$ such that $\mu_n(\La_n)\to 1$ and the upper-bound of $\la^c(x,\phi^{X_n}_1)$ on $\Lambda_n$ goes to zero, as $n\to\infty$.
\end{Lemma}
\begin{proof}
	By shrinking $\cU$ if necessary, we assume that $\|\Phi^Z_1\|_{C^1}\leq K$ for all $Z\in \cU$. In particular, the maximal Lyapunov exponent is uniformly bounded above by $\log K$.
	
	By Proposition \ref{prop:hyp-measure}, every nontrivial ergodic measure supported on $C(\si)$ is non-hyperbolic, {\itshape i.e.} it has exactly two zero Lyapunov exponent, one of which corresponds to the flow direction. Hence by Lemma \ref{lem:ergodic-measure} and ergodic decomposition \cite{Wa} we must have $\la^c(\mu,X)=0$. We need only show that $\la^c(\mu_n,X_n)\to 0$ as $n\to \infty$.
	
	Suppose this is not true, then combined with the result of Lemma \ref{lem:ergodic-measure}, there exists $\de_0>0$ such that the inequality $\la^c(\mu_n,X_n)>\de_0$ is satisfied by infinitely many $n$'s. Without loss of generality, we can assume that $\la^c(\mu_n,X_n)>\de_0$ for all $n$.
	Note that Lyapunov exponents are uniformly bounded above by $\log K$, equation \eqref{eq:int-le} implies the existence of a $\phi^{X_n}_1$-invariant set $A_n\subset\supp(\mu_n)$ such that $\mu_n(A_n)\geq \de_0/(2\log K)$ and $\la^c(x,\phi^{X_n}_1)\geq\de_0/2$ for all $x\in A_n$. Let $\mu'_n=\mu_n|_{A_n}$, {\it i.e.} $\mu'_n$ is defined such that for every measurable set $A$ it satisfies
	\[\mu'_n(A)=\frac{1}{\mu_n(A_n)}\mu_n(A_n\cap A).\]

	As in Remark \ref{rmk:BasicAwayTang}, we can choose $\de_1\in(0,\de_0/2)$ with $\eta_1>0$ small enough such that Lemma \ref{lem:BasicAwayTang} still holds after replacing $\de$, $\eta$ with $\de_1$, $\eta_1$ respectively. Then by Corollary \ref{cor:le-away-tang}, Lemma \ref{lem:pliss-for-measure} and ergodic decomposition, there exists a dominated splitting $\cN_{\supp(\mu'_n)}=G^s_n\oplus G^u_n$ and a constant $\theta>0$ such that for $n$ large enough one has
	\begin{equation}\label{eq:mu-prime-3}
		\mu'_n(H^u_n(\eta_1/2))\geq\theta,
	\end{equation}
	where $H^u_n(\eta_1/2)$ is the set of $(1,\eta_1/2,1,G^u_n)$-$\psi^{X_n,*}_t$-expanding points. Note that $G^s_n$ is uniformly contracting by partial hyperbolicity \ref{P:ph}. Let us assume without loss of generality that $\supp(\mu'_n)$ converges in the Hausdorff topology to a compact $\phi^X_t$-invariant set $A\subset\supp(\mu)$. The splitting $\cN_{\supp(\mu'_n)}=G^s_n\oplus G^u_n$ then induces a dominated splitting $\cN_{A}=G^s\oplus G^u$ such that $G^s$ is uniformly contracting.
	
	Let $\th'=(\de_0\th)/(2\log K)$. From inequality \eqref{eq:mu-prime-3}, one obtains
	\[\mu_n(H^u_n(\eta_1/2))\geq\th'.\]
	Since $\mu_n(\Sing(X_n))=\mu(\Sing(X))=0$, there exists an open neighborhood $U_0$ of $\Sing(X)$ such that $\mu(U_0)<\theta'/2$ and $\mu_n(U_0)<\theta'/2$ for all $n$. Note that the sets $H^u_n(\eta_1/2)\setminus U_0$ can be assumed to be compact for all $n$. By choosing a subsequence if necessary, we assume that $H^u_n(\eta_1/2)\setminus U_0$ converges in the Hausdorff topology to a compact set $H$. Since $\mu_n\to\mu$ in the weak* topology, we have
	\begin{equation}\label{eq:h-pts}
		\mu(H)=\lim_{n\to\infty}\mu(\Cl(\cup_{m\geq n} H^u_m(\eta_1/2))\setminus U_0) \geq \limsup_{n\to\infty}\mu_n(H^u_n(\eta_1/2)\setminus U_0)\geq\frac{\theta'}{2}.
	\end{equation}
	Since any limit of $(1,\eta_1/2,T,G^u_n)$-$\psi^{X_n,*}_t$-expanding points is $(1,\eta_1/2,T,G^u)$-$\psi^{X,*}_t$-expanding (if it is not a singularity), the set $H$ is comprised of  $(1,\eta_1/2,T,G^u)$-$\psi^{X,*}_t$-expanding points. Then inequality \eqref{eq:h-pts} implies that $\mu$ has a nontrivial hyperbolic ergodic component. It follows from Proposition \ref{prop:hyp-measure} that $C(\si)$ contains periodic orbits, contradicting to the assumption that $C(\si)$ is aperiodic.
	
	Hence we have shown that
	\[\la^c(\mu_n,X_n)=\int\la^c(x,\phi^{X_n}_1)d\mu_n\to 0, \quad \text{as}\ n\to \infty.\]
	By Lemma \ref{lem:ergodic-measure}, $\la^c(x,\phi^{X_n}_1)$ is non-negative for $\mu_n$-a.e. $x$ and for all $n$. The previous equation implies, in particular, the existence of an invariant set $\La_n\subset\supp(\mu_n)$ for each $n$ such that $\mu_n(\La_n)\to 1$ and the upper-bound of $\la^c(x,\phi^{X_n}_1)$ on $\Lambda_n$ goes to zero, as $n\to\infty$.
\end{proof}
\subsubsection{Gibbs $F$-states satisfy Pesin's entropy formula}

By property \ref{P:ph-nbhd}, we can consider for every $Z\in\cU$ the set of Gibbs $F$-states $\Gibbs^F(\phi^Z_1,U)$. By Theorem \ref{thm:physical-like}, $\Gibbs^F(\phi^Z_1,U)$ contains the set $\PhL(\phi^Z_1,U)$ of physical-like measures, which is a nonempty compact subset of $\mathfrak{M}(\phi^Z_1)$.

\begin{Lemma}\label{lem:gibbs-0}
	For any $Z\in\cU$ and any $\mu\in\Gibbs^F(\phi^Z_1,U)$, it satisfies $\mu(\Sing(Z))=0$.
\end{Lemma}
\begin{proof}
	Let $Z\in\cU$ and $\mu\in\Gibbs^F(\phi^Z_1,U)$.
	Suppose $\mu=a\nu + \sum b_i\de_i$, where $0\leq a, b_i\leq 1$, $a+\sum b_i=1$, $\nu(\Sing(Z))=0$ and $\{\de_i\}$ is a finite set of Dirac measures supported on singularities in $U$. Note that by Lemma \ref{lem:ergodic-measure} and ergodic decomposition, the Lyapunov exponents of $\nu$ along $F$ bundle is non-negative. Moreover, by partial hyperbolicity, the Lyapunov exponents of $\nu$ along $E^{ss}$ bundle are all negative. Then by Ruelle's inequality and since the metric entropy is affine in measure, we have
	\begin{align*}
		h_{\mu}(\phi^Z_1)&=ah_{\nu}(\phi^Z_1)+ \sum b_i h_{\de_i}(\phi^Z_1)\\
			&\leq a \int\log|\det(\Phi^Z_1|_{F})|d\nu.
	\end{align*}
	Hence by the definition of Gibbs $F$-states and the previous inequality,
	\begin{align}
		\sum b_i \int \log|\det(\Phi^Z_1|_{F})|d\de_i
			&= \int \log|\det(\Phi^Z_1|_{F})|d\mu - a \int\log|\det(\Phi^Z_1|_{F})|d\nu \notag \\
			&\leq h_{\mu}(\phi^Z_1) - a \int\log|\det(\Phi^Z_1|_{F})|d\nu  \notag \\
			&\leq 0. \label{eq:gibbs}
	\end{align}
	Note that by \ref{P:lorenz-like} and \ref{P:ph-nbhd}, singularities in $U$ are Lorenz-like and have index equal to $\ind(\si)$. This implies $\int \log|\det(\Phi^Z_1|_{F})|d\de_i > 0$ for all $\de_i$. Hence \eqref{eq:gibbs} holds if and only if $b_i=0$ for all $i$. Then it follows that $\mu(\Sing(Z))=\nu(\Sing(Z))=0$.
\end{proof}

\begin{Lemma}\label{lem:gibbs-1}
	Assume that $C(\si)$ is aperiodic. Let $Z\in\cU$ be a weak Kupka-Smale vector field and $\mu$ a Gibbs $F$-state of $\phi^Z_1$ such that $\supp(\mu)\subset C(\si_Z,Z)$, then $\mu$ satisfies Pesin's entropy formula.
	Moreover, let $\mu=\int_{\cM^{erg}(\phi^Z_1)} \nu d\tau(\nu)$ be the ergodic decomposition of $\mu$, where $\cM^{erg}(\phi^Z_1)$ denotes the set of ergodic measures for $\phi^Z_1$. Then $\tau$-a.e. $\nu$ satisfies Pesin's entropy formula.
\end{Lemma}
\begin{proof}
	By Lemma \ref{lem:gibbs-0}, we have $\mu(\Sing(Z))=0$. Then by Lemma \ref{lem:ergodic-measure}, all Lyapunov exponents of $\mu$ along $F$ bundle are non-negative. Also, by partial hyperbolicity \ref{P:ph}, all Lyapunov exponents of $\mu$ along $E^{ss}$ bundle are negative. Then as in Remark \ref{rmk:gibbs-pef}, the Gibbs $F$-state $\mu$ satisfies Pesin's entropy formula.
	 {\it i.e.}
	\[h_{\mu}(\phi^Z_1)=\int\log|\det(\Phi^Z_1|_F)d\mu =\int \sum_i\max\{\la_i,0\}d\mu.\]
	For the ergodic decomposition of $\mu$, since the metric entropy is affine in measure, one obtains
	\begin{equation}\label{eq:ergodic-decomposition}
		\int_{\cM^{erg}(\phi^Z_1)}\left(h_{\nu}(\phi^Z_1)-\sum_i\max\{\la_i(\nu,Z),0\}\right)d\tau(\nu)
		= h_{\mu}(\phi^Z_1)-\int \sum_i\max\{\la_i(\nu,Z),0\}d\mu=0,
	\end{equation}
	where $\{\la_i(\nu,Z)\}$ are the Lyapunov exponents of $\nu$. By Ruelle's inequality, one has
	$h_{\nu}(\phi^Z_1)-\sum_i\max\{\la_i(\nu,Z),0\}\leq0$  for every $\nu$.
	Hence by \eqref{eq:ergodic-decomposition} one must have
	\[h_{\nu}(\phi^Z_1)-\sum_i\max\{\la_i(\nu,Z),0\}=0\]
	for $\tau$-a.e. $\nu$. This concludes the proof.
\end{proof}

\subsection{Approximation with \texorpdfstring{$C^2$}{C\^2} vector fields}

We shall prove the general case of Proposition \ref{prop:main-proposition} by contradiction. Thus we assume from now on the following
\begin{itemize}
	\item $C(\si)$ is aperiodic.
\end{itemize}
Recall that in Section \ref{sect:inv-measures}, we have fixed an open neighborhood $\cU$ of $X$ and positive constants $C,\eta, \de, T$ given by Lemma \ref{lem:BasicAwayTang}, such that any periodic orbit $\ga$ of any vector field $Z\in \cU$ with $\tau(\ga)\geq T$ admits a $T$-dominated splitting $\cN_{\ga}=G^s\oplus G^c\oplus G^u$ with respect to $\psi^Z_t$, and satisfies other properties listed there. Also, we have assumed $T=1$ and $\cU$ is small enough such that the properties \ref{P:lorenz-like}$\sim$\ref{P:ph-nbhd} hold for every $Z\in\cU$.

Consider a sequence of $C^2$ weak Kupka-Smale vector fields $\{X_n\}$ that converges to $X$ in the $C^1$ topology. We assume that $X_n\in\cU$ and $C(\si_{X_n},X_n)$ is a quasi-attractor (Lemma \ref{lem:gen_0}, item \ref{item:gen_0-last}) for all $n$.
To simplify notations, let us denote
\[f=\phi^X_1,\ f_n=\phi^{X_n}_1, \ \text{and}\ C(\si_{X_n})=C(\si_{X_n},X_n).\]
Then each $f_n$ is a $C^2$ diffeomorphism on $M$.
\subsubsection{Weak* limits of physical-like measures are Gibbs $F$-states}\label{sect:phl-gibbs}
By Corollary \ref{cor:physical-like}, each $C(\si_{X_n})$ supports a physical-like measure $\mu_n$ of $f_n$. Without loss of generality, let us assume that $\{\mu_n\}$ converges to a probability measure $\mu$ in the weak* topology\footnote{We remark that the measures $\mu_n$ can not be assumed to be ergodic by considering an ergodic component. This is because an ergodic component of a physical-like measure may not be physical-like anymore. See Table \ref{tab:phl-gibbs}.}.
Then $\mu$ is $f$-invariant and
\[\supp(\mu)\subset\limsup_n\supp(\mu_n)\subset\limsup_n C(\si_{X_n})=C(\si).\]
By Theorem \ref{thm:semi-conti}, the measure $\mu$ is a Gibbs $F$-state of $f$.
Recall that physical-like measures are Gibbs $F$-states (Theorem \ref{thm:physical-like}).
By Lemma \ref{lem:gibbs-1}, the measures $\mu_n$ and $\mu$ all satisfy Pesin's entropy formula.
By Lemma \ref{lem:gibbs-0}, we have $\mu(\Sing(X))=0$ and $\mu_n(\Sing(X_n))=0$ for all $n$.
Then by Lemma \ref{lem:invariant-measure}, there exists a sequence of invariant sets $\La_n\subset\supp(\mu_n)$ such that $\mu_n(\La_n)\to 1$ and the upper-bound of $\la^c(x,\phi^{X_n}_1)$ on $\Lambda_n$ goes to zero as $n\to\infty$.

Moreover, since $\mu_n$ converges to $\mu$ in the weak* topology and $\mu(\Sing(X))=0$, one can obtain the following
\begin{Claim}\label{clm:nonsingular}
	For any $\vep>0$, there exists a neighborhood $U_0$ of $\Sing(X)$ such that $\mu_n(\Cl(U_0))<\vep$ for every $n$ large enough.
\end{Claim}

Let $\tilde{\mu}_n=\int_{[0,1]}(\phi^{X_n}_t)_*\mu_ndt$, then it is easy to see that $\tilde{\mu}_n$ is an invariant measure for the flow $\phi^{X_n}_t$. In particular, $\tilde{\mu}_n$ remains an invariant measure for $f_n$, with the same Lyapunov exponents as $\mu_n$ and also satisfies Pesin's entropy formula. Since $\mu_n$ converges to $\mu$ in the weak* topology and $\phi^{X_n}_t\to \phi^X_t$ uniformly for $t\in[0,1]$, it follows that $\{\tilde{\mu}_n\}$ converges to $\tilde{\mu}=\int_{[0,1]}(\phi^X_t)_*\mu dt$ in the weak* topology.
As $\mu(\Sing(X))=0$, one obtains the following

\begin{Claim}\label{clm:nonsingular-1}
	$\tilde{\mu}(\Sing(X))=0$.
\end{Claim}

By replacing $\mu_n$ with $\tilde{\mu}_n$, we assume that $\mu_n$ itself is a nontrivial invariant measure for $\phi^{X_n}_t$. Accordingly, we assume $\La_n$ is a $\phi^{X_n}_t$-invariant set by replacing it with $\cup_{t\in[0,1]}\phi^{X_n}_t(\La_n)$.

\subsubsection{Hyperbolic points}\label{sect:hyperbolic-pts}

Let $\Lambda_n$ and $\mu_n$ be as before and denote $\mu'_n = \mu_n|_{\Lambda_n}$. Since $\mu_n$ converges to $\mu$ in the weak* topology and $\mu_n(\La_n)\to1$, we have
\[\mu'_n\to\mu\ \text{in the weak* topology, as}\ n\to\infty.\]


Without loss of generality, we assume that $\la^c(x,f_n)\in [0,\eta)$ for all $n$ and for every $x\in\La_n$ (Lemma \ref{lem:ergodic-measure} and Lemma \ref{lem:invariant-measure}). By ergodic decomposition and Corollary \ref{cor:le-away-tang}, there exists a one-dominated splitting $\cN_{\supp(\mu'_n)}=G^s_n\oplus G^c_n\oplus G^u_n$ with $\dim G^s_n=\ind(\si)-1$, $\dim G^c=1$ and $\la^c(x,f_n)$ is the Lyapunov exponent corresponding to $G^c_n$.

Let $\eta'=\min\{\log(3/2), \eta/2\}$ and define $H_n=H(\La_n;\eta',1)$ as the set of $(\eta',G^u_n)$-$\psi^{X_n,*}_t$-expanding points\footnote{Here $(\eta', G^u_n)$-expanding means $(1, \eta',1, G^u_n)$-expanding.} of $\mu'_n$, {\it i.e.} $x\in H_n$ if and only if $x\in\supp(\mu'_n)\setminus\Sing(X_n)$ and $x$ is $(\eta',G^u_n)$-$\psi^{X_n,*}_t$-expanding.
By Lemma \ref{lem:pliss-for-measure} and ergodic decomposition, there exists $\theta>0$ independent of $n$, such that $\mu’_n(H_n)\geq\theta$.
%
%

Note that there is $\be_0>0$ (independent of $n$) such that for any $x\in M\setminus\Sing(X_n)$, the exponential map $\exp_x$ is a diffeomorphism from $\cN^{X_n}_x(\be_0)$ onto its image, and we can define $N^{X_n}_x(\be_0)=\exp_x(\cN^{X_n}_x(\be_0))$.

From the dominated splitting $\cN_{\supp(\mu'_n)}=(G^s_n\oplus G^c_n)\oplus G^u_n$, one can obtain a locally invariant plaque family $\mathcal{G}^{u}_n$ tangent to $G^u_n$ (Remark \ref{rmk:plaque-family}). Note that $\mathcal{G}^{u}_n(x_n)\subset N^{X_n}_{x_n}(\be_0)$ for each $x_n\in\supp(\mu'_n)\setminus\Sing(X_n)$. Moreover, by Lemma \ref{lem:stable-manifolds} there is $\de_0>0$ such that for each $x_n\in H_n$, it satisfies $\mathcal{G}^{u}_{n,\de_0\|X_n(x_n)\|}(x_n)\subset W^u_n(\orb(x_n))$. Note that this $\de_0$ is uniform in $n$ (Remark \ref{rmk:stable-manifolds}).
%
%

Since $\mu'_n$ satisfies Pesin's entropy formula, by Theorem \ref{thm:LeY}, there is a set $R_n\subset\supp(\mu'_n)$ with $\mu'_n(R_n)=1$, such that for every $x\in R_n$, $W^u_n(x)\subset\supp(\mu'_n)$. By $\phi_t^{X_n}$-invariance of $\mu'_n$, we have $W^u_n(\orb(x))\subset\supp(\mu'_n)$.
Let $H_n^*=H_n\cap R_n$, one has
\begin{equation}\label{eq:mu-prime-1}
	\mu'_n(H_n^*)=\mu'_n(H_n)\geq\th.
\end{equation}
Note that for every $x_n\in H^*_n$, one has
\[\mathcal{G}^{u}_{n,\de_0\|X_n(x_n)\|}(x_n)\subset W^u_n(\orb(x_n))\subset \supp(\mu'_n).\]

By choosing a subsequence, assume $\supp(\mu'_n)$ converges to $\La$ in the Hausdorff topology. Then one has $\supp(\mu)\subset\La\subset C(\si)$. The dominated splitting $\cN_{\supp(\mu'_n)}=G^s_n\oplus G^c_n\oplus G^u_n$ induces a dominated splitting $\cN_{\La}=G^s\oplus G^c\oplus G^u$. Note that $\dim G^s=\ind(\si)-1$, $\dim G^c=1$, and $G^s$ is uniformly contracting.


Let $H(\eta')$ be the set of $(\eta',G^u)$-$\psi^{X,*}_t$-expanding points of $\La$.
Define
\[H^*=\cap_{m\geq 0}\Cl(\cup_{n\geq m} H_n^*).\]
Then since any limit point of a sequence of $(\eta',G^u_n)$-$\psi^{X_n,*}_t$-expanding points is either a singularity or is $(\eta',G^u)$-$\psi^{X,*}_t$-expanding, one has $H^*\subset H(\eta')\cup\Sing(X)$.
Since $\mu(\Sing(X))=0$ and $\mu'_n\to \mu$ in the weak* topology, there exists a neighborhood $U_0$ of $\Sing(X)$ such that $\mu'_n(U_0)<\th/2$ for all $n$ large enough (see Claim \ref{clm:nonsingular}). By~\eqref{eq:mu-prime-1},
\begin{align}
	\mu(H^*\setminus U_0) &=\lim_m\mu(\Cl(\cup_{k\geq m}H^*_k)\setminus U_0) \notag\\
		&\geq \lim_m\limsup_n\mu'_n(\Cl(\cup_{k\geq m}H^*_k)\setminus U_0) \notag\\
		&\geq \limsup_n\mu'_n(H^*_n\setminus U_0) \notag\\
		&\geq \limsup_n(\mu'_n(H^*_n)-\mu'_n(U_0)) \notag\\
		&\geq \th/2. \label{eq:H-U0}
\end{align}
Since $\mathcal{G}^{u}_{n,\de_0\|X_n(x_n)\|}(x_n)\subset  \supp(\mu'_n)$ and converges uniformly outside $U_0$ (by \cite{HPS}) and $\supp(\mu'_n)$ converges to $\La$ in the Hausdorff topology, one obtains
\begin{equation}\label{eq:H-star}
	\mathcal{G}^{u}_{\de_0\|X(x)\|}(x)\subset \La \ \text{for any}\ x\in H^*\setminus U_0.
\end{equation}

Define
\begin{equation}\label{eq:H-sstar-def}
H^{**} = \{x\in H(\eta'):\exists t_n\to+\infty \mbox{ s.t. }\phi^X_{-t_n}(x)\in H^*\setminus U_0\}
\end{equation}
By inequality \eqref{eq:H-U0} and Poincar\'e's Recurrence Theorem, the set $H^{**}$ is nonempty and $\mu(H^{**})\geq\mu(H^*\setminus U_0)\geq\th/2$.

\begin{Claim}\label{clm:hyperbolic-points-La}
For any $x\in H^{**}$, one has $\mathcal{G}^{u}_{\de_0\|X(x)\|}(x)\subset\La$.
\end{Claim}
\begin{proof}
By definition, there is $t\geq 0$ such that $x'=\phi^X_{-t}(x)\in H^*\setminus U_0$. Then the conclusion follows from \eqref{eq:H-star} and the property of unstable manifolds.
\end{proof}

\subsection{Finishing the proof}
\subsubsection{The set $H^{**}$ accumulates on singularities}

In this section we are going to prove the following claim.

\begin{Claim}\label{clm:hyperbolic-points}
The set $H^{**}$ accumulates on singularities, {\itshape i.e.,} $\Cl(H^{**})\cap\Sing(X)\neq\emptyset$.
\end{Claim}

For the proof of Claim \ref{clm:hyperbolic-points}, we will show that the set $H^{**}$ does not have ``large gaps'' starting outside an arbitrarily small neighborhood of the singularities, {\it i.e.} for any neighborhood $U$ of $\Sing(X)$, there exists $\tau>0$ such that for any $x\in H^{**}$, the backward orbit segment $\phi^X_{[-\tau,0]}(x)$ contains a point $x'\in H^{**}\cap U$.

Let us begin by considering orbit segments going through an arbitrarily small neighborhood of a singularity. We define for each $\si_i\in\La\cap\Sing(X)$ and for each $r>0$ the set $U_r(\si_i)$ as the connected component of the set $\{x\in M: |X(x)|< r\}$ that contains $\si_i$. Since $X$ is Kupka-Smale, there exists $r_0>0$ such that for every $0<r\leq r_0$, the sets $\Cl(U_r(\si_i))$ are pairwise disjoint. Note that each $U_r(\si_i)$ is a neighborhood of $\si_i$ and it shrinks to $\si_i$ as $r$ goes to zero.

Recall that there is a dominated splitting $\cN_{\Lambda}=G^s\oplus G^c\oplus G^u$ with $\dim G^c=1$.
\begin{Lemma}\label{lem:H-sstar}
	For every $\si_i\in\La\cap\Sing(X)$, there exists $\eta_i>0$ with the following property: for every $0<r\leq r_0$ small enough, there exists $V^i_r\subset U_r(\si_i)$ a neighborhood of $\si_i$ such that for any $x\in \phi^X_{-1}U_r(\si_i)\setminus U_r(\si_i)$, if $\phi^X_k(x)\in U_r(\si_i)$ for $k=1,\ldots,l$, $\phi^X_{l+1}(x)\not\in U_r(\si_i)$ and $\{\phi^X_k(x):k=1,\ldots,l\}\cap V^i_r\neq\emptyset$, then it holds
	\[\frac{1}{l}\sum_{k=1}^{l}\log\|\psi^X_1|_{G^c(\phi^X_k(x))}\|\geq \eta_i.\]
\end{Lemma}
\begin{proof}
	Let us consider the 2-dimensional bundle $F^c=G^c\oplus \langle X\rangle$ over $\La\setminus\Sing(X)$. The bundle $F^c$ is $\Phi^X_t$-invariant and one has
	\begin{equation}\label{eq:Fc-bundle}
		\log|\det(\Phi^X_1|_{F^c(x)})|=\log\|\psi^X_1|_{G^c(x)}\|+\log|	X(\phi^X_1(x))|-\log|X(x)|.
	\end{equation}
	Recall that all singularities in $\La$ are Lorenz-like and have the same index $\ind(\si)=\dim G^s+1$. Let $E^{ws}(\si_i)$ be the one-dimensional weak stable direction and $F(\si_i)=E^{ws}(\si_i)\oplus E^u(\si_i)$, then $\Phi^X_t|_{F(\si_i)}$ is sectionally expanding. Without loss of generality, let us assume that \[\log|\det(\Phi^X_1|_{S})|\geq \al>0\]
for any 2-dimensional subspace $S\subset F(\si_i)$. We can define the $F$-cones in a small neighborhood of $\si_i$. Then since $\si$ is Lorenz-like, $F^c(x)$ is contained in an arbitrarily small $F$-cone when $x$ is arbitrarily close to $\si_i$. By continuity, we can fix a small neighborhood $U$ of $\si_i$ such that $\log|\det(\Phi^X_1|_{F^c(x)})|\geq \al/2$ for any $x\in U$.
	
	Let $0<r\leq r_0$ be small enough such that $\phi^X_{-1}U_r(\si_i)\cup U_r(\si_i)\subset U$.
	Then for any $x\in \phi^X_{-1}U_r(\si_i)\setminus U_r(\si_i)$ with $\phi^X_k(x)\in U_r(\si_i)$ for $k=1,\ldots,l$ and $\phi^X_{l+1}(x)\not\in U_r(\si_i)$,
	\[\sum_{k=1}^{l}\log|\det(\Phi^X_1|_{F^c(\phi^X_k(x))})|\geq l\al/2,\]
	and hence by \eqref{eq:Fc-bundle}, one obtains
	\begin{equation}\label{eq:Gc-bundle}
		\sum_{k=1}^{l}\log\|\psi^X_1|_{G^c(\phi^X_k(x))}\|\geq l\al/2+\log|X(\phi^X_1(x))|-\log|X(\phi^X_{l+1}(x))|.
	\end{equation}
	Since the singularity $\si_i$ is hyperbolic, we can assume that $r$ is small enough such that $|X(y)|\geq r$ for any $y\in(\phi^X_{-1}U_r(\si_i)\cup \phi^X_1U_r(\si_i))\setminus U_r(\si_i)$. In particular, $|X(x)|\geq r$ and $|X(\phi^X_{l+1}(x)|\geq r$. Let $c=\min\{m(\Phi^X_1),m(\Phi^X_{-1})\}\in(0,1)$, where $m(\cdot)$ stands for the minimal norm. Then one has $c r\leq c|X(x)|\leq|X(\phi^X_1(x))|<r$ and similarly, $c r\leq c|X(\phi^X_{l+1}(x)|\leq|X(\phi^X_l(x))|<r$.  It follows that $c<|X(\phi^X_1(x))|/|X(\phi^X_{l+1}(x))|<1/c$. Let $V^i_r\subset U_r(\si_i)$ be a neighborhood of $\si_i$ so small that for the point $x$ as above, the condition $\{\phi^X_k(x):k=1,\ldots,l\}\cap V^i_r\neq\emptyset$ implies $l>\frac{4}{\al}\log\frac{1}{c}$. Then it follows from \eqref{eq:Gc-bundle} that
	\[\frac{1}{l}\sum_{k=1}^{l}\log\|\psi^X_1|_{G^c(\phi^X_k(x))}\|\geq\frac{\al}{4}.\]
	This concludes the proof.
\end{proof}

Now let us prove Claim \ref{clm:hyperbolic-points}.

\begin{proof}[Proof of Claim \ref{clm:hyperbolic-points}]
Let us suppose that $H^{**}$ does not accumulate on singularities, {\it i.e.}, there is an open neighborhood $U$ of $\Sing(X)$ such that $H^{**}\cap U=\emptyset$. Since every chain-transitive subset of $C(\si)$ contains a singularity (Lemma \ref{lem:non-singular-set}), there exists a sequence of points $x_l\in H^{**}$ such that the orbit segment $\phi^X_{[0,l+1]}(x_l)=\{\phi^X_t(x_l): t\in[0,l+1]\}$ has no intersection with $H^{**}$ other than $x_l$ itself.
By definition of the set $H^{**}$, see \eqref{eq:H-sstar-def}, we have $x_l\in H(\eta')$ but $\phi^X_j(x_l)\notin H(\eta')$ for each $j=1,\ldots, l+1$.
It follows that
\[\prod_{k=1}^j\|\psi^{X,*}_{-1}|_{G^u(\phi^X_k(x_l))}\|>\e^{-j\eta'}, \quad \forall j=1,\ldots, l+1.\]
As $G^c$ is dominated by $G^u$ with domination constant $T=1$ and $\eta'=\min\{\log(3/2), \eta/2\}$ (see Section \ref{sect:hyperbolic-pts}), one obtains
\[\prod_{k=0}^{j-1}\|\psi^{X,*}_1|_{G^c(\phi^{X}_{k}(x_l))}\|\leq \e^{-j\eta''},\quad \forall j=1,\ldots,l+1.\]
where $\eta''=\log (4/3)$ (see similar estimation in \eqref{eqn:intersection-2} and also \cite[page 990]{PuS}).
By choosing a subsequence, we can assume that $x_l\to x^*$. Then the inequality above implies that for all $l\geq 1$,
\begin{equation}\label{eqn:central-le-3}
\sum_{k=0}^{l-1}\log\|\psi^{X,*}_1|_{G^c(\phi^X_{k}(x^*))}\|\leq -l\eta''.
\end{equation}
Since singularities in $\La$ are Lorenz-like and have index equal to $\ind(\si)=\dim G^s+1$, the inequality \eqref{eqn:central-le-3} implies that $x^*$ is not contained in the stable manifold of any singularity. Let us assume the neighborhood $U$ is small and take a sequence $\{\frac{1}{l_j}\sum_{k=0}^{l_j-1}\de_{\phi^X_{k}(x^*)}\}_{j\in\NN}$ with $\phi^X_{l_j }(x^*)\notin U$. Without loss of generality, assume the sequence converges to a $\phi^X_1$-invariant measure $\nu$. Note that $x^*\not\in U$. Then following from \eqref{eqn:central-le-3}, there exists $C_1\in\RR$ such that
\begin{equation}\label{eqn:central-le-4}
	\sum_{k=0}^{l_j-1}\log\|\psi^X_1|_{G^c(\phi^X_{k}(x^*))}\|\leq -l_j\eta''+C_1,\quad \forall j\in\NN.
\end{equation}

For $r>0$ arbitrarily small and for each $\si_i$, one can fix a neighborhood $V^i_r\subset U_r(\si_i)$ of $\si_i$, given by Lemma \ref{lem:H-sstar}. Define $W^i_r$ as the union of connected orbit segments in $U_r(\si_i)$ that have nonempty intersection with $V^i_r$. Then $W^i_r$ is also a neighborhood of $\si_i$ and it shrinks to $\si_i$ as $r$ goes to zero. We can consider $r>0$ arbitrarily small such that $\nu(\partial W^i_r)=0$ for all $i$.
Let us assume that $\nu=a\nu_0+\sum_{i}b_i\de_i$, where $0\leq a,b_i\leq 1$, $a+\sum_i b_i=1$, $\nu_0(\Sing(X))=0$, and $\de_i$ is the Dirac measure supported on $\si_i\in \La\cap\Sing(X)$.
For each $j\in \NN$ and $\si_i\in\La\cap\Sing(X)$, we define
\[T^i_j=\{k:\phi^X_k(x^*)\in W^i_r, k=0,1,\ldots,l_j-1\},\]
and
\[S_j=\{k:\phi^X_k(x^*)\in \La\setminus \cup_i W^i_r, k=0,1,\ldots,l_j-1\}.\]
Let $t^i_j$ and $s_j$ be the number of elements in $T^i_j$ and $S_j$ respectively, for all $i$ and $j$. By \eqref{eqn:central-le-4} and Lemma \ref{lem:H-sstar},
\begin{align*}
	-\eta''+\frac{C_1}{l_j} &\geq\frac{1}{l_j}\sum_{k=0}^{l_j-1}\log\|\psi^X_1|_{G^c(\phi^X_{k}(x^*))}\|\\
	&= \frac{1}{l_j}\sum_{k\in S_j}\log \|\psi^X_1|_{G^c(\phi^X_{k}(x^*))}\|+\frac{1}{l_j}\sum_i\sum_{k\in T^i_j}\log \|\psi^X_1|_{G^c(\phi^X_{k}(x^*))}\|\\
	&\geq \frac{s_j}{l_j}\cdot\frac{1}{s_j}\sum_{k\in S_j}\log \|\psi^X_1|_{G^c(\phi^X_{k}(x^*))}\|+\sum_i \frac{t^i_j}{l_j}\cdot\eta_i
\end{align*}
where each $\eta_i>0$ is given by Lemma \ref{lem:H-sstar} for $\si_i\in\La\cap\Sing(X)$.
Note that $\frac{s_j}{l_j}\to \nu(\La\setminus\cup_i W^i_r)$ and $\frac{t^i_j}{l_j}\to \nu(W^i_r)$, when $j\to\infty$. Note also that $r>0$ can be chosen arbitrarily small. One can take $r>0$ small enough such that $\frac{s_j}{l_j}$ is arbitrarily close to $a$ and $\frac{t^i_j}{l_j}$ arbitrarily close to $b_i$ for each $i$ and for $j$ large enough. Then one obtains from the last inequality that $a>0$ and
\[\frac{1}{s_j}\sum_{k\in S_j}\log\|\psi^X_1|_{G^c(\phi^X_{k}(x^*))}\|\leq -\eta''/2.\]
Note that $S_j$ depends on $r$ and by letting $r\to 0$, we have $\frac{1}{s_j}\sum_{k\in S_j}\de_{\phi^X_k(x^*)}$ converges to $\nu_0$ in the weak* topology. The previous inequality then implies
\[\int\log\|\psi^X_1|_{G^c(x)}\|d\nu_0(x)\leq-\eta''/2.\]
Therefore, the measure $\nu_0$ (and hence $\nu$) has an ergodic component $\nu_1$ with a negative central Lyapunov exponent (and $\nu_1(\Sing(X))=0$), which is a contradiction to Lemma \ref{lem:ergodic-measure}.
\end{proof}

%
%
\subsubsection{The contradiction}\label{sect:contradiction}
By Claim \ref{clm:hyperbolic-points}, the set $H^{**}$ accumulates on singularities. Recall that all singularities in $C(\si)$ are Lorenz-like and have the same index. Without loss of generality, we assume that $\si\in\Cl(H^{**})\cap\Sing(X)$.
Note that $H^{**}\subset \La$. By Remark \ref{rmk:elp-dom}, there is a dominated splitting $E^u(\si)=E^{c}(\si)\oplus E^{uu}(\si)$ with respect to the tangent flow, where $\dim E^{c}(\si)=1$. Note that here $E^c(\si)$ is the one-dimensional {\em weak unstable} direction.

We denote by $E^{sc}(\si)=E^s(\si)\oplus E^c(\si)$, where $E^s(\si)$ is the stable subspace of $T_{\si}M$. Then $T_{\si}M=E^{sc}(\si)\oplus E^{uu}(\si)$ is a dominated splitting. By the plaque family theorem of \cite{HPS} (see also \cite[Section 3]{BW} and \cite[Section 3]{LVY}), there is a locally invariant plaque  $W^{sc}(\si)$ tangent to $E^{sc}(\si)$ at $\si$. For $r>0$ small, define $W^{sc}_r(\si)$ as the intersection of $W^{sc}(\si)$ with the ball $B_r(\si)$ of radius $r$, centered at $\si$. By changing the Riemannian metric around $\si$, we assume that the subspaces $E^s(\si)$, $E^c(\si)$ and $E^{uu}(\si)$ are mutually orthogonal.

Let us fix a small neighborhood $V_0$ of $\si$, on which we can define local cones at $\si$ {\em on the manifold} as follows:
\[D^{sc}_{\al}(\si)=\{x\in V_0:x^{uu}\leq \al x^{sc}\},\quad D^{uu}_{\al}(\si)=\{x\in V_0:x^{sc}\leq \al x^{uu}\}\]
where $\al>0$, $x^{uu}$ is the distance from $x$ to $W^{sc}(\si)$ and $x^{sc}$ the distance from $x$ to $W^{uu}_{loc}(\si)$.

The following results will be used to obtain the final contradiction and to finish the proof of Proposition \ref{prop:main-proposition}.

\begin{Lemma}\label{lem:contradiction}
	Suppose there exists a sequence of points $x_n\in H^{**}\cap D^{sc}_{\al}(\si)$ with $\al>0$ small enough such that $x_n\to\si$ as $n\to\infty$, then one has $\La=C(\si)$.
\end{Lemma}

\begin{Claim}\label{clm:loc-of-hyp-pts}
	For any $\al>0$, there exists a sequence of points $x_n\in H^{**}\cap D^{sc}_{\al}(\si)$ such that $x_n\to\si$ as $n\to\infty$.
\end{Claim}

Lemma \ref{lem:contradiction} and Claim \ref{clm:loc-of-hyp-pts} will be proven in Section \ref{sect:pf-contradiction} and Section \ref{sect:loc-of-hyp-pts}, respectively.
By Lemma \ref{lem:contradiction} and Claim \ref{clm:loc-of-hyp-pts}, one has $\La=C(\si)$ and it follows from Lemma \ref{lem:gen_1} that
\begin{equation}\label{eqn:contradiction}
	B(\La)\cap (E^s(\si)\oplus E^{uu}(\si))\setminus(E^s(\si)\cup E^{uu}(\si))\neq\emptyset.
\end{equation}
Recall that $B(\La)=\{L\in G^1: \exists x_n\in\La, \txtrm[such that] \langle X(x_n)\rangle\to L\}$.
Since there is a dominated splitting $\cN_{\La}= (G^s\oplus G^c)\oplus G^u$ of index $\ind(\si)$, it follows that $B(\La)$ also admits a dominated splitting of index $\ind(\si)$ (see Remark \ref{rmk:elp}). This is a contradiction\footnote{Alternatively, a contradiction can be obtained as follows. Since the central exponent of $\mu'_n$ is small (Section \ref{sect:small-central-le}), it follows from ergodic decomposition and Corollary \ref{cor:le-away-tang} that $\supp(\mu'_n)$ can be approximated by periodic orbits admitting a dominated splitting of index $\ind(\si)$. Then since $\supp(\mu'_n)\to C(\si)$, one can show that there exists a fundamental sequence $(\ga_n,X_n)\to (C(\si),X)$ which admits a dominated splitting of index $\ind(\si)$. This can be achieved by requiring the limit of the fundamental sequence to contain a point $x\in R$, where $R$ is the subset on the local strong unstable manifold of $\si$ given by Lemma \ref{lem:gen_1}. By Remark \ref{rmk:elp-dom}, one has
	\[\De(\fF)\cap(E^{s}(\si)\oplus E^{uu}(\si))\setminus(E^{s}(\si)\cup E^{uu}(\si))=\emptyset.\]
Note that $B(\La)\subset \De(\fF)$, the previous equation contradicts with \eqref{eqn:contradiction}.
} to Lemma \ref{lem:matching}. Thus we have proven Proposition \ref{prop:main-proposition}, owing only the proof of Lemma \ref{lem:contradiction} and Claim \ref{clm:loc-of-hyp-pts}.

\subsubsection{Proof of Lemma \ref{lem:contradiction}}\label{sect:pf-contradiction}
Let us first give some rough ideas for the proof of Lemma \ref{lem:contradiction}. The first step is to show that when $\al$ is small enough and $x_n\in H^{**}\cap D^{sc}_{\al}(\si)$ is close enough to $\si$, the unstable plaque $\cG^u_{n,\de_0\|X_n(x_n)\|}(x_n)$ intersects transversely with $W^{sc}(\si)$ at a point $y_n$. This is because we will have $X(x_n)\in C^{sc}_a(x_n)$ (Lemma \ref{lem:cone-relation}) and $G^u(x_n)\in C^{uu}_a(x_n)$ (Remark \ref{rmk:elp-dom}) with $a>0$ small, where $C^{sc}_a$ and $C^{uu}_a$ are cone fields on the tangent bundle on a small neighborhood of $\si$ (see definition below).
Then, if some $y_n$ is contained in the local stable manifold of $\si$, we show that $\La$ contains an open subset of the local strong unstable manifold of $\si$ (note that $\cG^u_{n,\de_0\|X_n(x_n)\|}(x_n)\subset\La$ for all $x_n$). Otherwise, we consider the forward iterates of $y_n$ and of a disk $D_n\subset\cG^u_{n,\de_0\|X_n(x_n)\|}(x_n)$ centered at $y_n$, and show that before getting away from a small neighborhood of $\si$, the iterates of the plaques $D_n$ contain disks with a uniform size that converge in the Hausdorff topology to a codimension-one disk $D$ of the local unstable manifold $W^u_{loc}(\si)$. Moreover, the disk $D$ is transverse to the flow direction. Hence $\La$ contains the set $\{\phi^X_t(x):x\in D,t\in(-s,s)\}$ for $s>0$ small, which is an open subset of $W^u_{loc}(\si)$. Then it follows from Lemma \ref{lem:gen_1} that $\La=C(\si)$.

Before the proof, we define the cones on the tangent bundle as follows. Assume that $V_0=B_{r_0}(\si)$ is small enough such that the splitting $T_{\si}M=E^{sc}(\si)\oplus E^{uu}(\si)$ can be extended to $V_0$ in a natural way by letting $E^{sc}(x)$ be the parallel translation of $E^{sc}(\si)$ to $x$ (as in local coordinates), and $E^{uu}(x)$ the parallel translation of $E^{uu}(\si)$ to $x$. Then one defines for each $a>0$ an $sc$-cone $C^{sc}_{a}$ on $V_0$ such that for any $x\in V_0$,
\[C^{sc}_{a}(x)=\{u=v+w\in T_xM:v\in E^{sc}(x),w\in E^{uu}(x),\|w\|\leq a\|v\|\},\]
and similarly a $uu$-cone $C^{uu}_{a}$ on $V_0$ such that for any $x\in V_0$,
\[C^{uu}_{a}(x)=\{u=v+w\in T_xM: v\in E^{sc}(x), w\in E^{uu}(x),\|v\|\leq a\|w\|\}.\]
The next lemma follows from the smoothness of $X$, see \cite[Lemma 3.1]{PYY2}.
\begin{Lemma}\label{lem:cone-relation}
There exists $L\geq 1$ and a neighborhood $V_1\subset V_0$ of $\si$, such that for all $\al>0$ small enough, for any $x\in V_1$,
\begin{itemize}
\item if $x\in D^{sc}_{\al}(\si)$, we have $X(x)\in C^{sc}_{L\al}(x)$;
\item if $X(x)\in C^{sc}_{\al}(x)$, then $x\in D^{sc}_{L\al}(\si)$.
\end{itemize}
Moreover, the same holds for $D^{uu}_{\al}(\si)$ and $C^{uu}_{\al}$.
\end{Lemma}

Now let us prove Lemma \ref{lem:contradiction}.

\begin{proof}[Proof of Lemma \ref{lem:contradiction}]
Let us be more precise with the local dynamics near the singularity $\si$. Let $\eta^c$ be the Lyapunov exponent in the direction $E^{c}(\si)$, and $\eta^u$ be the smallest Lyapunov exponent along $E^{uu}(\si)$. We have $0<\eta^c<\eta^u$.
Let  $\eta^c<\tilde{\eta}<\eta^u$, and $\ga=e^{\tilde{\eta}}$.
Let us fix a constant $0<r\leq r_0$ (recall that $V_0=B_{r_0}(\si)$)   small enough, along with other constants given below, such that the following properties are satisfied:
\begin{enumerate}[(a)]
	\item \label{p:a}
		By continuity, there is a cone field $C^{uu}_{b}$ with $b>0$ small such that for any $x\in B_r(\si)$ it satisfies
			\begin{equation}\label{eqn:expansion}
				\|Df(u)\|\geq \ga\|u\|,\quad \forall u\in C^{uu}_{b}(x),
			\end{equation}
 where $f=\phi^X_1$.
	\item \label{p:b}
		By domination, there exists $\kappa\in(0,1)$ such that if $\phi_t(x)\in B_r(\si)$ for all $0\leq t\leq 1$, then $Df(C^{uu}_{b}(x))\subset C^{uu}_{\kappa b}(f(x))$.
	\item \label{p:c}
		Note that we have assume the subspaces $E^{s}(\si)$, $E^c(\si)$ and $E^{uu}(\si)$ to be mutually orthogonal (Section \ref{sect:contradiction}).
		By Lemma \ref{lem:elp-dom} (see also Remark \ref{rmk:elp-dom}) and continuity of the dominated splitting $\cN_{\La}=(G^s\oplus G^c)\oplus G^u$, there exists $a>0$ such that for any $x\in(\La\setminus\Sing(X))\cap B_r(\si)$ and $X(x)\in C^{sc}_{a}(x)$, one has $G^u(x)\subset C^{uu}_{b/2}(x)$.
	\item \label{p:d}
	
		Recall that for any $r'>0$ small, $W_{r'}^{sc}(\si)$ is the connected component containing $\si$ of the intersection $W^{sc}(\si)\cap B_{r'}(\si)$. Then $W^{sc}_{r}(\si)$ is tangent to the cone field $C^{sc}_{a}$ and contains the local stable manifold $W^s_{r}(\si)=W^s_{loc}(\si)\cap B_r(\si)$.
			By reducing $a>0$ (and hence $r$), we can further assume that for any $x\in W^{sc}_{r}(\si)$ and any non-zero $v\in T_xW^{sc}_{r}(\si)$, it holds $\|Df(v)\|<\ga\|v\|$.
	\item \label{p:e}
		There exist $0<\de'\leq\de_0$ (recall that $\de_0$ is a uniform constant for the plaque family $\cG^{u}_{\de_0\|X(x)\|}(x)$) such that for any $x\in H(\eta')\cap B_r(\si)$ (recall that $H^{**}\subset H(\eta')$), if $X(x)\in C^{sc}_{a}(x)$, then one has $G^u(x)\subset C^{uu}_{b/2}(x)$ (by \ref{p:c}) and hence $\mathcal{G}^{u}_{\de'\|X(x)\|}(x)$ is tangent to $C^{uu}_{b}$. Moreover, assuming $b$ is small and by further reducing $a>0$ if necessary, $\mathcal{G}^{u}_{\de'\|X(x)\|}(x)$ has a unique and transverse intersection with $W^{sc}_{r}(\si)$ at $y\in \mathcal{G}^{u}_{(\de'/2)\|X(x)\|}(x)$ (see \cite[Lemma 3.5]{PYY2}).
	\item \label{p:f}
		By local invariance, there exists $\vep>0$ such that $f(W^{sc}_{\vep}(\si))\subset W^{sc}_{r/2}(\si)$ and $f^{-1}(W^{sc}_{\vep}(\si))\subset W^{sc}_{r/2}(\si)$.
\end{enumerate}

As in the statement of the lemma, suppose there exists a sequence of points $x_n\in H^{**}\cap D^{sc}_{\al}(\si)$ with $\al$ arbitrarily small and $x_n\to\si$ as $n\to\infty$, then we can assume that $x_n\in B_r(\si)$ and $X(x_n)\in C^{sc}_a(x_n)$ for all $n$ (see Lemma \ref{lem:cone-relation}).
Hence by \ref{p:e}, $\mathcal{G}^{u}_{\de'\|X(x_n)\|}(x_n)$ is tangent to $C^{uu}_b$ and has a transverse intersection with $W^{sc}_{r}(\si)$ at $y_n\in \mathcal{G}^{u}_{(\de'/2)\|X(x_n)\|}(x_n)$.
Note that by Claim \ref{clm:hyperbolic-points-La}, we have $\mathcal{G}^{u}_{\de'\|X(x_n)\|}(x_n)\subset \La$ for all $n$.

If there exists $n$ such that $y_n$ is contained in the local stable manifold of $\si$, then by the $\la$-lemma and \ref{p:b}, one can show that $\La$ contains an open set of the local strong unstable manifold $W^{uu}_{loc}(\si)$.

Now suppose $y_n$ is not contained in the local stable manifold of $\si$ for all $n$, we are going to show that $\La$ contains an open disk in the local unstable manifold $W^u_{loc}(\si)$.
As $x_n\to \si$, one has $y_n\in W^{sc}_{\vep}(\si)$ for $n$ large enough. Moreover, $y_n\to \si$ as $n\to\infty$.
Let $k_n\geq 1$ be such that $f^i(y_n)\in W^{sc}_{\vep}(\si)$ for all $0\leq i\leq k_n-1$ and $z_n=f^{k_n}(y_n)\in W^{sc}_{r/2}(\si)\setminus W^{sc}_{\vep}(\si)$ (see \ref{p:f}). Then $k_n\to\infty$ as $n\to\infty$.
By choosing a subsequence, let us assume $z_n\to z \in \Cl(W^{sc}_{r/2}(\si)\setminus W^{sc}_{\vep}(\si))$. Then $f^{-i}(z)\in \Cl(W^{sc}_{r/2}(\si))$ for all $i\geq 0$ and $z\neq\si$. This implies that $z\in W^u_{loc}(\si)$.

Let $\vep_n=d(y_n,\si)$.
By \ref{p:d}, for any $x\in W^{sc}_{r}(\si)$ and any non-zero $v\in T_xW^{sc}_{r}(\si)$, one has $\|Df(v)\|<\ga\|v\|$. Therefore, $z_n=f^{k_n}(y_n)\in W^{sc}_{\vep'_n}(\si)$, where $\vep'_n=\ga^{k_n}\vep_n$. This implies that $\ga^{k_n}\vep_n\geq\vep$, hence
\begin{equation}\label{eqn:estimates-k_n}
	\ga^{k_n}\geq \vep/\vep_n.
\end{equation}

Since  $y_n\in \mathcal{G}^{u}_{(\de'/2)\|X(x_n)\|}(x_n)$, there is a disk $D_n\subset \mathcal{G}^{u}_{\de'\|X(x_n)\|}(x_n)$ centered at $y_n$ with diameter $\de'\|X(x_n)\|$ for the induced metric on $\mathcal{G}^{u}_{\de'\|X(x_n)\|}(x_n)$.
Let us consider its iterates $f^i(D_n)$ for $0\leq i\leq k_n$. See Figure \ref{fig:local-dynamics} for illustration.
\begin{figure}[hbtp]
	\centering
	\includegraphics[width=.6\textwidth]{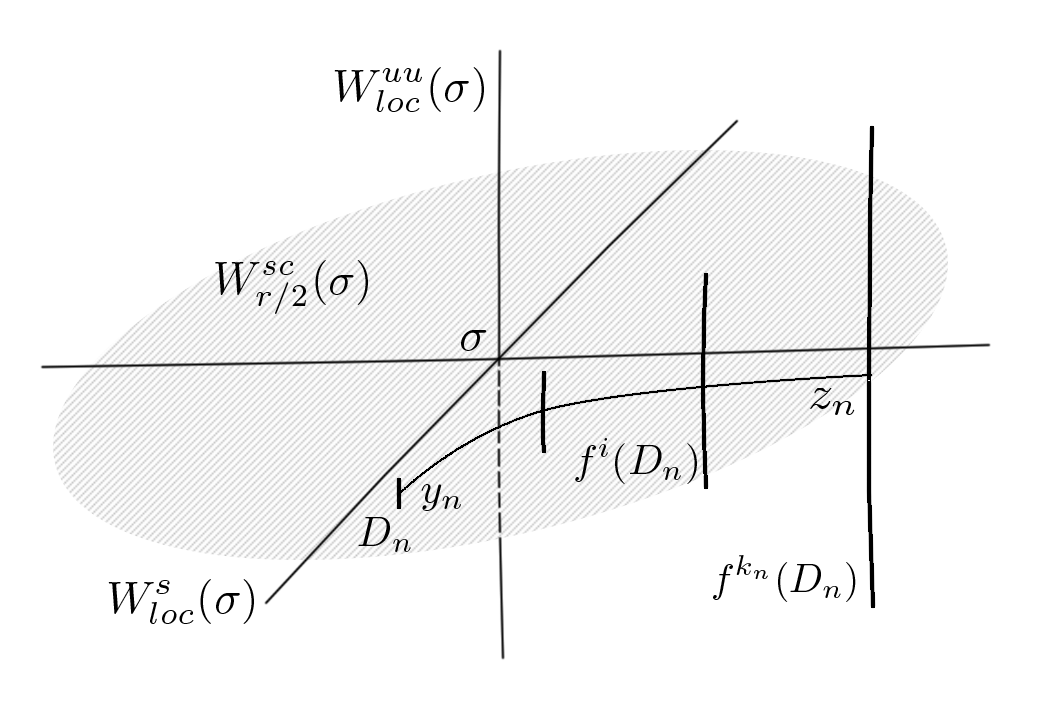}	
	\caption{Local dynamics around the singularity}\label{fig:local-dynamics}
\end{figure}

Since $D_n$ is tangent to $C^{uu}_b$ (see \ref{p:e}), the inequality \eqref{eqn:expansion} and \ref{p:b} imply that $f^{k_n}(D_n)\cap B_r(\si)$ contains a $C^1$ disk $D'_n$ of diameter
\[d_n=\min\{\vep^*, \ga^{k_n}\de'\|X(x_n)\|\},\]
where $\vep^*>0$ is the infimum of interior diameters of all $(\dim M-\ind(\si)-1)$-dimensional $C^1$ disks contained in $B_{r}(\si)$ that is tangent to $C^{uu}_b$, intersecting with $B_{r/2}(\si)$ and having its boundary contained in the boundary of $B_{r}(\si)$. By hyperbolicity of $\si$, there exist $K,K'>0$, such that for all $n$ large enough,
\[\|X(x_n)\|\geq Kd(x_n,\si)\geq K'd(y_n,\si)=K'\vep_n.\]
It follows from \eqref{eqn:estimates-k_n} that
\[\ga^{k_n}\|X(x_n)\|\geq K'\vep.\]
Hence $d_n\geq \min\{\vep^*, K'\de'\vep\}$ for all $n$ large enough. Therefore, we can assume that $D'_n$ converges in the Hausdorff topology to a $C^1$ disk $D^*\subset B_r(\si)$ with diameter $d^*>0$. One can show that $f^{-k}(D^*)\subset B_r(\si)$ and is tangent to $C^{uu}_{b}$ for all $k\in \NN$, which implies that $D^*\subset W^u_{loc}(\si)$ and it contains an open disk $D^{**}$ that is transverse to the flow. Since $D_n\subset \La$ for all $n$, we have $D^{**}\subset \La$, and it follows that $\La$ contains the set $\{\phi^X_t(x):x\in D^{**},t\in(-s,s)\}$ for $s>0$ small, which is an open subset of $W^u_{loc}(\si)$.

Therefore, we have shown that either $\La$ contains an open set of the local strong unstable manifold $W^{uu}_{loc}(\si)$ or it contains an open set of the local unstable manifold $W^u_{loc}(\si)$.
Then it follows from Lemma \ref{lem:gen_1} that $\La=C(\si)$. This completes the proof.
\end{proof}
\subsubsection{Proof of Claim \ref{clm:loc-of-hyp-pts}}
\label{sect:loc-of-hyp-pts}

Recall that $H^{**}\subset H(\eta')$ and the latter is the set of $(\eta',G^u)$-$\psi^{X,*}_t$-expanding points in $\La\setminus\Sing(X)$.
To prove Claim \ref{clm:loc-of-hyp-pts}, we will show that for any $\al>0$, if a point $x\in H(\eta')$ is close enough to $\si$, then there exists a backward iterates $x'=\phi^X_{-t}(x)$ contained in $D^{sc}_{\al}(\si)$ such that $x'$ is $(\eta_0,G^u)$-$\psi^{*}_t$-expanding for some $\eta_0$ slightly smaller than $\eta'$. Moreover, $x'$ is arbitrarily close to $\si$ if $x$ is.

We observe firstly that for a point $x$ arbitrarily close to $\si$ and contained in an arbitrarily small local cone $D^{uu}_{\be}(\si)$, the vector field at $x$ is arbitrarily close to $E^{uu}(\si)$ (see Lemma \ref{lem:cone-relation}) and $G^u(x)$ is arbitrarily close to $E^{sc}(\si)$ (assuming $E^{sc}(\si)$ and $E^{uu}(\si)$ are orthogonal to each other). Since the bundle $E^{sc}(\si)$ is dominated by $E^{uu}(\si)$, one can see that $\psi_{-t}^{X,*}(v)=\frac{\psi^X_{-t}(v)}{\|\Phi^X_{-t}|_{\langle X(x)\rangle}\|}$ is never contracting for a nonzero vector $v\in G^{u}(x)$ and $t>0$ as long as the orbit segment from $x$ to $\phi^{X}_{-t}(x)$ remains in the local cone $D^{uu}_{\be}(\si)$. This implies that $x\not\in H(\eta')$. See \cite[Lemma 3.9]{PYY2} for a similar argument with more details. Thus we obtain the following lemma.

\begin{Lemma}\label{lem:bad-region}
	There exist $r>0$ and $\be>0$ such that $H(\eta')\cap B_r(\si)\cap D^{uu}_{\be}(\si)=\emptyset$.
\end{Lemma}

Moreover, as shown in \cite[Lemma 3.9]{PYY2} (although in a different setting), there exists $\al>0$ such that the local cone $D^{sc}_{\al}(\si)$ is a ``good'' region for $H(\eta')$ in the sense that if an orbit segment $\phi^{X}_{[-T,0]}(x)$ is contained in $D^{sc}_{\al}(\si)$ and is close enough to $\si$, then $\psi^{X,*}_{-t}|_{G^u(x)}$ is contracting for $t\in [0,T]$. We refer to $D^{uu}_{\be}(\si)$ as the ``bad'' region for $H(\eta')$. See Figure \ref{fig:location}.

\begin{figure}[hbtp]
	\centering
	\includegraphics[width=.5\textwidth]{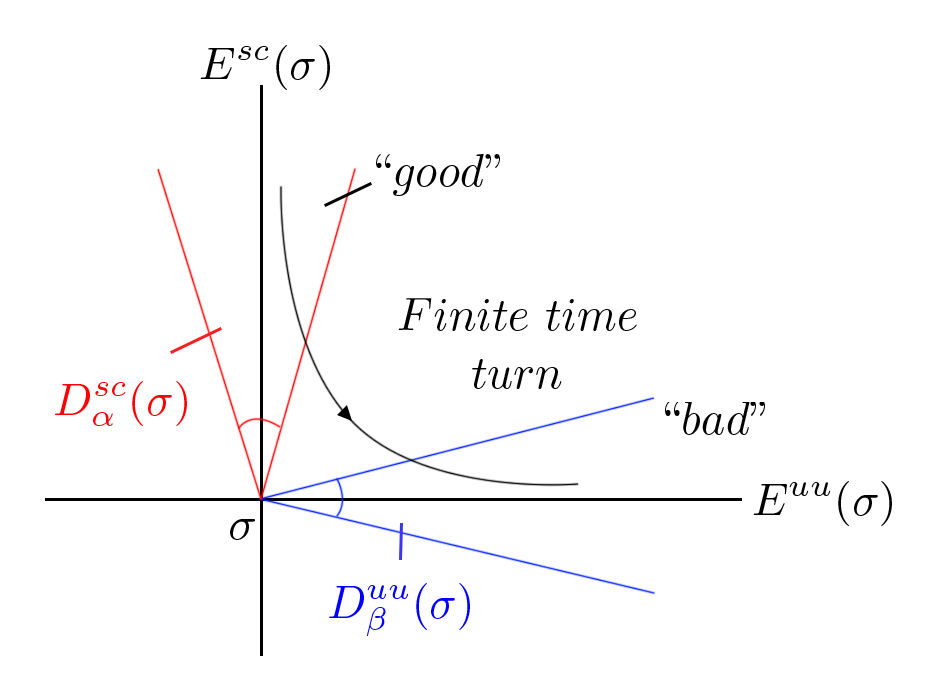}	
	\caption{The ``good'' region, ``bad'' region and finite time turn}\label{fig:location}
\end{figure}

The following lemma shows that there exists a constant $T_{\al,\be}>0$ such that any point $x$ outside the ``bad'' region $D^{uu}_{\be}(\si)$ and close enough to $\si$ goes backward under the flow to the ``good'' region $D^{sc}_{\al}(\si)$ within a time bounded by $T_{\al,\be}$.

\begin{Lemma}\label{lem:finite-time-turn}
	Given $r>0$ small and $\al,\be>0$, there exists $T_{\al,\be}>0$ such that for any $x\in B_r(\si)\setminus D^{uu}_{\be}(\si)$, there exists $t\in [0, T_{\al,\be}]$ with $\phi^X_{-t}(x)\in D^{sc}_{\al}(\si)$.
\end{Lemma}

Lemma \ref{lem:finite-time-turn} is essentially proven in \cite[Lemma 3.2]{PYY2}. We omit its proof here.

Thus, to prove Claim \ref{clm:loc-of-hyp-pts}, the idea goes as follows.  Let $x_n\in H^{**}$ be a sequence of points that converges to $\si$ as $n\to\infty$. By Lemma \ref{lem:bad-region}, for every $n$ large the point $x_n$ is not contained in the ``bad'' region; then by Lemma \ref{lem:finite-time-turn} there is a uniform upper bound $T_0$ such that every $x_n$ goes backward into the ``good'' region within a time less than $T_0$. This allows us to obtain for every $n$ large a point $x'_n$ on the backward orbit of $x_n$ and in the ``good'' region such that it is $(\eta_0, G^u)$-$\psi^*_t$-expanding for some $\eta_0$ slightly smaller than $\eta'$. This is a standard argument, see \cite[Section 3]{PYY2}. More precisely, we have the following lemma.

\begin{Lemma}\label{lem:enlarge-hyperbolic-pts}
Suppose there is a sequence of points $x_n\in H(\eta')$ that converges to $\si$. Then for any $\eta_0\in (0,\eta')$, any $\al>0$, there is a sequence $x'_n$ such that
	\begin{itemize}
		\item each $x'_n$ is on the backward orbit of $x_n$,
		\item the sequence $x'_n$ converges to $\si$ as $n\to\infty$,
		\item for every $n$ large, $x'_n$ is $(\eta_0,G^u)$-$\psi^*_t$-expanding and $x'_n\in D^{sc}_{\al}(\si)$.
	\end{itemize}
\end{Lemma}
\begin{proof}

Let $r>0$ and $\be>0$ be given by Lemma \ref{lem:bad-region}, then $x_n\notin B_r(\si)\cap D^{uu}_{\be}(\si)$ for all $n$.
By changing the Riemannian metric around $\si$, we assume that $E^{sc}(\si)\perp E^{uu}(\si)$.
Let us consider any subsequence $x_{n_k}$ such that $\langle X(x_{n_k})\rangle\to\Ga\in T_{\si}M$. Then $\Ga$ is $(\eta',G^u)$-$\tilde\psi^*_t$-expanding.
By Lemma \ref{lem:cone-relation}, $\Ga$ is not contained in the interior of $C^{uu}_{\be/L}(\si)$. If we consider the $\al$-limit of $\Ga$ under the map $\hat{f}=\hat{\Phi}^X_1$, which we denote by $\al(\Ga,X)$, then either $\al(\Ga,X)=E^{c}(\si)$ or $\al(\Ga,X)\subset E^s(\si)$. Note that $\si$ is Lorenz-like, hence in the latter case, $\al(\Ga,X)$ is in fact the one-dimensional weak stable direction. In both cases, we have $G^u(\al(\Ga,X))=E^{uu}(\si)$ by Remark \ref{rmk:elp-dom}. As $\Ga$ is $(\eta',G^u)$-$\tilde{\psi}^*_t$-expanding, so is $\al(\Ga,X)$. It follows that given $\eta_0\in(0,\eta')$, there exist $\al_0>0$, $N_0>0$, $T_0=T_{\al_0,\be}>1$ (we may assume that $T_0$ is an integer) such that for any $n>N_0$, let $y_n=\phi_{-T_0}(x_n)$, then
\begin{itemize}
\item $\{\phi_{-s}(x_n): 0\leq s\leq T_0\}\subset B_r(\si)$ and $X(y_n)\in C^{sc}_{\al_0}(y_n)$;
\item for any $t\geq 0$, if $\{\phi_{-s}(y_n):0\leq s\leq t+1\}\subset B_r(\si)$ then $X(\phi_{-t}(y_n))\in C^{sc}_{\al_0}(\phi_{-t}(y_n))$ and
\begin{equation}\label{eqn:hp-1}
\|\psi^*_{-1}|_{G^u(\phi_{-t}(y_n))}\|\leq \e^{-\eta_0}.
\end{equation}
\end{itemize}
By compactness, there is $\kappa\geq\eta'$ such that for any $n>N_0$, if $\{\phi_{-s}(x_n): 0\leq s\leq T_0+1\}\subset B_r(\si)$, then for any $0\leq t\leq T_0$, $0\leq s\leq 1$,
\begin{gather}
-\log\|\psi^*_{-T_0}|_{G^u(x_n)}\|\leq \kappa T_0,\label{eqn:hp-2}\\
-\log\|\psi^*_{-s}|_{G^u(\phi_{-t}(x_n))}\|\leq \kappa . \label{eqn:hp-3}
\end{gather}
Since $x_n\to\si$, there is a sequence of neighborhoods $B_r(\si)\supset V_1\supset V_2\supset\cdots\supset V_n\supset\cdots\supset\{\si\}$, $\cap_{n\geq 1}V_n=\{\si\}$, and correspondingly a sequence of positive integers $l_n\to\infty$, such that $\{\phi_{-s}(x_n):0\leq s\leq l_n+T_0\}\subset V_n$, for all $n>N_0$. Denote $\xi=(\kappa-\eta')/(\eta'-\eta_0)$. Take $N_1\geq N_0$, such that for $n>N_1$, we have
\[ k_n=\lfloor(l_n-\xi T_0)/(\xi+1)\rfloor>0.\]
Let $x'_n=\phi_{-k_n-T_0}(x_n)=\phi_{-k_n}(y_n)$, then by \eqref{eqn:hp-2} and \eqref{eqn:hp-3}, for any $l\geq l_n-k_n$,
\begin{align*}
\sum_{i=0}^{l-1} & \log\|\psi^*_{-1}|_{G^u(\phi_{-i}(x'_n))}\|\\
&=\left(\log\|\psi^*_{-T_0}|_{G^u(x_n)}\|+\sum_{i=0}^{l+k_n-1}\log\|\psi^*_{-1}|_{G^u(\phi_{-i}(y_n))}\|\right)\\
&\ \ \ -\left(\log\|\psi^*_{-T_0}|_{G^u(x_n)}\|+\sum_{i=0}^{k_n-1}\log\|\psi^*_{-1}|_{G^u(\phi_{-i}(y_n))}\|\right)\\
&\leq -(l+k_n+T_0)\eta'+(k_n+T_0)\kappa  \\
&\leq -l\eta_0.
\end{align*}
The last inequality and \eqref{eqn:hp-1} implies that $x'_n$ is $(\eta_0,G^u)$-$\psi^*_t$-expanding. By the way we choose $l_n$, $x'_n\in V_n$, hence $x'_n\to \si$. Moreover, since $l_n\to\infty$, $k_n\to\infty$, the contraction of cone $C^{sc}_{\al_0}$ by $\Phi_{-1}$ implies that any limit direction of $\langle X(x'_n)\rangle$ is contained in $E^{sc}(\si)$. In particular, for any $\al>0$, $X(x'_n)\in C^{sc}_{\al/L}(x'_n)$ for all $n$ large enough and by Lemma \ref{lem:cone-relation}, $x'_n\in D^{sc}_{\al}(\si)$.
\end{proof}

Note that $\eta'\in (0, \min\{\log(3/2), \eta/2\})$ was chosen arbitrarily (see Section \ref{sect:hyperbolic-pts}).
Let us write $H^{**}=H^{**}(\eta')$ to emphasize its dependence on the constant $\eta'$.
Then one can show by definition that for any $0<\eta_0< \eta_1<\min\{\log(3/2), \eta/2\}$, it holds $H(\eta_1)\subset H(\eta_0)$ and $H^{**}(\eta_1)\subset H^{**}(\eta_0)$.
Let us fix $\eta_1\in (\eta',\min\{\log(3/2), \eta/2\})$. We have $H^{**}(\eta_1)$ accumulates on a singularity (Claim \ref{clm:hyperbolic-points}), say $\si$. Let $x_n$ be a sequence in $H^{**}(\eta_1)$ such that $x_n\to \si$. Then by Lemma \ref{lem:enlarge-hyperbolic-pts}, for any $\al>0$, there exists a sequence $x'_n$ such that:
\begin{itemize}
	\item $x'_n$ is on the backward orbit of $x_n$,
	\item $x'_n\to\si$ as $n\to\infty$,
	\item $x'_n$ is $(\eta',G^u)$-$\phi^*_t$-expanding and $x'_n\in D^{sc}_{\al}(\si)$ for $n$ large.
\end{itemize}
Since $x_n\in H^{**}(\eta_1)\subset H^{**}(\eta')$, one has by definition that there exists a sequence $t_k\to+\infty$ such that $\phi^X_{-t_k}(x_n)\in H^*(\eta')\setminus U_0$ (see Section \ref{sect:hyperbolic-pts}). This property also holds for $x'_n$ as it is on the backward orbit of $x_n$. Then one concludes that $x'_n\in H^{**}(\eta')$ for $n$ large and Claim \ref{clm:loc-of-hyp-pts} holds for $H^{**}=H^{**}(\eta')$.

By ending the proof of Claim \ref{clm:loc-of-hyp-pts}, we also end the proof of Proposition \ref{prop:main-proposition}.

\vspace{0.5cm}
\noindent Shaobo Gan, School of Mathematical Sciences, Peking University, Beijing 100871, China\\
E-mail address: gansb@pku.edu.cn
\vspace{0.5cm}

\noindent Jiagang Yang, Departamento de Geometria, Instituto de Matem\'{a}tica e
Estat\'{i}stica, Universidade Federal Fluminense, Niter\'{o}i, Brazil\\
E-mail address: yangjg@impa.br
\vspace{0.5cm}

\noindent Rusong Zheng, Joint Research Center on Computational Mathematics and Control, Shenzhen MSU-BIT University, Shenzhen 518172, China\\
E-mail address: zhengrs@smbu.edu.cn

\end{document}